\documentclass[12pt]{amsart}

\usepackage[utf8]{inputenc}
\usepackage{various}
\usepackage{hyperref}
\usepackage{graphicx}
\usepackage{amssymb}
\usepackage{float}
\usepackage{fullpage}
\usepackage{physics}
\usepackage{comment}
\usepackage{tikz}
\usetikzlibrary{positioning, arrows.meta, shapes.geometric}
\usepackage{mathtools}
\usepackage{bm}

\renewtheorem{theorem}{Theorem}[section]
\renewtheorem{lemma}[theorem]{Lemma}
\renewtheorem{corollary}[theorem]{Corollary}
\renewtheorem{proposition}[theorem]{Proposition}
\renewtheorem{conjecture}{Conjecture}

\theoremstyle{definition}
\newtheorem{definition}{Definition}

\theoremstyle{remark}
\renewtheorem{remark}{Remark}

\numberwithin{equation}{section}

\providecommand{\B}{\mathcal{B}}
\providecommand{\A}{\mathcal{A}}

\providecommand{\PP}{\mathcal{P}}
\providecommand{\C}{{\mathbb C}}
\providecommand{\R}{{\mathbb R}}

\providecommand{\Z}{{\mathbb Z}}
\providecommand{\1}{{\mathbf 1}}

\providecommand{\rd}{\mathrm{ d}}
\providecommand{\eps}{\varepsilon}

\providecommand{\precprec}{\prec\mathrel{\mkern-5mu}\prec}

\renewcommand{\Re}{\operatorname{Re}}

\renewcommand{\exp}{\operatorname{exp}}
\renewcommand{\P}{\mathbb{P}}

\renewcommand{\k}{\mathrm{k}}

\renewcommand{\c}{\mathfrak c}
\renewcommand{\d}{\mathrm d}
\newcommand{\m}{\mathfrak{m}}
\newcommand{\w}{\mathrm w}

\renewcommand{\pmod}[1]{\,\,(\mathrm{mod}\,{#1})}
\newcommand{\one}{\mathbf{1}}

\newcommand{\mA}{\mathcal{A}}
\newcommand{\mB}{\mathcal{B}}
\newcommand{\mC}{\mathcal{C}}
\newcommand{\mD}{\mathcal{D}}
\newcommand{\mE}{\mathcal{E}}

\title{Primes and the Möbius function over quadratic polynomials}

\author{Lasse Grimmelt}
\email{lasse.grimmelt@gmail.com}

\author{Jori Merikoski}
\address{Department of Mathematics and Statistics, P.O. Box 68, Pietari Kalmin katu 5, University of Helsinki, FI-00014, Finland}
\email{jori.merikoski@helsinki.fi}

\author{Mayank Pandey}
\address{Courant Institute of Mathematical Sciences, 251 Mercer Street, New York, NY, 10012}
\email{mayankpandey9973@gmail.com}

\author{Alexandru Pascadi}
\address{Mathematisches Institut, Endenicher Allee 60, 53115 Bonn, Germany}
\email{alexpascadi@gmail.com}


\begin{document}
\begin{abstract}
We prove an asymptotic formula for the number of primes $p\leq X$ of the form $n^2+h$ for almost all $h \leq X^{1/2-o(1)}$.  We prove that, for almost all $h \leq X^{2/5}$,  the number of primes  satisfies a correct-order lower bound.  We also show cancellation of the Möbius function over $n^2+h$ for almost all $h < X^{1/2-o(1)}$. 
\end{abstract}

\maketitle

\tableofcontents
\section{Introduction}
A famous open question asks if there are infinitely many prime numbers represented by any non-constant irreducible polynomial $P \in \Z[X]$ with positive leading coefficient and no fixed prime divisor. The case $\deg P=1$ is Dirichlet's theorem on primes in arithmetic progressions, but already for $\deg P=2$ the conjecture appears to be far beyond our current methods. 

Let $\mathfrak{S}_h=\prod_p \frac{1- \rho_h(p)/p}{1-1/p}$ denote the singular series constant for primes of the form $n^2+h$, where $\rho_h(p)$ denotes the number of  solutions $x \pmod{p}$ to the congruence $x^2+h \equiv 0 \pmod{p}$.  Baier and Zhao \cite{baierzhao} have shown that for $X^{1/2 + \eps} < H < X/2 $ and any $C>0$ we have 
\begin{align} \label{eq:baierzhaoresult}
        \sum_{n^2+h \leq X} \Lambda(n^2+h) = \mathfrak{S}_h (X-h)^{1/2} + O_C\bigg(\frac{X^{1/2}}{(\log X)^{C}}\bigg)
\end{align}
for all but $ \ll_C H/(\log X)^C$ values of $h \leq H$. 

The value $X^{1/2}$ is a natural barrier since for $H> 2X^{1/2}$ the intervals $[n^2,n^2+H]$ overlap. The method of Baier and Zhao fails completely past this threshold and does not even give lower bounds for the number of primes. Our main theorem breaks this barrier.

\begin{theorem} \label{thm:primesasymp}
 Let $\nu > 0$ and $H=X^{1/2-\nu}$. Then
    \begin{align*}
      \sum_{n^2+h \leq X} \Lambda(n^2+h) = (1+O(\nu))\mathfrak{S}_h X^{1/2} 
    \end{align*}
    for all but $\ll_C H/(\log X)^C$ values of $h \leq H$.    In particular, for $H=X^{1/2-o(1)}$ we have  $$        \sum_{n^2+h \leq X} \Lambda(n^2+h) = (1+o(1)) \mathfrak{S}_hX^{1/2}$$     for all but $\ll_C H/(\log X)^C$ values of $h \leq H$. 
\end{theorem}

The error term $O(\nu\mathfrak{S}_h X^{1/2})$ can be made more precise. In particular, we prove the following lower bound result.
\begin{theorem} \label{thm:primes}
Let $X^{2/5}\le H\le X/2 $ and $C>0$. Then 
\begin{align*}
        \sum_{n^2+h \leq X} \Lambda(n^2+h) \geq 0.098 \mathfrak{S}_h (X-h)^{1/2}
\end{align*}
for all but $\ll_C H/(\log X)^C$ values of $h \leq H$.
\end{theorem}

As a corollary, for $X^{2/5} \leq H \leq X/2$ and for any subset $\mathcal{H} \subseteq [1,H] \cap \Z$ with $\# \mathcal{H} \geq H/(\log X)^C$, we obtain a correct order lower bound for the number of primes $p=n^2+h \leq X$ with $h \in \mathcal{H}$. Theorem \ref{thm:primes} is then the first time that primes have been detected in a thin polynomial sequence which does not factor over number fields. Previously the density barrier has been overcome only  for factorizing polynomials such as $X^2+Y^4$ by Friedlander and Iwaniec \cite{FI} and $X^3+2Y^3$ by Heath-Brown \cite{hb} (cf. also their descendants \cite{fouvryi,FIalmost,greensawhney,hbli,lds,merikoskisparse,merikoskipolyprimes,pratt,stanleyFI} and \cite{hbmoroz1,hbmoroz2,li,maynard}).

The exponent $2/5$ in Theorem \ref{thm:primes} can be strengthened slightly by optimizing the sieve more carefully \footnote{A quick check with ChatGPT 6 Pro indicates (unverified) that the exponent $2/5$ can be improved to $0.391$, replacing the lower bound $0.098$  by  $0.0001953$.}. Restricting to $H \geq X^{2/5}$ simplifies the sieve argument considerably. We face a hard barrier at $1/3$. The constant 0.098 in Theorem \ref{thm:primes} can be replaced by the following expression
\begin{align*}
  & 1-o(1)- \left( \int_{1/5}^{1/4} \int_{\alpha_2}^{1/2} + \int_{1/4}^{1/3} \int_{\alpha_2}^{1-2\alpha_2} \right) \frac{1}{\alpha_1 \alpha_2 (1-\alpha_1-\alpha_2)} \d \alpha_1 \d \alpha_2
\\
&- 
\int_{1/5}^{1/4} \int_{\alpha_2}^{1-3\alpha_2} \frac{\log\left(\frac{1-\alpha_1-2\alpha_2}{\alpha_2}\right)}{\alpha_1 \alpha_2 (1-\alpha_1-\alpha_2)} \d \alpha_1 \d \alpha_2 -\int_{1/5}^{1/4} \int_{\alpha_1}^{1/2-\alpha_1} \frac{1 + \log(4 - 5\alpha_1 - 5\alpha_2)}{\alpha_1 \alpha_2 (1-\alpha_1-\alpha_2)} \, \d\alpha_2 \, \d\alpha_1.
\end{align*}

For the Möbius function we prove cancellation provided that $H$ is close to $X^{1/2}$.
\begin{theorem} \label{thm:mobius}
    For any $\nu > 0$ we have uniformly for $H=X^{1/2-\nu}$
    \begin{align*}
        \sum_{h \leq H} \bigg| \sum_{n^2+h \leq X} \mu(n^2+h) \bigg| \ll \nu H X^{1/2}.
    \end{align*}
    In particular, for $H=X^{1/2-o(1)}$ we have  $$   \sum_{h \leq H} \bigg| \sum_{n^2+h \leq X} \mu(n^2+h) \bigg|= o(HX^{1/2}).$$
\end{theorem}

It is instructive to compare these  results to the twin prime analogues $\Lambda(n)\Lambda(n+h)$ and $\mu(n)\mu(n+h)$ (i.e. replacing $n^2+h$ by $n(n+h)$). There it is known that the analogue of \eqref{eq:baierzhaoresult} holds for $H\geq X^{8/33 + \eps}$  and that the analogue of Theorem \ref{thm:mobius} holds for $H$ growing arbitrarily slowly by the works of Matomäki, Radziwi\l\l, and Tao \cite{MRTchowla,MRT_primes}. Thus, at least from this perspective, the irreducible case $n^2+h$ appears to be significantly harder.

Finally, we mention the following corollary  on the ternary divisor function $\tau_3(n) = \sum_{n=abc}1$. This can be shown using slight modifications of Propositions \ref{prop:typeI}, \ref{prop:typeIfact}, and \ref{prop:typeI2}. The proof is omitted as the details are standard but laborious. Note that we do not have an asymptotic formula for the sum of $\tau_3$ over any irreducible quadratic polynomial. In fact, even an almost all result breaking the 1/2-barrier is not known for irreducible quadratic polynomials. 
\begin{theorem}
    Let $0<\nu < 1/18$ and $H=X^{1/2-\nu}$. Then for some $\eps= \eps(\nu) > 0$ and for certain polynomials $P_h$ of degree $2$ we have
    \begin{align*}
        \sum_{h \leq H} \bigg| \sum_{n^2+h \leq X} \tau_3(n^2+h) - X^{1/2} P_h (\log X) \bigg| \ll_\nu H X^{1/2-\eps}.
    \end{align*}
\end{theorem}

\subsection{Sketch of the proof}
Consider first the lower bound Theorem \ref{thm:primes}. Let $a_n^{(h)}$ denote the indicator that $n=r^2+h \leq X$. Consider $H=X^{1/2-\nu}= X^{2/5}$ so that $\nu=1/10$. Let $\eta>0$ be very small. Using Harman's prime-detecting sieve \cite{harman} gives a combinatorial lower bound sieve which reduces the proof of Theorem \ref{thm:primes} to studying three types of sums, for almost all $h \leq H$:
\begin{itemize}
    \item Type I sums $\sum_{d \leq X^{1/2-\eta}} \lambda_d \sum_{n } a_{dn}^{(h)}.$ 
    \item Type $\mathrm{I}_2$ sums $\sum_{d \leq X^{\nu+\eta}} \lambda_d \sum_{n} \tau(n) a_{dn}^{(h)}$, where $\tau(n)$ is the divisor function. 
    \item Type II sums $\sum_{m \sim M} \sum_{n \sim N}\alpha_m \beta_n a^{(h)}_{mn}$, where $X^{\nu+\eta}<N< X^{-\eta} H^{1/2}$.
\end{itemize}
Cf. Section \ref{sec:setup} for more precise statements (Propositions \ref{prop:typeI}, \ref{prop:typeI2}, \ref{prop:typeII}) and Section \ref{sec:primesthmproof} for the sieve construction. The estimation of Type II sums is the main novelty of the paper and, as usual, this allows us to break the parity barrier of the sieve and thus detect primes. The estimates for the Type I and $\mathrm{I}_2$ sums are essentially pointwise in $h$. 
 
Type I sums are very easy to handle as the modulus $d$ is shorter than the smooth variable $r$ in $n=r^2+h$. This is done in Section \ref{sec:typeI}.

For Type $\mathrm{I}_2$ we deal with a quadratic divisor problem $\tau(r^2+h)$. This is handled in Section \ref{sec:typeI2} via spectral methods of automorphic forms on $\Gamma_0(d)$, as developed by the first two authors in \cite{GM,GMquadratic}. What makes this tricky is that we require large uniformity in both the discriminant $h \leq H=X^{2/5}$ as well as the level $d \leq X^{1/10+\eta}$ and  results of this type were not available prior to \cite{GMquadratic}.

For the Type II sums our approach is inspired by the work of the second author \cite{mx2y3} which gave a conditional result on primes of the form $x^2+y^3$. An application of Cauchy-Schwarz on $m$ creates a binary quadratic form $n_1 r_2^2-n_2 r_1^2$ whose representation counting function is morally like a Fourier coefficient of a GL(2) modular form of conductor $\asymp N^2$. Another application of Cauchy-Schwarz on the $h$-average creates smooth sums that allow us to use Voronoï summation for the representation counting functions  (cf. Lemma \ref{le:voronoi}).  A detailed sketch is given in Section \ref{sec:typeiisketch} and technical lemmas concerning binary quadratic forms are given in Section \ref{sec:quadratic}.

Our Type II estimate is non-trivial as long as $H \geq X^{1/3+\eps}$. Therefore, for $E_k$-numbers (exactly $k$ prime factors) with $k\geq 2$ as well as $X^\eps$-smooth numbers, we get the analogue of Theorem \ref{thm:primes} in the extended range $H \geq X^{1/3+\eps}$. We also remark it should be possible to obtain a power
saving bound for the exceptional set using the new OpenAI breakthrough on Landau-Siegel zeros. Our key estimates (Propositions \ref{prop:typeI} to \ref{prop:typeII}) are already suitable for a power-saving result since we use a major arc model of length $R$ which is a power of $X$.

With just these three types of sums we cannot hope to get an asymptotic formula even for $H=X^{1/2-o(1)}$. This is because none of our estimates cover the case $p \rightsquigarrow n_1n_2n_3$ of three factors $n_i \in [X^{1/4+\eta},X^{1/2-\eta}]$.  For the proof of Theorem \ref{thm:primesasymp} we also need further 
\begin{itemize}
    \item Factorable Type I sums $\sum_{d_1 \leq D_1} \sum_{d_2 \leq D_2}  \alpha_{d_1} \beta_{d_2} \sum_{n } a_{d_1 d_2n}^{(h)}.$  
\end{itemize}
The precise statement is Proposition \ref{prop:typeIfact}. The argument using quadratic reciprocity is sketched in Section \ref{sec:typeifactorablesketch} and a detailed proof is given in Section \ref{sec:typeIfact}. For $\nu < 1/18$ the range of $D_1,D_2$ covers the critical range $D_1=D_2=X^{1/3}$, and in particular, the ternary divisor problem $\tau_3(r^2+h)$. However, we need to take $\nu > 0$ small since none of our estimates cover the case $p \rightsquigarrow n_1n_2n_3n_4$ four factors of roughly equal size $ n_i \in [X^{1/4- O(\nu)},X^{1/4+O(\eta)}]$. Using a sieve argument we are able to show that these contribute only $O(\nu)$ compared to the main term (cf. Section \ref{sec:primesasymp}).

The proof of Theorem \ref{thm:mobius} relies on anatomy of integers, that is, almost all $n^2+h$ have a small prime factor $p|n^2+h$. This creates a Type II sum with a small $N$, cf. Section \ref{sec:möbius} for a detailed proof.

\subsection*{AI statement}
Large language models (Gemini, ChatGPT, and Claude) have been used as assistants for writing up the manuscript, checking, and typesetting some routine details. The authors take full responsibility for the correctness. All new ideas in the paper were conceived without the use of LLMs. 

\subsection*{Acknowledgments}
LG is supported by the European Research Council (ERC) under the European Union's Horizon research and innovation programme No. 101162746. JM is supported by the Research Council of Finland grant No. 369650. AP is supported by the ERC Advanced Grant No. 101054336 and Germany’s Excellence Strategy grant EXC-2047/1-390685813. 

\section{Set-up and arithmetic information}
\subsection{Notations and parameters}

We fix a global large parameter $X$. We define $f \precprec g$ to mean $|f| \leq X^{o(1)} g$.

We let $\delta> 0$ be a small parameter satisfying $\delta^{-1} \leq (\log X)^{O(1)}$. We let $f$ denote a non-negative, non-zero, smooth real function supported in $[1/2,1]$ and satisfying the derivative bound $ f^{(J)} \ll_J \delta^{-J}$ for all $J \geq 0$.
We write
\[
    I_f:=\int_0^\infty f(u^2)\rd u
       =\frac12\int_0^\infty \frac{f(u)}{\sqrt u}\rd u.
\]  

We let $\eta>0$ be a small parameter, set $\eta_1=\eta/100$, and define
\begin{align} \label{eq:Rdef}
    R=X^{\eta_1}.
\end{align}

\subsection{Sequence $\mathcal{A}$ and  model $\mathcal{B}$} \label{sec:setup}
Let us denote the quadratic Gauss sum by
\begin{align} \label{eq:gausssumdef}
    G(\alpha,r) := \sum_{u \pmod{r}} e_r(\alpha u^2), \quad e_r(x) = e^{2 \pi i x /r}.
\end{align}

\begin{definition}[Squares and the major arc model] \label{def:squaremodel}
We define $\sigma_1, \sigma_0, \sigma^{\flat}$ by
\begin{align*}
      \sigma_1(k) &=  f\left(\frac{k}{X}\right) \mathbf{1}_{k = \square}  \\
     \sigma_0(k) &= f\left(\frac{k}{X}\right) \frac{1}{2\sqrt{k}} \sum_{r \leq R} \frac{1}{r} \sum_{\alpha \pmod r^\times} G(\alpha, r) e_r(-\alpha k) \\ 
     \sigma^{\flat}(k) &= \sigma_1(k)-\sigma_0(k).
\end{align*}
\end{definition}
Here $\sigma_0$ is the major arc model for squares, mimicking the behaviour of $\sigma_1$ both analytically and in arithmetic progressions with small moduli.

Let $\gamma_h$ be a sequence supported on $h \leq H$ satisfying the divisor bound $|\gamma_h| \leq X^{o(1)}$. 
We define the sequences $\mathcal{A}^{(\gamma)}=(a^{(\gamma)}_n)$, $\mathcal{B}^{(\gamma)}=(b^{(\gamma)}_n)$, and their difference $\mathcal{W}^{(\gamma)}=(w^{(\gamma)}_n)$ by 
\begin{align*}
    a_n^{(\gamma)} &= \sum_{k+h=n} \gamma_h \sigma_1(k), \\
    b^{(\gamma)}_n &= \sum_{k+h=n} \gamma_h \sigma_0(k), \\
    w^{(\gamma)}_n &= a^{(\gamma)}_n-b^{(\gamma)}_n = \sum_{k+h=n} \gamma_h \sigma^{\flat}(k).
\end{align*}
We note that by definition these sequences are supported on $n \in [X/2,X+H]$.

\subsection{Arithmetic information}

For the sieve argument we have the following arithmetic information. These are proved in Sections \ref{sec:typeI}, \ref{sec:typeI2}, and \ref{sec:typeII}, respectively. The saving over the trivial estimate $HX^{1/2}$ depends on the length $R$ of the major arc model $\sigma_0(k)$ in \eqref{eq:Rdef}.
\begin{proposition}[Type I information] \label{prop:typeI}
Let $X^{\eta}< H \leq X^{1/2+o(1)}$. We have for $D=X^{1/2 -\eta}$
    \begin{align*}
   \sum_{d \leq D} \mu^2(d) \bigg| \sum_{n}  w^{(\gamma)}_{dn} \bigg|   \precprec H X^{1/2} R^{-1/2}.
    \end{align*}
\end{proposition}
\begin{proposition}[Type $\mathrm{I}$-factorable information] \label{prop:typeIfact}
Let $X^\eta < H = X^{1/2-\nu} \le X^{1/2+o(1)}$ and $D_1, D_2$ satisfy
\[
    D_1 \le X^{1/2-\nu-\eta},
    \qquad 
    D_2 \le X^{3/8 - 3\nu/4 - \eta},
    \qquad
    D_1^{5/2} D_2^3 \le X^{2-\eta}.
\]
Let $\alpha_{d_1}$, $\beta_{d_2}$ be divisor-bounded coefficients supported on square-free integers. Then we have
\[
    \mathop{\sum_{d_1 \sim D_1} \sum_{d_2 \sim D_2}}_{(d_1,d_2)=1} \alpha_{d_1} \beta_{d_2} \sum_n w^{(\gamma)}_{d_1d_2n} \precprec H X^{1/2} R^{-1/2}.
\]
\end{proposition}
\begin{proposition}[Type $\mathrm{I}_2$ information] \label{prop:typeI2}
Let $X^{\eta}< H \leq X^{1/2+o(1)}$ and $D_2=X^{1/4-\eta}$. Let $\psi$ be a smooth function compactly supported in $[1, 2]$. We have for any $M< X^{1-\eta}$
  \begin{align*}
    \sum_{d \leq D_2}  \bigg| \sum_{m,n} \mu(dm)^2  \psi\bigg(\frac{dm}{M}\bigg)   w^{(\gamma)}_{dmn} \bigg| \precprec H X^{1/2}  R^{-1/2}.
  \end{align*}
\end{proposition}
In the following Type II estimate, the range for $N$ is nonempty for $H > X^{1/3+2\eta}$.
\begin{proposition}[Type II information]  \label{prop:typeII}
Let $\nu  \leq 1/6-2\eta$ and $
H=X^{1/2-\nu} \leq X^{1/2+o(1)}$. Let $M,N>1$ satisfy for some $\eta>0$
 \begin{align*}
     X^{\nu+\eta} < N < X^{1/4 -\nu/2-\eta}, \quad MN= X.
 \end{align*}
Let $\alpha_m, \beta_n$ be divisor-bounded coefficients with $\beta_n$ supported on square-free integers. Then we have
 \begin{align*}
     \sum_{m \sim M} \sum_{n} \alpha_m \beta_n w^{(\gamma)}_{mn} \precprec H X^{1/2}  R^{-1/4}.
 \end{align*}
\end{proposition}

\subsection{Sketch of Type I-factorable treatment} \label{sec:typeifactorablesketch}
Let us informally outline the key ideas behind Proposition~\ref{prop:typeIfact}, ignoring smooth weights, coprimality constraints, non-critical ranges, and several other technical details.
Unpacking the definition of $w_{d_1d_2n}^{(\gamma)}$, we wish to estimate a sum of the shape
\[
    \sum_{h \sim H} \gamma_h \sum_{d_1 \sim D_1} \alpha_{d_1} \sum_{d_2 \sim D_2} \beta_{d_2}
    \sum_{n \sim \frac{X}{D_1D_2}}
    \one_{\square}(d_1d_2n-h).
\]
The trivial bound is roughly $X^{1/2} H$, and we aim for an error term with a power-saving over this bound. We write $d_1d_2n - h = r^2$, sum over $n$ to obtain a congruence $r^2 \equiv -h \pmod{d_1d_2}$, and explicitate the possible residue classes of $r \bmod d_1d_2$ to reach
\[
    \sum_{h \sim H} \gamma_h \sum_{d_1 \sim D_1} \alpha_{d_1} \sum_{d_2 \sim D_2} \beta_{d_2} 
    \sum_{\substack{\omega \pmod{d_1d_2} \\ \omega^2 \equiv -h \pmod{d_1d_2}}}
    \sum_{r \sim \sqrt{X}} 
    \one_{r \equiv \omega \pmod{d_1d_2}}.
\]
We then apply Poisson summation, which extracts the main term: the resulting principal frequency  can be compared to the major-arc model. In the error term, focusing on the largest frequencies $t$, we are left to bound
\[
    \sum_{h \sim H} \gamma_h \sum_{d_1 \sim D_1} \alpha_{d_1} \sum_{d_2 \sim D_2} \beta_{d_2}
    \sum_{\substack{\omega \pmod{d_1d_2} \\ \omega^2 \equiv -h \pmod{d_1d_2}}}
    \sum_{t \sim \frac{D_1D_2}{\sqrt{X}}}
    e\left(\frac{\omega t}{d_1d_2}\right)
\]
by a power-saving over $HD_1D_2$. We need to save a bit more than $D_1D_2/\sqrt{X}$ (the length of the $t$-variable) over the trivial bound.

At this point, we apply Cauchy--Schwarz with $h, d_1$ outside, which gives a diagonal condition $D_1 < \sqrt{X}$. We also aim to keep `as much' of the $\omega$-variable outside as possible; this is achieved via a delicate root coupling argument in Section~\ref{subsec:typeIfact-Cauchy-Schwarz}, partly inspired by a trick of the second author \cite{merikoski2023largest}. Focusing on the `most off-diagonal' terms where the two copies of the $d_2$-variable are coprime, and moving the smooth sum over $h$ inside, we reach a sum like
\[
    \sum_{\substack{d_2, d_2' \sim D_2 \\ (d_2, d_2') = 1}} \beta_{d_2} \bar\beta_{d_2'} 
    \sum_{d_1 \sim D_1}
    \sum_{t, t' \sim \frac{D_1D_2}{\sqrt{X}}}
    \sum_{\omega \pmod{d_1d_2d_2'}} 
    e\left(\frac{\omega t}{d_1d_2} - \frac{\omega t'}{d_1d_2'}\right)
    \sum_{\substack{h \sim H \\ h \equiv - \omega^2 \pmod{d_1d_2d_2'}}}
    1.
\]
The target bound is now a power-saving over $HD_1D_2^2$. We apply Poisson summation again in the $h$-variable, which leaves us with bounding
\[
    \sum_{\substack{d_2, d_2' \sim D_2 \\ (d_2, d_2') = 1}} \beta_{d_2} \bar\beta_{d_2'}
    \sum_{d_1 \sim D_1}
    \sum_{t, t' \sim \frac{D_1D_2}{\sqrt{X}}}
    \sum_{\ell \sim \frac{D_1D_2^2}{H}}
    \sum_{\omega \pmod{d_1d_2d_2'}}
    e\left(\frac{-\ell \omega^2 + (td_2'-t'd_2)\omega}{d_1d_2d_2'}\right).
\]
by a power-saving over $D_1^2 D_2^4$ (the frequency $\ell = 0$ leads to a complete linear sum over $\omega$, which essentially vanishes outside of the diagonal terms).

We then recognize the quadratic Gauss sum in $\omega \pmod{d_1d_2d_2'}$, and apply quadratic reciprocity to obtain a similar Gauss sum in $u \pmod{\ell}$. The condition $D_1^{5/2} D_2^3 < X$ is used to crudely remove an extraneous B\'ezout-type exponential phase during this reciprocity step, although a more sophisticated argument could likely avoid this. Accounting for the square-root factors of the normalized Gauss sums, we are left to bound a sum like
\[
    \sum_{\substack{d_2, d_2' \sim D_2 \\ (d_2, d_2') = 1}} \beta_{d_2} \bar\beta_{d_2'}
    \sum_{d_1 \sim D_1}
    \sum_{t, t' \sim \frac{D_1D_2}{\sqrt{X}}}
    \sum_{\ell \sim \frac{D_1D_2^2}{H}} 
    \sum_{u \pmod{\ell}} 
    e\left(\frac{-d_1d_2d_2'u^2 + (td_2'-t'd_2)u}{\ell}\right)
\]
by less than $D_1^2 D_2^4 H^{-1/2}$.
We apply Poisson summation once again in $d_1, t, t'$, which creates dual variables $m, n, n'$ obeying congruences
\[
    \begin{cases}
    m \equiv d_2d_2'u^2 \\ 
    n \equiv d_2'u \\ 
    n' \equiv d_2u
    \end{cases}
    \pmod{\ell}
    \qquad 
    \longleftrightarrow
    \qquad 
    \begin{cases}
    m \equiv nn' \\
    nd_2 \equiv n'd_2' \\ 
    u \equiv n'\overline{d}_2
    \end{cases}
    \pmod{\ell}.
\]
In particular, the $n, n'$ variables came from two applications of Poisson summation ($r \rightsquigarrow t \rightsquigarrow n$) with different moduli due to the reciprocity step ($d_1d_2$, then $\ell$); the fact that the original $r$-variable is supported away from $0$ allows us to remove the principal frequencies $n, n' = 0$. Focusing on the critical ranges, the sum over $u \bmod \ell$ is absorbed by the last congruence above, and it remains to bound
\[
    \sum_{m \sim \frac{D_2^2}{H}}
    \sum_{n, n' \sim \frac{D_2\sqrt{X}}{H}} \sum_{\substack{\ell \mid nn'-m \\ \ell \sim \frac{D_1D_2^2}{H}}}\
    \sum_{\substack{d_2, d_2' \sim D_2 \\ (d_2, d_2') = 1}} \beta_{d_2} \bar\beta_{d_2'} \one_{nd_2 \equiv n'd_2' \pmod{\ell}}
\]
by a power-saving over $XD_2^2/(D_1\sqrt{H})$. Note that we have arrived at a new counting problem. A short geometry-of-numbers argument\footnote{One can alternatively proceed by replacing $\beta_{d_2}, \bar\beta_{d_2'}$ with smooth weights for an upper bound, and Poisson summation in $d_2, d_2'$. This route requires leveraging the averaging over the other variables.} (Lemma~\ref{lem:prim-points-count}) bounds the inner sum over $d_2, d_2'$ by $\ll 1 + (n, n', \ell) D_2^2/\ell$, which completes our sketch-of-proof as long as
\[
    \frac{D_2^2}{H} \cdot \frac{D_2^2X}{H^2} \left(1 + \frac{H}{D_1}\right)
    < \frac{XD_2^2}{D_1\sqrt{H}} \qquad \iff \qquad 
    \begin{cases} 
    D_1D_2^2 < H^{5/2}, \\ 
    D_2 < H^{3/4}.
    \end{cases}
\]
But the condition $D_2 < H^{3/4}$ (with a power-saving) is part of the assumption of Proposition~\ref{prop:typeIfact}, while $D_1D_2^2 < H^{5/2}$ follows from the assumptions $D_1 < H$ and $D_2 < H^{3/4}$.

\subsection{Sketch of Type II treatment}
\label{sec:typeiisketch}
In this subsection, we will briefly outline the treatment of Proposition \ref{prop:typeII}, the Type II 
information, with the hope of exposing where the ranges come from. We will be choosing to work with the indicator
function of the squares directly, and ignore smoothings and main terms that arise via the model $\sigma_0$. 
What we describe below is morally the content of \S\ref{sec:typeII}, though in practice the argument needs to be arranged more carefully to ensure cancellation of main terms and to deal with possible common factors. 

We are essentially led to nontrivially bound/evaluate, for $X^{1/2 + \eta}/H < N < X^{-\eta}\sqrt{H}$, the quantity
\begin{equation*}
    \sum_{\substack{h\sim H }} \sum_{m \sim M} \sum_{n\sim N} \alpha_m \beta_n\gamma_h \mathbf{1}_{\square}(mn - h).
\end{equation*}
The trivial bound can be seen to be $X^{1/2}H$, and it is by a power saving over this we wish to bound
the remainder term in the coming evaluations (the treatment of the main terms, and the strength of our approximant
$\sigma_0$ is what ultimately dictates the strength of the final bound in Proposition \ref{prop:typeII}).

After applying Cauchy-Schwarz on $m$ and rearranging, we are led to evaluate
\begin{align*}
  \sum_{\substack{n_1, n_2\sim N \\ h_1, h_2\sim H}} &\beta_{n_1}\conj{\beta_{n_2}}\gamma_{h_1} \overline{\gamma_{h_2}} \sum_{r_1, r_2\sim X^{1/2}  } \sum_{m\sim M } \mathbf{1}_{\substack{mn_1 - h_1 = r_1^2 \\ mn_2 - h_2 = r_2^2 }} \\ 
  &= \sum_{\substack{n_1, n_2\sim N \\ h_1, h_2 \sim H }}\gamma_{h_1} \overline{\gamma_{h_2}} \beta_{n_1} \overline{\beta_{n_2}} \sum_{r_1, r_2\sim X^{1/2} }
  \mathbf{1}_{ n_2 r_1^2 - n_1r_2^2 = n_1h_2 - n_2h_1 }. 
\end{align*}
It is at this point that the lower bound on $N$ is required, to ensure that the diagonal terms $n_1=n_2$ do not exceed
the trivial bound for the off-diagonal.

This inner sum we shall call $r_{Q_{n_1, n_2}}(n_1h_2 - n_2h_1; X^{1/2}) $ in this sketch, to indicate that it is nearly
the number of representations of $n_1h_2 - n_2h_1 $ by the quadratic form $Q_{n_1, n_2}(x_1, x_2) = n_2x_1^2 - n_1x_2^2 $.
We use this notation to suggest that, in the generic case $(n_1, n_2) = 1 $, $r_{Q_{n_1, n_2}}(u) $ is 
morally similar to the Fourier coefficient of a modular form of level $\ll N^2 $, with the difference that $r_{Q_{n_1, n_2}}(u) $  is supported on a set of density $\asymp 1/N$. The dependence on the second component
may be omitted; it turns out to be inconsequential\footnote{A way to anticipate this may be to weaken
  the restrictions and recapture it with Archimedean characters of low conductor.}.
In reality, in place of the representation number $r_Q$ itself, we work with $\mathcal{Q}_{i, j}$, defined in 
\S\ref{sec:quadratic}, which incorporate the Archimedean constraints and through the parameters $i, j\in \{0, 1\}$, 
unify the treatment of the representation function with the expected main term on average.

Smoothing out the $\gamma_{h_i} $ factors with Cauchy-Schwarz, we are reduced to evaluating
\begin{equation*}
  \sum_{h_1, h_2\sim H } \sum_{n_1, n_2, n_3, n_4\sim N } \beta_{n_1} \overline{\beta_{n_2}}\overline{\beta_{n_3}}\beta_{n_4}
  r_{Q_{n_1, n_2}}(n_1h_2 - n_2 h_1) r_{Q_{n_3, n_4}}(n_3h_2 - n_4h_1).
\end{equation*}
At this point, the trivial bound is $ H^2 N^2  $. 

Writing $u_1 = n_1h_2 - n_2h_1 $ and $u_2 = n_3h_2 - n_4h_1 $, and letting $\Delta_\mathbf{n} = n_1n_4 -  n_2n_3 $, 
we have that the above is (ignoring the diagonal case $\Delta = 0 $, and assuming coprimality), roughly,
\begin{equation*}
  \sum_{n_1,n_2, n_3, n_4\sim N } \beta_{n_1} \overline{\beta_{n_2}}\overline{\beta_{n_3}}\beta_{n_4} \sum_{\substack{u_1, u_2\sim H N \\ u_1 \equiv n_1 \overline{n_3} u_2 \, (\Delta_\mathbf{n})}}
  r_{Q_{n_1, n_2}}(u_1) r_{Q_{n_3, n_4}}(u_2).
\end{equation*}
Separating $u_1 $, $u_2 $ with orthogonality yields that the above is 
\begin{equation*}
  \ll \sum_{n_1, n_2, n_3, n_4\sim N } \frac{1}{|\Delta_\mathbf{n}|} \sum_{a\, (\Delta_\mathbf{n}) } \biggl| \sum_{u_1 \sim H N} r_{Q_{n_1, n_2}}(u_1)e\biggl(\frac{an_3 u_1}{\Delta_\mathbf{n}}\biggr) \biggr|
  \biggl| \sum_{u_2\sim H N } r_{Q_{n_3, n_4}}(u_2) e\biggl(-\frac{a n_1 u_2}{\Delta_{\mathbf{n}}}\biggr) \biggr|. 
\end{equation*}
Restricting to the generic case of $\Delta_\mathbf{n} \asymp N^2 $, by Cauchy-Schwarz to separate $u_1$ and $u_2$, we are thus reduced to
beating the bound of $O(H^2N^2) $ for 
\begin{equation}\label{eq:pre_voronoi}
  \sum_{n_1, n_2, n_3, n_4 \sim N } \frac{1}{|\Delta_\mathbf{n}|} \sum_{a\, (\Delta_\mathbf{n}) } \biggl| \sum_{u\sim H N } r_{Q_{n_1, n_2}}(u) e \biggl(\frac{a u}{\Delta_\mathbf{n}}\biggr) \biggr|^2.
\end{equation}
To this end, we apply Voronoï summation (underneath, Poisson twice, as executed by Lemma \ref{le:voronoi})
to evaluate the inner sum. Here the main term from the Voronoï summation cancels the suppressed main terms from the model $\sigma_0$ when $a$ and $\Delta_{\mathbf{n}}$ have a large common factor. In this sketch we focus on $(a,\Delta_\mathbf{n})=1$ and the dual sum is then essentially
\begin{equation*}
  \frac{H}{N^2}\sum_{u\ll N^5/H} r_{Q_{n_1, n_2}}(u) e \biggl(-\frac{\overline{a} u}{\Delta_{\mathbf{n}}}\biggr).
\end{equation*}

The large sieve yields a bound of
\begin{align*}
  \sum_{n_1, n_2, n_3, n_4  \sim N } &\frac{1}{|\Delta_\mathbf{n}|} \sum_{a\, (\Delta_\mathbf{n})^\times } \biggl| \sum_{u\sim H N } r_{Q_{n_1, n_2}}(u) e \biggl(\frac{a u}{\Delta_\mathbf{n}}\biggr) \biggr|^2 \\ 
                               &\ll \frac{H^2}{N^6} \sum_{n_1, n_2, n_3, n_4 \sim N } \sum_{a\, (\Delta_\mathbf{n})^\times } \biggl| \sum_{u\ll  N^5/H } r_{Q_{n_1, n_2}}(u)e \biggl(\frac{a u}{\Delta_\mathbf{n}}\biggr) \biggr|^2 \\
                               &\ll  \frac{H^2}{N^6} \sum_{n_1, n_2\sim N } \sum_{\Delta\sim N^2 } \sum_{a(\Delta)^\times}  \biggl| \sum_{u\ll N^5/H } r_{Q_{n_1, n_2}}(u) e\biggl(\frac{a u}{\Delta}\biggr)\biggr|^2 \\ 
                               &\ll  \frac{H^2}{N^6}N^2 \biggl( N^4 + \frac{N^5}{H} \biggr) \frac{N^4}{H}\ll
                                 \frac{H^2}{N^6}N^2 N^4 \frac{N^4}{H} = H N^4,
\end{align*}
recalling that $r_{Q_{n_1, n_2}}(u) $ has support of density roughly $1/N$. This bound is a power saving over the reference bound of $O(H^2N^2)$ for \eqref{eq:pre_voronoi} when 
$N\ll \sqrt{H}X^{-\eta}$, precisely the upper bound of Proposition \ref{prop:typeII}. 

\section{Preliminaries}
In this section we have collected basic results needed in the proofs of Propositions \ref{prop:typeI},  \ref{prop:typeI2}, and \ref{prop:typeII}.
\subsection{Poisson summation}
We need the following truncated Poisson summation formula. The proof is standard.
\begin{lemma}[Truncated Poisson summation formula] \label{le:poisson} Let $\delta,\eps >0$ and  let $\psi$ be a smooth function supported in $[-10,10]$ with $\psi^{(J)} \ll_J\delta^{-J}$ for all $J\geq 0$. Let $N> 1$ and let $q \geq 1$ be an integer. Let
\begin{align*}
  T \geq  N^\eps \delta^{-1} q/N.  
\end{align*}
Then for any $C >0$
\begin{align*}
\sum_{n \equiv a \pmod{q}} \psi\bigg(\frac{n}{N}\bigg) = \frac{N}{q} \widehat{\psi}(0) + \frac{N}{q} \sum_{1 \leq |t| \leq T} \widehat{\psi} \bigg( \frac{t N}{q}\bigg) e_q (at) + O_{\eps,C}(N^{-C}),
\end{align*}
where $\widehat{\psi}(h):= \int_\R \psi(u)e(-hu) \d u$ is the Fourier transform.
\end{lemma}

\subsection{Large Sieve inequalities}
We also need the following classical large sieve for additive characters.
\begin{lemma}\label{lem:LaSi}
    We have for $M$, $N$ and complex coefficients $\alpha_n$ that
\begin{align*}
    \sum_{q\leq Q}\sum_{\substack{a\pmod q^\times}} \bigg|\sum_{n\in[M,M+N]}\alpha_n e_q(an)\bigg|^2 \ll ( Q^2+N)\sum_n |\alpha_n|^2.
\end{align*}
\end{lemma}

To show that the model $\sigma_0$ agrees with the squares, we use the large sieve for quadratic characters due to Heath-Brown \cite[Corollary 4]{hbquad}.
\begin{lemma} \label{lem:quadlargesieve}
Let $M,N>2$ and let $\alpha_m, \beta_n$ be 1-bounded complex numbers. Then
    \begin{align*}
      \sum_{\substack{m \leq M \\ m \, \mathrm{odd}}} \sum_{n \leq N} \alpha_m \beta_n \left( \frac{n}{m}\right) \ll (MN)^{o(1)} (M N^{1/2} + NM^{1/2}).
    \end{align*}
    Furthermore, if $M < N^{1/2-\eps}$ or $M > N^{2+\eps}$, then the term $(MN)^{o(1)}$ may be replaced by a logarithmic factor $(\log MN)^{1/2}$. 
\end{lemma}
\begin{proof}
 Only the additional remark needs to be justified. If $M >N^{2+\eps}$, we have by the Cauchy-Schwarz inequality
    \begin{align*}
         \sum_{\substack{m \leq M \\ m \, \mathrm{odd}}} \alpha_m \sum_{n \leq N} \beta_n \left( \frac{n}{m}\right)  \ll M^{1/2} V^{1/2}
    \end{align*}
    where, using a smooth non-negative weight function $\psi$ supported on $[-1, 2]$ such that $\psi(x) \geq 1$ for $x \in [0, 1]$, we define
    \begin{align*}
        V = \sum_{m \, \mathrm{odd}} \psi\left(\frac{m}{M}\right) \bigg| \sum_{n \leq N} \beta_n \left( \frac{n}{m}\right)   \bigg|^2.
    \end{align*}
    Expanding the square and pulling the bounded coefficients $\beta_{n}$ outside the inner sum, we obtain
    \begin{align*}
        V \leq \sum_{n_1,n_2 \leq N} |\beta_{n_1} \beta_{n_2}| \bigg| \sum_{m \, \mathrm{odd}} \psi\left(\frac{m}{M}\right)  \left( \frac{n_1 n_2}{m}\right) \bigg| \ll \sum_{n_1,n_2 \leq N} \bigg| \sum_{m \, \mathrm{odd}} \psi\left(\frac{m}{M}\right)  \left( \frac{n_1 n_2}{m}\right) \bigg|.
    \end{align*}

Since $M > N^{2+\eps}$, by the Poisson summation formula (Lemma \ref{le:poisson}) and quadratic reciprocity we get
    \begin{align*}
        V  \ll M  \sum_{n_1,n_2 \leq N}  \mathbf{1}_{n_1 n_2 = \square} \ll M N \log N.
    \end{align*}
    Substituting this back yields the bound $M N^{1/2} (\log N)^{1/2}$. 
    
    The case $M < N^{1/2-\eps}$ is handled similarly, swapping the roles of $m$ and $n$.
\end{proof}

\subsection{The model for square-free moduli}
Let us denote
\begin{align*}
    \varrho_{h}(d) := \#\{\omega \in \Z/ d \Z: \omega^2+h \equiv 0 \pmod{d}  \}.
\end{align*}
This is multiplicative and for square-free $d$ we have in terms of the Jacobi symbol
\begin{align} \label{eq:numberofroots}
     \varrho_{h}(d)  = \prod_{2<p|d} \left(1 + \left(\frac{-h}{p} \right)\right) = \sum_{\substack{r|d \\ r \, \text{odd}}}  \left(\frac{-h}{r} \right).
\end{align}
To compare this with the model we need the following standard fact.
\begin{lemma} \label{le:modelforsquarefree}
Let $    G(\alpha,r) $ be as in \eqref{eq:gausssumdef}.
    For square-free $r \geq 1$, we have
    \begin{align*}
      \frac{1}{r}  \sum_{\alpha \pmod r^\times} G(\alpha, r) e_r(-\alpha k) = \mathbf{1}_{r\, \mathrm{odd}} \left( \frac{k}{r} \right).
    \end{align*}
\end{lemma}

\begin{proof}
Let $S(k, r)$ denote $r$ times the sum in question. Expanding the Gauss sum, we can rewrite $S(k, r)$ by swapping the order of summation
$$S(k, r) = \sum_{\alpha \pmod r^\times} \sum_{u \pmod r} e_r(\alpha u^2) e_r(-\alpha k) = \sum_{u \pmod r} \sum_{\alpha \pmod r^\times} e_r(\alpha (u^2 - k)).$$ Because $r$ is square-free, by the Chinese remainder theorem we have
$$S(k, r) = \prod_{p \mid r} S(k, p) = \prod_{p \mid r} \left( \sum_{u \pmod p} \sum_{\alpha \pmod p^\times} e_p(\alpha (u^2 - k)) \right).$$ 

If $p=2 \mid r$, we evaluate  $S(k, 2)$:
$$S(k, 2) = \sum_{u = 0}^1 \sum_{\alpha = 1}^1 e_2(\alpha (u^2 - k)) = e_2(-k) + e_2(1 - k) = 0.$$

If $p$ is an odd prime, for a fixed $u$, the inner sum over $\alpha$ is 
$$\sum_{\alpha = 1}^{p-1} e_p(\alpha (u^2 - k)) = \begin{cases} p - 1 & \text{if } u^2 \equiv k \pmod p \\ -1 & \text{if } u^2 \not\equiv k \pmod p \end{cases}$$
Let $N(k, p)$ be the number of solutions to the congruence $u^2 \equiv k \pmod p$, expressed using the Legendre symbol $N(k, p) = 1 + \left( \frac{k}{p} \right)$. There are $N(k, p)$ values of $u$ that yield $p - 1$, and the remaining $p - N(k, p)$ values of $u$ yield $-1$. Summing over all $u \pmod p$ we get
$$S(k, p) = p \left( 1 + \left( \frac{k}{p} \right) \right) - p = p \left( \frac{k}{p} \right).$$ The claim now follows by multiplying over all $p|r$ and dividing by $r$.
\end{proof}

\section{Proof of Proposition \ref{prop:typeI}} \label{sec:typeI}
We consider
\begin{align*}
    \sum_{n}  w^{(\gamma)}_{dn} = \sum_{n}  a^{(\gamma)}_{dn} - \sum_{n}  b^{(\gamma)}_{dn}.
\end{align*}
We apply Poisson summation (Lemma \ref{le:poisson}) to the $k=x^2$ variable in $a_n^{(\gamma)}$ with the modulus $d \leq X^{1/2-\eta}$ and only get the $t=0$ contribution. Similarly, we evaluate the model $b_n^{(\gamma)}$ using Poisson summation (Lemma \ref{le:poisson}) with the modulus $[d,r] < X^{1/2}$ and the main term vanishes unless $r|d$ since $(\alpha,r)=1$. Since $d$ is square-free, comparing the main terms using \eqref{eq:numberofroots}  and  Lemma \ref{le:modelforsquarefree} we get
\begin{align*}
    \bigg| \sum_{n} w^{(\gamma)}_{dn} \bigg| \ll \frac{X^{1/2}}{d} \bigg| \sum_{h \leq H} \gamma_h \sum_{\substack{r \mid d \\ r > R \\ r \text{ odd}}} \left(\frac{-h}{r}\right) \bigg| + X^{-100}.
\end{align*}
By the triangle inequality we get
\begin{align*}
   \sum_{d \leq D} \mu^2(d) \bigg| \sum_{n}  w^{(\gamma)}_{dn} \bigg|    \ll  X^{1/2} (\log X) \sum_{\substack{R < r \leq D \\ r \text{ odd}}} \frac{\mu^2(r)}{r} \bigg| \sum_{h \leq H} \gamma_h \left(\frac{-h}{r}\right) \bigg|.
\end{align*}
We split the sum into dyadic intervals $r \sim E$ where $R \ll E \ll D$. Applying the quadratic large sieve (Lemma \ref{lem:quadlargesieve}) we conclude
\begin{align*}
   \sum_{d \leq D} \mu^2(d) \bigg| \sum_{n}  w^{(\gamma)}_{dn} \bigg|  \precprec X^{1/2} \max_{R \leq E \leq D} \frac{1}{E} ( H  E^{1/2} + E H^{1/2})    \precprec X^{1/2} H R^{-1/2}. \qed
\end{align*}

\section{Proof of Proposition \ref{prop:typeIfact}} \label{sec:typeIfact}

\subsection{Extracting the main term}
Our proof begins with Poisson summation, similarly as in Section \ref{sec:typeI} -- but this time we obtain an error term with potentially many frequencies $t \neq 0$.

\begin{lemma}\label{lem:typeIfact-poisson}
Under the assumptions of Proposition \ref{prop:typeIfact}, we have
\[
    \mathop{\sum_{d_1 \sim D_1} \sum_{d_2 \sim D_2}}_{(d_1,d_2)=1}
    \alpha_{d_1} \beta_{d_2}
    \sum_n w_{d_1d_2n}^{(\gamma)}
    \precprec 
    HX^{1/2} R^{-1/2}
    +
    |\mE_1|, 
\]
\begin{equation} \label{eq:E1}
    \mE_1 := 
    X^{1/2} \mathop{\sum_{d_1 \sim D_1} \sum_{d_2 \sim D_2}}_{(d_1,d_2)=1}
    \frac{\alpha_{d_1} \beta_{d_2}}{d_1 d_2}
    \sum_h \gamma_h \sum_{\substack{\omega \pmod{d_1d_2} \\ \omega^2+h \equiv 0 \pmod{d_1d_2}}}
    \sum_{t \neq 0}
    \widehat\psi_f\left(\frac{tX^{1/2}}{d_1d_2}\right) e_{d_1d_2}(\omega t).
\end{equation}
\end{lemma}

\begin{proof}
We consider again
\[
\begin{aligned}
    \sum_n w_{dn}^{(\gamma)}
    &=
    \sum_n a_{dn}^{(\gamma)}
    -
    \sum_n b_{dn}^{(\gamma)},
\end{aligned}
\]
where $d = d_1d_2$ is square-free, $d \asymp D_1D_2$.
As in Section \ref{sec:typeI}, we apply Poisson summation with moduli $d, [d, r] \le X^{0.9}$. Specifically, letting $\psi_f(u) := \frac12 f(u^2)$, so that $\widehat\psi_f(0)=I_f$, we have by Poisson summation (as in Lemma \ref{le:poisson} but without truncating the sum over $t$, for later technical convenience)
\begin{equation} \label{eq:adn-poisson}
\begin{aligned}
    \sum_n a_{dn}^{(\gamma)}
    &=
    \sum_h \gamma_h
    \sum_{\substack{\omega \pmod{d} \\ \omega^2+h \equiv 0 \pmod{d}}}
    \sum_{x \equiv \omega \pmod{d}} \psi_f\left(\frac{x}{\sqrt{X}}\right) 
    \\
    &=
    \frac{X^{1/2}}{d}
    \sum_h \gamma_h \sum_{\substack{\omega \pmod{d} \\ \omega^2+h \equiv 0 \pmod{d}}}
    \sum_t
    \widehat\psi_f\left(\frac{tX^{1/2}}{d}\right) e_d(\omega t).
\end{aligned}
\end{equation}
Similarly, by Lemma \ref{le:poisson} and then Lemma \ref{le:modelforsquarefree}, we have
\[
\begin{aligned}
    \sum_n b_{dn}^{(\gamma)}
    &=
    \sum_h \gamma_h
    \sum_{r \le R} \frac{1}{r} \sum_{\alpha \pmod{r}^\times} G(\alpha, r) 
    \sum_{\substack{\beta \pmod{[d,r]} \\ \beta \equiv -h \pmod{d}}}
    e_r(-\alpha\beta)
    \sum_{k \equiv \beta \pmod{[d,r]}} \psi_f\left(\frac{\sqrt{k}}{\sqrt{X}}\right) \frac{1}{\sqrt{k}}
    \\
    &=
    \sum_h \gamma_h
    \sum_{r \le R} \frac{1}{r} \sum_{\alpha \pmod{r}^\times} G(\alpha, r)\, 
    e_r(\alpha h)
    \one_{r \mid d}
    \frac{X^{1/2}}{[d,r]}\widehat\psi_f(0) + O(X^{-100})
    \\
    &=
    \frac{X^{1/2}}{d} \widehat\psi_f(0)
    \sum_h \gamma_h
    \sum_{\substack{r \mid d \\ r \le R \\ r \text{ odd}}} \left(\frac{-h}{r}\right)
    + O(X^{-100}).
\end{aligned}
\]
Applying \eqref{eq:numberofroots} to the term $t = 0$ from \eqref{eq:adn-poisson}, it follows that
\[
\begin{aligned}
    \sum_n w_{dn}^{(\gamma)}
    &=
    \frac{X^{1/2}}{d}
    \sum_h \gamma_h \sum_{\substack{\omega \pmod{d} \\ \omega^2+h \equiv 0 \pmod{d}}}
    \sum_{1 \le |t| \le D_1D_2 X^{\eps-1/2}}
    \widehat\psi_f\left(\frac{tX^{1/2}}{d}\right) e_d(\omega t)
    \\
    &+
    O\bigg(\frac{X^{1/2}}{d}\bigg\vert \sum_{h \le H} \gamma_h \sum_{\substack{r \mid d \\ r > R \\ r \text{ odd}}} \left(\frac{-h}{r}\right) \bigg\vert \bigg)
    +
    O(X^{-100}).
\end{aligned}
\]
Summing over square-free moduli $d = d_1d_2$, with $d_1 \sim D_1$, $d_2 \sim D_2$ and divisor-bounded coefficients, the first line above leads to \eqref{eq:E1}. The error term on the second line gives an acceptable contribution by the quadratic large sieve (Lemma \ref{lem:quadlargesieve}), just as in Section \ref{sec:typeI}.
\end{proof}

\subsection{The Cauchy--Schwarz step}
\label{subsec:typeIfact-Cauchy-Schwarz}
We now use the factorization of the moduli $d = d_1d_2$ by applying Cauchy--Schwarz with $d_1, h$ outside and $d_2, t$ inside.
\begin{lemma} \label{lem:typeIfact-CS}
For $\mE_1$ as in \eqref{eq:E1} and any $\eps \in (0, 1)$, one has
\[
    |\mE_1|^2 \precprec_\eps H^2XR^{-1} + H^2 X^{1/2} D_1 + |\mE_2|,
\]
\begin{equation} \label{eq:E2}
\begin{aligned}
    \mE_2 
    := \frac{XH^2}{D_1}
    \sum_{g \le 2D_2} 
    \frac{1}{g^3}
    \bigg\vert
    \sum_{\substack{b, b' \sim D_2/g \\ (b, b') = 1}}
    \frac{\beta_{gb}\bar\beta_{gb'}}{(bb')^2}
    \sum_{d_1}
    \frac{(d_1, g)}{d_1}
    \psi_1\left(\frac{d_1}{D_1}\right)
    \sum_{1 \le |\ell| \le X^\eps \frac{D_1D_2^2}{gH}}
    \widehat\psi_0\left(\frac{\ell H}{d_1gbb'}\right)
    \\
    \times \sum_{t, t' \neq 0}
    \widehat\psi_f\left(\frac{tX^{1/2}}{d_1gb}\right) 
    \overline{\widehat\psi_f}\left(\frac{t'X^{1/2}}{d_1gb'}\right) 
    \sum_{\omega \pmod{d_1gbb'}}
    e_{d_1gbb'}(-\ell\omega^2 + \omega(b't-bt'))
    \bigg\vert.
\end{aligned}
\end{equation}
\end{lemma}

\begin{remark}
The term $X^{1/2} H^2 D_1$ above matches the contribution of the diagonal terms from the Cauchy--Schwarz step, and gives an acceptable error in Proposition \ref{prop:typeIfact} as long as
\begin{equation} \label{eq:typeIfact-diagonals}
    H^2 X^{1/2} D_1 
    \precprec 
    (H X^{1/2} R^{-1/2})^2
    \qquad 
    \iff 
    \qquad 
    D_1 \precprec X^{1/2} R^{-1}.
\end{equation}
\end{remark}

\begin{proof}[Proof of Lemma \ref{lem:typeIfact-CS}]
We first remove the coprimality constraint $(d_1, d_2) = 1$ by M\"obius inversion in \eqref{eq:E1} with a parameter $r$, and retain the value of $\omega \equiv \omega_1 \pmod{d_1}$. Using the triangle inequality in $d_1, \omega_1, r, h$, we reach
\begin{equation} \label{eq:E1-triangle}
    |\mE_1|
    \precprec
    \frac{X^{1/2}}{D_1} \sum_{\substack{d_1 \sim D_1, r \mid d_1 \\ \omega_1 \pmod{d_1} \\ \omega_1^2+h \equiv 0 \pmod{d_1}}}
    \sum_{h \le H} 
    \bigg\vert
    \sum_{\substack{d_2 \sim D_2 \\ d_2 \equiv 0 \pmod{r}}}
    \frac{\beta_{d_2}}{d_2}
    \sum_{t \neq 0}
    \widehat\psi_f\left(\frac{tX^{1/2}}{d_1d_2}\right) 
    \sum_{\substack{\omega \pmod{d_1d_2} \\ \omega^2+h \equiv 0 \pmod{d_1d_2} \\ \omega \equiv \omega_1 \pmod{d_1}}}
    e_{d_1d_2}(\omega t)
    \bigg\vert.
\end{equation}
Note that $d_1$ may no longer be square-free, but $d_2$ is (due to the support of $\beta_{d_2}$). 

Let us call an $h$-root $\omega \pmod{d}$ any solution to $\omega^2 + h \equiv 0 \pmod{d}$.
Applying Cauchy--Schwarz in \eqref{eq:E1-triangle} would normally produce two $h$-roots $\omega \pmod{d_1d_2}$ and $\omega' \pmod{d_1d_2'}$, which may not agree modulo $(d_2,d_2')$. We intend to keep `as much as possible' of the sum over $\omega$ outside the Cauchy--Schwarz step, to ultimately arrive at a single $h$-root $\omega \pmod{d_1[d_2,d_2']}$. Avoiding extraneous dependencies on the $h$-variable requires a somewhat technical set-up.

\textbf{Preparations for root coupling:}
For any prime $p \le 2D_2$, define 
\[
    j_{d_1}(p) := 
    \begin{cases}
    2, & p \nmid d_1, \\ 
    p, & p \mid d_1,
    \end{cases}
    \qquad\quad 
    J_{d_1}(p) := 
    \{1, 2, \ldots, j_{d_1}(p)\}.
\] 
Given an $h$-root $\omega_1 \pmod{d_1}$, 
denote by $i_{d_1}(p, h, \omega_1)$ the number of ways how $\omega_1 \pmod{p^{v_p(d_1)}}$ extends to an $h$-root $\omega \pmod{p^{v_p(d_1)+1}}$; then we always have $i_{d_1}(p, h, \omega_1) \le j_{d_1}(p)$. To the \emph{active} indices $j \le i_{d_1}(p, h, \omega_1)$, we assign the extended $h$-roots $\omega_{d_1,j}(p, h, \omega_1) \pmod{p^{v_p(d_1)+1}}$.

For square-free $d_2$, we denote
\[
    j_{d_1}(d_2) := \prod_{\text{prime }p \mid d_2} j_{d_1}(p),
    \qquad 
    J_{d_1}(d_2) := \prod_{\text{prime }p \mid d_2} J_{d_1}(p).
\]
We call an index $\vec{j} \in J_{d_1}(d_2)$ \emph{active} (with respect to $h$ and $\omega_1$) iff all of its local components are active. We assign $h$-roots $\omega_{d_1,\vec{j}}(d_2, h, \omega_1) \pmod{d_1d_2}$ to the active indices $\vec{j} \in J_{d_1}(d_2)$ in the unique way that is compatible with the local extended $h$-roots by the Chinese Remainder Theorem, and we let
\begin{equation} \label{eq:E-def}
    E_{d_1,\vec{j}}(d_2, h, \omega_1, t) := 
    \begin{cases}
        e_{d_1d_2}(\omega_{d_1,\vec{j}}(d_2,h,\omega_1) t),
        & \vec{j} \in J_{d_1}(d_2) \text{ is active},
        \\ 
        0, & \text{otherwise}.
    \end{cases}
\end{equation}
Globally, we write 
\[
    Q := \prod_{\text{prime }p \le 2D_2} p,
\]
and we may restrict global indices $\vec{j} \in J_{d_1}(Q)$ to those components at primes $p \mid d_2$ by writing $\vec{j} \vert_{d_2} \in J_{d_1}(d_2)$.
Then we have 
\begin{equation} \label{eq:CRT-trick}
    \sum_{\substack{\omega \pmod{d_1d_2} \\ \omega^2+h \equiv 0 \pmod{d_1d_2} \\ \omega \equiv \omega_1 \pmod{d_1}}}
    e_{d_1d_2}(\omega t)
    =
    \sum_{\vec{j} \in J_{d_1}(d_2)}
    E_{d_1,\vec{j}}(d_2, h, \omega_1, t)
    =
    \frac{j_{d_1}(d_2)}{j_{d_1}(Q)}
    \sum_{\vec{j} \in J_{d_1}(Q)}
    E_{d_1,\vec{j}\vert_{d_2}}(d_2, h, \omega_1, t),
\end{equation}
since each index $\vec{j} \in J_{d_1}(d_2)$ extends to an index $\vec{j} \in J_{d_1}(Q)$ (which may or may not be active at primes $p \nmid d_2$) in exactly $j_{d_1}(Q/d_2) = j_{d_1}(Q)/j_{d_1}(d_2)$ ways.

\textbf{Cauchy--Schwarz and root coupling:}
We now plug \eqref{eq:CRT-trick} into \eqref{eq:E1-triangle}, pull $\vec{j}$ out of the triangle inequality, and apply Cauchy--Schwarz with the sums over $d_1, \omega_1, r, h, \vec{j}$ outside. We then insert smooth weights in the $h$ and $d_1$ variables:
let $\psi_0, \psi_1 : \R \to [0, \infty)$ be fixed smooth functions supported in $[-0.5, 1.5]$ and $[0.5, 2.5]$ respectively, with $\psi_0(t_0), \psi_1(t_1) \gg 1$ for $t_0 \in [0, 1]$, $t_1 \in [1, 2]$. After inserting these weights, expanding the square, and writing
\[
    g = (d_2, d_2') \equiv 0 \pmod{r},
    \qquad 
    d_2 = gb, 
    \qquad 
    d_2' = gb',
\]
where $b, b', g$ are square-free and pairwise-coprime, we obtain
\[
\begin{aligned}
    |\mE_1|^2 \precprec 
    \frac{XH}{D_1} \sum_{\substack{d_1 \\ r \mid d_1 \\ \omega_1 \pmod{d_1} \\ \omega_1^2+h \equiv 0 \pmod{d_1}}}
    \psi_1\left(\frac{d_1}{D_1}\right)
    \sum_h \psi_0\left(\frac{h}{H}\right)
    \frac{1}{j_{d_1}(Q)}
    \sum_{\vec{j} \in J_{d_1}(Q)}
    \sum_{\substack{g \le 2D_2 \\ g \equiv 0 \pmod{r}}} \sum_{\substack{b, b' \sim D_2/g \\ (b, b') = 1}}
    \frac{\beta_{gb} \bar\beta_{gb'}}{g^2bb'}
    \\
    \times\, 
    j_{d_1}(gb) j_{d_1}(gb')
    \sum_{t, t' \neq 0}
    \widehat\psi_f\left(\frac{tX^{1/2}}{d_1gb}\right)
    \overline{\widehat\psi_f}\left(\frac{t'X^{1/2}}{d_1gb'}\right) 
    E_{d_1,\vec{j}\vert_{gb}}(gb,h,\omega_1,t)
    \overline{E}_{d_1,\vec{j}\vert_{gb'}}(gb',h,\omega_1,t').
\end{aligned}
\]
One easily sees from \eqref{eq:E-def} that
\[
    E_{d_1,\vec{j}\vert_{gb}}(gb,h,\omega_1,t)
    \overline{E}_{d_1,\vec{j}\vert_{gb'}}(gb',h,\omega_1,t')
    =
    E_{d_1,\vec{j}\vert_{gbb'}}(gbb',h,\omega_1,b't-bt'),
\]
and applying \eqref{eq:CRT-trick} backwards (with $d_2$ replaced by $gbb'$ and $t$ replaced by $b't-bt'$) gives
\[
    \frac{j_{d_1}(gbb')}{j_{d_1}(Q)}
    \sum_{\vec{j} \in J_{d_1}(Q)}
    E_{d_1,\vec{j}\vert_{gbb'}}(gbb',h,\omega_1,b't-bt')
    =
    \sum_{\substack{\omega \pmod{d_1gbb'} \\ \omega^2+h \equiv 0 \pmod{d_1gbb'} \\ \omega \equiv \omega_1 \pmod{d_1}}}
    e_{d_1gbb'}(\omega(b't-bt')).
\]
We may now plug this into our bound for $|\mE_1|^2$ and combine the sums over $\omega_1, \omega$. Since $j_{d_1}(gb) j_{d_1}(gb') = j_{d_1}(g) j_{d_1}(gbb')$, this leaves behind a factor of $j_{d_1}(g)$; by combining this with the sum over $r \mid (d_1, g)$, we obtain (since $g$ is square-free)
\[
    j_{d_1}(g) \sum_{r \mid (d_1, g)} 1
    =
    \prod_{p \mid g} (2 \cdot \one_{p \nmid d_1} + 2p \cdot \one_{p \mid d_1})
    =
    2^{\Omega(g)} 
    (d_1, g).
\]
Moving the sum over $g$ to the front and the sum over $h$ to the end, we conclude by the triangle inequality and the divisor bound in $g$ that
\begin{equation} \label{eq:before-poisson-h}
\begin{aligned}
    |\mE_1|^2 \precprec 
    \frac{XH}{D_1} 
    \sum_{g \le 2D_2} 
    \bigg\vert 
    \sum_{\substack{b, b' \sim D_2/g \\ (b, b') = 1}}
    \frac{\beta_{gb} \bar\beta_{gb'}}{g^2bb'}
    \sum_{d_1}
    (d_1, g)
    \psi_1\left(\frac{d_1}{D_1}\right)
    \sum_{t, t' \neq 0}
    \widehat\psi_f\left(\frac{tX^{1/2}}{d_1gb}\right)
    \overline{\widehat\psi_f}\left(\frac{t'X^{1/2}}{d_1gb'}\right) 
    \\
    \times\, 
    \sum_{\omega \pmod{d_1gbb'}}
    e_{d_1gbb'}(\omega(b't-bt'))
    \sum_{h \equiv -\omega^2 \pmod{d_1gbb'}} \psi_0\left(\frac{h}{H}\right)
    \bigg\vert.
\end{aligned}
\end{equation}

\textbf{Poisson summation and the trivial frequency:}
Finally, we apply Lemma \ref{le:poisson} to the inner sum from \eqref{eq:before-poisson-h} and obtain
\[
\begin{aligned}
    \sum_{h \equiv -\omega^2 \pmod{d_1gbb'}}
    \psi_0\left(\frac{h}{H}\right)
    =
    \frac{H}{d_1gbb'}
    \sum_{|\ell| \le X^\eps \frac{D_1D_2^2}{gH}} \widehat\psi_0\left(\frac{\ell H}{d_1gbb'}\right) e_{d_1gbb'}(-\ell\omega^2)
    + 
    O_\eps(X^{-100}).
\end{aligned}
\]
The terms with $\ell \neq 0$ lead to \eqref{eq:E2} after swapping sums and using the triangle inequality in \eqref{eq:before-poisson-h}, so it remains to bound the contribution of $\ell = 0$ to \eqref{eq:before-poisson-h}. After evaluating the complete linear exponential sum over $\omega$, the latter contribution is
\[
\begin{aligned}
    \precprec
    \frac{XH^2}{D_1 D_2^2}
    \sum_{g \le 2D_2} \sum_{\substack{b, b' \sim D_2/g \\ (b, b') = 1}}
    \sum_{d_1 \asymp D_1} (d_1, g)
    \bigg\vert 
    \sum_{\substack{t, t' \neq 0 \\ b't \equiv bt' \pmod{d_1gbb'}}}
    \widehat\psi_f\left(\frac{tX^{1/2}}{d_1gb}\right) 
    \overline{\widehat\psi_f}\left(\frac{t'X^{1/2}}{d_1gb'}\right)
    \bigg\vert.
\end{aligned}
\]
The congruence $b't \equiv bt' \pmod{d_1gbb'}$ implies $t = bu$ and $t' = b'u'$ where $u \equiv u' \pmod{d_1g}$. But for small $\eps' > 0$, we may truncate the sum over $u, u'$ at $|u|, |u'| \le d_1gX^{\eps'-1/2}$ up to an error of $O_\delta(X^{-100})$, due to the Schwartz decay of $\widehat\psi_f$. The congruence then forces $u = u'$, and we are left with a contribution of
\[
\begin{aligned}
    \precprec
    \frac{XH^2}{D_1 D_2^2}
    \sum_{g \le 2D_2} \sum_{\substack{b, b' \sim D_2/g \\ (b, b') = 1}}
    \sum_{d_1 \asymp D_1} (d_1, g)
    \sum_{1 \le |u| \le d_1gX^{\eps'-1/2}} 
    1
    &\ll 
    \frac{XH^2}{D_1 D_2^2} \sum_{g \le 2D_2} \frac{D_2^2}{g^2} D_1^2g X^{\eps'-1/2}
    \\
    &\ll_{\eps'} 
    H^2 X^{1/2+2\eps'} D_1.
\end{aligned}
\]
This completes the proof of the lemma.
\end{proof}

\begin{remark}
It is crucial for the last step above that the main-term contribution of the trivial frequency $t = 0$ was excluded from $\mE_1$ in Lemma \ref{lem:typeIfact-poisson}: otherwise, the term $u = 0$ would have contributed roughly $\sqrt{XH^2}$ to the bound for $|\mE_1|$, which is not an acceptable error in Proposition \ref{prop:typeIfact}. The condition $t, t' \neq 0$ will \emph{not} be critical in our computations henceforth.
\end{remark}

\subsection{The quadratic reciprocity step}
The innermost sum in \eqref{eq:E2} can be recognized as a shifted quadratic Gauss sum, so let us quickly introduce some notation.
Throughout this section, for any integers $r_1, r_2, s$ with $r_2 \neq 0$, we denote
\begin{equation} \label{eq:shifted-gausssumdef}
    G(r_1, r_2, s) := 
    \sum_{u \pmod{|r_2|}} e_{r_2}(r_1u^2 + su).
\end{equation}
In the notation of \eqref{eq:gausssumdef}, we have $G(r_1, r_2) = G(r_1, r_2, 0)$. The following general bound holds.
\begin{lemma}[Square-root cancellation for Gauss sums] \label{lem:quadr-gauss-bound} Let $r_1, r_2, s \in \Z$ with $r_2 \neq 0$. Then we have
\[
    |G(r_1, r_2, s)| 
    \le 
    \sqrt{2|r_2|(r_1,r_2,s)}.
\]
Moreover, we have $G(r_1,r_2,s) = 0$ unless $(r_1, r_2) \mid s$.
\end{lemma}
\begin{proof}
For the vanishing assertion, let $d = (r_1, r_2)$ and change variables $u \gets u+(r_2/d)$; the quadratic term does not change, so $G(r_1, r_2, s) = e_d(s)\, G(r_1, r_2, s)$.

The bound $|G(r_1,r_2,s)|^2 \le 2|r_2|(r_1,r_2)$ follows by expanding the square, changing variables $u' \gets u'+u$, and evaluating the linear sum over $u$. If the sum does not vanish, then we also have $(r_1,r_2) = (r_1,r_2,s)$ by the previous paragraph, which proves the general bound.
\end{proof}

We will also need a quadratic reciprocity law for shifted Gauss sums, without assuming any coprimality or parity constraints on $r_1, r_2$. Lemma \ref{lem:quadr-gauss-recipr} below can in fact be used to deduce the more standard quadratic reciprocity law for the Jacobi symbol.
\begin{lemma}[Quadratic Gauss reciprocity] \label{lem:quadr-gauss-recipr}
Let $r_1, r_2, s \in \Z$ with $r_1r_2 \neq 0$. Then we have
\[
    G(r_1, r_2, s) = 
    \sqrt{\frac{|r_2|}{8|r_1|}}\,
    e\left(\frac{\sgn(r_1r_2)}{8} - \frac{s^2}{4r_1r_2}\right) G(-r_2, 4r_1, -2s).
\]
\end{lemma}
\begin{proof}
This is a fairly classical result stemming from Poisson summation for Gaussians, but we include a proof for the sake of completeness. All complex square roots below are principal. 

For any $\eps>0$, we consider the sum
\begin{equation} \label{eq:Seps}
S_\eps
:=
\sum_{n\in\mathbb Z}
e_{r_2}(r_1n^2+sn)e^{-\pi\eps n^2},
\end{equation}
and we aim to compute the limit $\lim_{\eps \to 0} \sqrt{\eps}\, S_\eps$ in two ways. We first record the following fact: for any positive integer $q$, any $q$-periodic function $h:\mathbb Z\to\mathbb C$, and any $\beta\in\mathbb R$,
\begin{equation}
\label{eq:gaussian-average}
\sqrt{t}\sum_{n\in\mathbb Z}h(n)e^{-\pi t(n+\beta)^2}
\longrightarrow
\frac{1}{q}\sum_{u \pmod{q}}h(u)
\qquad
\text{as }\Re(1/t)\to \infty,\quad \Re t>0.
\end{equation}
Indeed, for each $u \pmod{q}$, Gaussian Poisson summation gives
\begin{equation} \label{eq:gauss-poisson}
\sqrt{t}
\sum_{\substack{n\in\mathbb Z\\ n\equiv u \pmod{q}}}
e^{-\pi t(n+\beta)^2}
=
\frac{1}{q}\sum_{m\in\mathbb Z}
e^{-\pi t^{-1} \frac{m^2}{q^2} + 2\pi i (u+\beta) \frac{m}{q}}.
\end{equation}
The $m=0$ term is $1/q$, while the sum of the absolute values of the other terms tends to zero. Summing over $u \pmod{q}$ proves \eqref{eq:gaussian-average}.

On the one hand, the oscillatory factor in \eqref{eq:Seps} is $|r_2|$-periodic, so \eqref{eq:gaussian-average} directly gives
\begin{equation}
\label{eq:Seps-direct-limit}
\lim_{\eps\to0}\sqrt{\eps}\,S_\eps
=
\frac{G(r_1,r_2,s)}{|r_2|}.
\end{equation}

On the other hand, setting
\[
    \alpha := \frac{r_1}{r_2},
    \qquad 
    \beta := \frac{s}{r_2},
\]
we can apply \eqref{eq:gauss-poisson} with $q = 1$ and $u = 0$ to obtain
\[
S_\eps
=
\sum_{m \in \Z} e^{-\pi (\eps - 2i\alpha) m^2 + 2\pi i \beta m}
=
\left(\eps-2i\alpha\right)^{-1/2}
\sum_{n\in\mathbb Z}
e^{-\pi(\eps-2i\alpha)^{-1}(n+\beta)^2}.
\]
Writing
\[
\frac{1}{\eps-2i\alpha}
=
\frac{i}{2\alpha} + t_\eps
,
\qquad
t_\eps
:=
\frac{\eps}{2i\alpha(\eps-2i\alpha)},
\]
the dual sum becomes
\begin{equation}
\label{eq:dual-sum}
S_\eps
=
\left(\eps-2i\alpha\right)^{-1/2}
e\left(-\frac{s^2}{4r_1r_2}\right)
\sum_{n\in\mathbb Z}
e_{4r_1}(-r_2n^2-2sn)\,
e^{-\pi t_\eps(n+\beta)^2}.
\end{equation}
Here the oscillatory factor is $4|r_1|$-periodic. Moreover, as $\eps \to 0$, we have
\[
\Re(1/t_\eps)
=
\frac{4\alpha^2}{\eps} \to \infty,
\quad 
\Re t_\eps>0,
\]
and
\begin{equation}
\label{eq:prefactor-limit}
\frac{\sqrt{\eps}}{\sqrt{t_\eps}}
\to \left|2\alpha\right|,
\qquad\quad
\left(\eps-2i\alpha\right)^{-1/2}
\longrightarrow
|2\alpha|^{-1/2}
e\left(\frac{\sgn(\alpha)}{8}\right).
\end{equation}
Applying \eqref{eq:gaussian-average} and \eqref{eq:prefactor-limit}
to \eqref{eq:dual-sum}, we conclude that
\begin{equation}
\label{eq:Seps-dual-limit}
\lim_{\eps\to0}\sqrt{\eps}\,S_\eps
=
\frac{\sqrt{|2\alpha|}}{4|r_1|}
e\left(
\frac{\sgn(\alpha)}{8}
-\frac{s^2}{4r_1r_2}
\right)
G(-r_2,4r_1,-2s).
\end{equation}
Comparing \eqref{eq:Seps-direct-limit} with \eqref{eq:Seps-dual-limit} and recalling that $\alpha = \tfrac{r_1}{r_2}$ proves the desired identity.
\end{proof}

The following consequence of Poisson summation with a GCD-weight will also be helpful.

\begin{lemma}[Consequence of Poisson summation] \label{lem:poisson-consequence}
Let $\delta, \eps > 0$ and let $\psi$ be a smooth function supported in $[-10, 10]$ with $\psi^{(J)} \ll_J \delta^{-J}$ for all $J \ge 0$. Let $N > 1$, $g, q \ge 1$ be integers, and $M \ge X^\eps \delta^{-1} gq / N$. Then for any $C > 0$ and $a \in \Z$, one has
\[
    \sum_n (n, g)\, \psi\left(\frac{n}{N}\right) e_q(na)
    \ll 
    g^{o(1)} N \sum_{\substack{|m| \le M \\ m \equiv ga \pmod{q}}} 1
    +
    O_{\eps,C}(X^{-C}gqN),
\]
\end{lemma}
\begin{proof}
We expand $(n, g)$ in terms of the Euler totient function and apply Poisson summation (as in Lemma \ref{le:poisson} but without truncating yet) to rewrite the left-hand side as
\[
\begin{aligned}
    \sum_{r \mid g} \phi(r) \sum_{n'} \psi\left(\frac{n'r}{N}\right) e_q(n'ra)
    &=
    \sum_{r \mid g} \phi(r) \frac{N}{r} \sum_{m \equiv ra \pmod{q}} \widehat\psi\left(\frac{-mN}{rq}\right)
    \\
    &\ll 
    N
    \sum_{r \mid g}
    \sum_{m \equiv ra \pmod{q}}
    \bigg\vert \widehat\psi\left(\frac{-mN}{rq}\right) \bigg\vert.
\end{aligned}
\]
We then extend the sum over $m \in \Z$ to a sum over $m \in (r/g)\Z$, and change variables $m \gets (r/g)m$ to upper bound the sum above by
\[
    N \sum_{r \mid g} \sum_{m \equiv ga \pmod{q}} \bigg\vert \widehat\psi\left(\frac{-mN}{gq}\right) \bigg\vert
    \ll 
    g^{o(1)}
    N \sum_{m \equiv ga \pmod{q}} \bigg\vert \widehat\psi\left(\frac{-mN}{gq}\right) \bigg\vert.
\]
The last sum can be truncated at $|m| \le M$ similarly as in Lemma \ref{le:poisson} (by the Schwartz decay of $\widehat\psi$), up to an error of $O_{\eps,C}(X^{-C}gqN)$.
\end{proof}

We now return to bounding the exponential sum $\mE_2$ from \eqref{eq:E2}. 

\begin{lemma} \label{lem:typeIfact-reciprocity}
For $\mE_2$ as in \eqref{eq:E2}, one has
\[
     \mE_2 \precprec_\eps \frac{H^2 D_1^{5/2} D_2^3}{X} + |\mE_3|,
\]
\begin{equation} \label{eq:E3}
\begin{aligned}
    \mE_3 := \frac{XH^2}{D_1^{1/2} D_2^3}
    \sum_{g \le 2D_2} g^{1/2}
    \sum_{1 \le |\ell| \le X^\eps \frac{D_1D_2^2}{gH}}
    \frac{1}{\sqrt{|\ell|}}
    \sum_{\substack{b, b' \sim D_2/g \\ (b, b') = 1}}
    \sum_{u \pmod{|\ell|}}
    \sum_{\substack{|m| \le X^{2\eps}g|\ell|/D_1 \\ m \equiv g^2bb'u^2 \pmod{|\ell|}}}
    \\ 
    \times\,
    \bigg( 
    1 + \frac{D_1D_2}{X^{1/2}}
    \sum_{\substack{\frac{X^{1/2}|\ell|}{4D_1D_2} \le |n| \le 4\frac{X^{1/2}|\ell|}{D_1D_2} \\ n \equiv 2b'u \pmod{|\ell|}}} 1
    \bigg)
    \bigg( 
    1 + \frac{D_1D_2}{X^{1/2}}
    \sum_{\substack{\frac{X^{1/2}|\ell|}{4D_1D_2} \le |n'| \le 4\frac{X^{1/2}|\ell|}{D_1D_2} \\ n' \equiv 2bu \pmod{|\ell|}}} 1
    \bigg).
\end{aligned}
\end{equation}

\end{lemma}

\begin{remark}
The term $H^2 D_1^{5/2} D_2^3 X^{-1}$ above comes from crudely eliminating an extraneous phase from the quadratic reciprocity step, but it seems difficult to remove even with a more sophisticated argument. It gives an acceptable error in Proposition \ref{prop:typeIfact} as long as (accounting for the square in Lemma \ref{lem:typeIfact-CS})
\begin{equation}\label{eq:typeIfact-bezout-terms}
    \frac{H^2 D_1^{5/2} D_2^3}{X} \precprec (HX^{1/2}R^{-1/2})^2
    \qquad \iff \qquad 
    D_1^{5/2} D_2^3 \precprec X^2 R^{-1}.
\end{equation}
\end{remark}

\begin{proof}[Proof of Lemma \ref{lem:typeIfact-reciprocity}]
By \eqref{eq:shifted-gausssumdef} and Lemma \ref{lem:quadr-gauss-recipr}, the innermost sum over $\omega$ in \eqref{eq:E2} is
\[
    G(-\ell, d_1gbb', b't-bt')
    =
    \sqrt{\frac{d_1gbb'}{8|\ell|}}
    e\left(-\frac{\sgn \ell}{8} + \frac{(b't-bt')^2}{4d_1gbb'\ell}\right)
    G(d_1gbb', 4\ell, 2(b't-bt')).
\]
By plugging this into \eqref{eq:E2}, using the triangle inequality in $\ell$, renaming $4\ell$ as $\ell$, replacing $\psi_0(x)$ by $4\psi_0(4x)$, and extending the sum over $\ell$ after harmlessly enlarging $\eps$, we arrive at
\[
\begin{aligned}
    \mE_2 
    \ll \frac{XH^2}{D_1}
    \sum_{g \le 2D_2} \frac{1}{g^{5/2}}
    \sum_{1 \le |\ell| \le X^\eps \frac{D_1D_2^2}{gH}}
    \frac{1}{\sqrt{|\ell|}}
    \bigg\vert
    \sum_{\substack{b, b' \sim D_2/g \\ (b, b') = 1}}
    \frac{\beta_{gb}\bar\beta_{gb'}}{(bb')^{3/2}}
    \sum_{d_1}
    \frac{(d_1, g)}{\sqrt{d_1}}
    \psi_1\left(\frac{d_1}{D_1}\right)
    \widehat\psi_0\left(\frac{\ell H}{d_1gbb'}\right)
    \\ 
    \times\,
    \sum_{t, t' \neq 0}
    \widehat\psi_f\left(\frac{tX^{1/2}}{d_1gb}\right) 
    \overline{\widehat\psi_f}\left(\frac{t'X^{1/2}}{d_1gb'}\right) 
    e\left(\frac{(b't-bt')^2}{d_1gbb'\ell}\right) G(d_1gbb', \ell, 2(b't-bt'))
    \bigg\vert.
\end{aligned}
\]

\textbf{Removing the extraneous phase from reciprocity:} 
We write
\[
    e\left(\frac{(b't-bt')^2}{d_1gbb'\ell}\right)
    =
    1 + O\left(\frac{(|t|+|t'|)^2}{D_1g|\ell|}\right)
    \qquad 
    \rightsquigarrow 
    \qquad 
    |\mE_2| \le \mE_2' + \mE_2'',
\]
where $\mE_2'$ contains the contribution of the $1$-term after the triangle inequality, and $\mE_2''$ contains the error term. We first bound $\mE_2''$ using Lemma \ref{lem:quadr-gauss-bound} and the Schwartz decay of $\widehat\psi_f$:
\[
\begin{aligned}
    \mE_2'' 
    \precprec
    \frac{XH^2}{D_1}
    \sum_{g \le 2D_2} \frac{1}{g^{5/2}} 
    \sum_{1 \le |\ell| \le X^\eps \frac{D_1D_2^2}{gH}}
    \frac{1}{\sqrt{|\ell|}}
    \sum_{b, b' \sim D_2/g} \frac{1}{(bb')^{3/2}}
    \sum_{d_1 \asymp D_1} \frac{(d_1, g)}{\sqrt{D_1}}
    \\
    \times
    \left(\frac{D_1D_2}{X^{1/2}}\right)^4 \frac{1}{D_1g|\ell|}
    \sqrt{|\ell|(\ell,d_1gbb')}.
\end{aligned}
\]
The sum over $\ell$ of $|\ell|^{-1/2-1+1/2}\sqrt{(\ell, d_1gbb')}$ is $X^{o(1)}$. After bounding $\sum_{d_1 \asymp D_1} (d_1, g) \precprec D_1$ and $\sum_{b, b' \sim D_2/g} (bb')^{-3/2} \ll g/D_2$, the sum over $g$ of $g^{-5/2+1-1}$ is $O(1)$. It follows that
\[
    \mE_2'' 
    \precprec 
    \frac{XH^2}{D_1D_2}
    \sqrt{D_1} \frac{D_1^4D_2^4}{X^2} \frac{1}{D_1} 
    \ll 
    \frac{H^2 D_1^{5/2} D_2^3}{X},
\]
which is an acceptable contribution for our lemma. 

\textbf{Poisson summation three more times:} It remains to develop the sum $\mE_2'$. After expanding the Gauss sum and swapping sums, we have
\[
\begin{aligned}
    \mE_2' 
    \ll \frac{XH^2}{D_1}
    \sum_{g \le 2D_2} \frac{1}{g^{5/2}}
    \sum_{1 \le |\ell| \le X^\eps \frac{D_1D_2^2}{gH}}
    \frac{1}{\sqrt{|\ell|}}
    \bigg\vert
    \sum_{\substack{b, b' \sim D_2/g \\ (b, b') = 1}}
    \frac{\beta_{gb}\bar\beta_{gb'}}{(bb')^{3/2}}
    \sum_{u \pmod{|\ell|}}
    \sum_{d_1}
    \frac{(d_1, g)}{\sqrt{d_1}} 
    e_\ell(d_1gbb'u^2)
    \\ 
    \times\,
    \psi_1\left(\frac{d_1}{D_1}\right)
    \widehat\psi_0\left(\frac{\ell H}{d_1gbb'}\right)
    \bigg( 
    \sum_{t \neq 0} 
    \widehat\psi_f\left(\frac{tX^{1/2}}{d_1gb}\right) e_\ell(2b'tu)
    \bigg)
    \bigg( 
    \overline{\sum_{t' \neq 0}
    \widehat\psi_f\left(\frac{t'X^{1/2}}{d_1gb'}\right) e_\ell(2bt'u)}
    \bigg)
    \bigg\vert.
\end{aligned}
\]
Now by Poisson summation (as in Lemma \ref{le:poisson} but without truncating the sum over $t$), we have
\[
    \sum_{t \neq 0} 
    \widehat\psi_f\left(\frac{tX^{1/2}}{d_1gb}\right) e_\ell(2b'tu)
    =
    -
    \widehat\psi_f(0)
    +
    \frac{d_1gb}{X^{1/2}}
    \sum_{n \equiv 2b'u \pmod{|\ell|}} \psi_f\left(\frac{n d_1gb}{X^{1/2}\ell}\right)
    .
\]
Importantly, note that $\psi_f(0) = 0$, so the last sum above does not include the term $n = 0$; in fact it is supported on $|n| \asymp X^{1/2}|\ell|/(D_1D_2)$.

We similarly transform the sum over $t'$ into a sum over $n' \equiv 2bu \pmod{|\ell|}$.
After expanding the resulting product of the two transformed sums, we bring the sum over $d_1$ inside each of the four terms. In each term, after multiplying by suitable powers of $D_1$, the weight of $d_1$ has the shape $(d_1, g)\, \psi(d_1/D_1)\, e_\ell(d_1gbb'u^2)$ for some smooth function $\psi$ which may depend on $g, \ell, b, b', n, n'$, but which satisfies the assumption of Lemma \ref{lem:poisson-consequence} with $\delta^{-1} = X^\eps$; this is achieved by writing, e.g.,
\[
    \widehat\psi_0\left(\frac{\ell H}{d_1gbb'}\right) = 
    \widehat\psi_0\left(\frac{\ell g H}{D_1D_2^2} \cdot \frac{D_2^2}{g^2bb'} \cdot \frac{D_1}{d_1}\right),
    \qquad 
    \psi_f\left(\frac{nd_1gb}{X^{1/2}\ell}\right)
    =
    \psi_f\left(\frac{nD_1D_2}{X^{1/2}|\ell|} \cdot \frac{gb}{D_2} \cdot \frac{d_1}{D_1}\right).
\]
Then by applying Lemma \ref{lem:poisson-consequence} to each term, using the triangle inequality, and grouping back terms, we reach
\[
\begin{aligned}
    \mE_2' 
    \precprec_\eps X^{-100} + \frac{XH^2}{D_1}
    \sum_{g \le 2D_2} \frac{1}{g^{5/2}}
    \sum_{1 \le |\ell| \le X^\eps \frac{D_1D_2^2}{gH}}
    \frac{1}{\sqrt{|\ell|}}
    \sum_{\substack{b, b' \sim D_2/g \\ (b, b') = 1}}
    \frac{g^3}{D_2^3}
    \sum_{u \pmod{|\ell|}}
    \frac{1}{\sqrt{D_1}}
    \sum_{\substack{|m| \le X^{2\eps}g|\ell|/D_1 \\ m \equiv g^2bb'u^2 \pmod{|\ell|}}}
    \\ 
    \times\,
    D_1
    \bigg( 
    1 + \frac{D_1D_2}{X^{1/2}}
    \sum_{\substack{\frac{X^{1/2}|\ell|}{4D_1D_2} \le |n| \le 4\frac{X^{1/2}|\ell|}{D_1D_2} \\ n \equiv 2b'u \pmod{|\ell|}}} 1
    \bigg)
    \bigg( 
    1 + \frac{D_1D_2}{X^{1/2}}
    \sum_{\substack{\frac{X^{1/2}|\ell|}{4D_1D_2} \le |n'| \le 4\frac{X^{1/2}|\ell|}{D_1D_2} \\ n' \equiv 2bu \pmod{|\ell|}}} 1
    \bigg).
\end{aligned}
\]
Rearranging this yields the sum from \eqref{eq:E3}.
\end{proof}

\subsection{The final counting problem}
To complete our proof, it remains to bound the sum $\mE_3$ from \eqref{eq:E3} by a counting argument. The following lemmas will be helpful.

\begin{lemma}[Primitive points on lattices] \label{lem:prim-points-count}
Let $n, n', q$ be integers, $q \ge 1$, and $A, A' > 0$. Then
\[
    \mathop{\sum_{a \sim A} \sum_{a' \sim A'}}_{(a, a') = 1} \one_{an \equiv a'n' \pmod{q}}
    \ll 
    1
    +
    (n, n', q)
    \frac{AA'}{q}.
\]
\end{lemma}
\begin{proof}
Dividing $n, n', q$ by $(n, n', q)$, we may assume without loss of generality that $(n, n', q) = 1$. Order all the points $(a; a')$ with $a \sim A$, $a' \sim A'$, $(a, a') = 1$, and $an \equiv a'n'$ in increasing order of the slopes $a'/a$; let $\{(a_j; a_j')\}_{j=1}^M$ be the resulting sequence, where $M$ is the total number of points. For each $j \in \{1, \ldots, M-1\}$, note that
\[
\begin{cases}
    a_j n \equiv a'_j n' \pmod{q} \\ 
    a_{j+1} n \equiv a'_{j+1} n' \pmod{q}
\end{cases}
\quad 
\Rightarrow 
\quad 
\begin{cases}
    a_j a'_{j+1} n \equiv a'_j a'_{j+1} n' \equiv a'_j a_{j+1} n\pmod{q} \\ 
    a_j a'_{j+1} n' \equiv a_j a_{j+1} n \equiv a_j' a_{j+1} n' \pmod{q}
\end{cases}
\]
and thus $q \mid (a_j a_{j+1}' - a_j' a_{j+1})(n, n')$. Since $(n, n', q) = 1$ and $a_j a_{j+1}' > a_j' a_{j+1}$, we find that
\[
    a_j a_{j+1}' - a'_j a_{j+1} \ge q
    \qquad 
    \Rightarrow 
    \qquad 
    \frac{a'_{j+1}}{a_{j+1}} - \frac{a'_j}{a_j} \ge \frac{q}{a_j a_{j+1}} 
    \gg 
    \frac{q}{A^2}.
\]
The sum of all such slope differences telescopes to
\[
    \frac{a'_M}{a_M} - \frac{a'_1}{a_1}
    \gg 
    (M-1) \frac{q}{A^2}.
\]
The left-hand side is $O(A'/A)$, so we can conclude that $M \ll 1 + \tfrac{AA'}{q}$.
\end{proof}

\begin{lemma}[Number of roots and square parts] \label{lem:root-bound}
For any positive integer $d$, denote by $d_\square = \prod_{\text{prime }p \mid d} p^{2\lfloor v_p(d)/2\rfloor}$ the largest square divisor of $d$.
Then for any integer $h$, one has
\begin{equation} \label{eq:root-bound}
    \varrho_h(d) \ll d^{o(1)} (h, d_\square)^{1/2},
\end{equation}
where $d_\square$ is the largest square divisor of $d$. Moreover, for any $D \ge 1$, one has
\begin{equation} \label{eq:sqf-bound}
    \sum_{d \sim D} d_\square^{1/2} \ll D^{1 + o(1)}.
\end{equation}
\end{lemma}

\begin{proof}
The first part follows from elementary local considerations at prime powers. For the second part, write $d=ab^2$ with $a$ square-free, so that $d_\square^{1/2}=b$. Then
\[
    \sum_{d\sim D}d_\square^{1/2}
    \leq \sum_{b\leq\sqrt{2D}} b\,\frac{2D}{b^2}
    \ll D\log(2D).
\]
\end{proof}

\begin{lemma} \label{lem:typeIfact-counting}
For $\mE_3$ as in \eqref{eq:E3}, assuming that $\eps < 10^{-3}$ and $D_2 \le X^{0.49}$, one has
\[
    \mE_3 \precprec X^{1+5\eps} \left(H^{3/2} + H^{1/2}D_2^2 + H^{1/2}D_1
    +
    \frac{D_1D_2^2}{H^{1/2}} \right).
\]
\end{lemma}

\begin{remark}
The bound above is acceptable in Proposition \ref{prop:typeIfact} as long as (accounting for the square in Lemma \ref{lem:typeIfact-CS} and taking $\eps \to 0$)
\begin{equation} \label{eq:typeIfact-final-bound}
    X \left(H^{3/2} + H^{1/2}D_2^2 + H^{1/2}D_1
    +
    \frac{D_1D_2^2}{H^{1/2}} \right) \precprec H^2XR^{-1}
    \quad 
    \iff 
    \quad 
    \begin{cases}
    R \precprec H^{1/2}, \\ 
    D_2 \precprec H^{3/4} R^{-1/2}, \\ 
    D_1 \precprec H^{3/2} R^{-1}, \\ 
    D_1 D_2^2 \precprec H^{5/2} R^{-1}.
    \end{cases}
\end{equation}
\end{remark}

\begin{proof}[Proof of Lemma \ref{lem:typeIfact-counting}]
We put $|\ell|$ in dyadic ranges, then expand the product of the last two sums in \eqref{eq:E3} and accordingly split
\[
    \mE_3 \precprec |\mA_3| + |\mB_3| + |\mC_3| + |\mD_3|,
\]
where
\[
    \mA_3 := \frac{XH^2}{D_1^{1/2} D_2^3}
    \sum_{g \le 2D_2} g^{1/2}
    \sup_{\frac{1}{2} \le L \le X^\eps \frac{D_1D_2^2}{gH}} \frac{1}{\sqrt{L}}
    \sum_{|\ell| \sim L}
    \sum_{\substack{b, b' \sim D_2/g \\ (b, b') = 1}}
    \sum_{u \pmod{|\ell|}}
    \sum_{\substack{|m| \le X^{2\eps}gL/D_1 \\ m \equiv g^2bb'u^2 \pmod{|\ell|}}} 1,
\]
\[
    \mB_3 := 
    \frac{XH^2}{D_1^{1/2} D_2^3}
    \sum_{g \le 2D_2} g^{1/2}
    \sup_{\frac{1}{2} \le L \le X^\eps \frac{D_1D_2^2}{gH}} \frac{1}{\sqrt{L}}
    \sum_{|\ell| \sim L}
    \sum_{\substack{b, b' \sim D_2/g \\ (b, b') = 1}}
    \sum_{u \pmod{|\ell|}}
    \sum_{\substack{|m| \le X^{2\eps}gL/D_1 \\ m \equiv g^2bb'u^2 \pmod{|\ell|} \\ \frac{X^{1/2}L}{8D_1D_2} \le |n| \le 8\frac{X^{1/2}L}{D_1D_2} \\ n \equiv 2b'u \pmod{|\ell|}}}
    \frac{D_1 D_2}{\sqrt{X}},
\]
\[
    \mC_3 := 
    \frac{XH^2}{D_1^{1/2} D_2^3}
    \sum_{g \le 2D_2} g^{1/2}
    \sup_{\frac{1}{2} \le L \le X^\eps \frac{D_1D_2^2}{gH}} \frac{1}{\sqrt{L}}
    \sum_{|\ell| \sim L}
    \sum_{\substack{b, b' \sim D_2/g \\ (b, b') = 1}}
    \sum_{u \pmod{|\ell|}}
    \sum_{\substack{|m| \le X^{2\eps}gL/D_1 \\ m \equiv g^2bb'u^2 \pmod{|\ell|} \\ \frac{X^{1/2}L}{8D_1D_2} \le |n'| \le 8\frac{X^{1/2}L}{D_1D_2} \\ n' \equiv 2bu \pmod{|\ell|}}}
    \frac{D_1 D_2}{\sqrt{X}},
\]
\[
    \mD_3 := 
    \frac{XH^2}{D_1^{1/2} D_2^3}
    \sum_{g \le 2D_2} g^{1/2}
    \sup_{\frac{1}{2} \le L \le X^\eps \frac{D_1D_2^2}{gH}} \frac{1}{\sqrt{L}}
    \sum_{|\ell| \sim L}
    \sum_{\substack{b, b' \sim D_2/g \\ (b, b') = 1}}
    \sum_{u \pmod{|\ell|}}
    \sum_{\substack{|m| \le X^{2\eps}gL/D_1 \\ m \equiv g^2bb'u^2 \pmod{|\ell|} \\ \frac{X^{1/2}L}{8D_1D_2} \le |n| \le 8\frac{X^{1/2}L}{D_1D_2} \\ n \equiv 2b'u \pmod{|\ell|} \\ \frac{X^{1/2}L}{8D_1D_2} \le |n'| \le 8\frac{X^{1/2}L}{D_1D_2} \\ n' \equiv 2bu \pmod{|\ell|}}}
    \frac{D_1^2 D_2^2}{X}.
\]

\textbf{Bounding $\mA_3$:} We move the sum over $u$ inside and bound, using \eqref{eq:root-bound},
\[
    \sum_{|m| \le X^{2\eps}gL/D_1} 
    \sum_{\substack{u \pmod{|\ell|} \\ g^2bb'u^2 \equiv m \pmod{|\ell|}}} 1 
    \precprec
    \sum_{|m| \le X^{2\eps}gL/D_1} 
    (g^2bb', \ell) 
    (m, |\ell/(g^2bb',\ell)|_\square)^{1/2}.
\]
The term $m = 0$ contributes $(g^2bb',\ell) |\ell/(g^2bb',\ell)|_\square^{1/2}$, which is absorbed by the dyadic sum over $\ell$ due to \eqref{eq:sqf-bound}. The terms with $m \neq 0$ contribute $(g^2bb',\ell) X^{2\eps+o(1)} gL/D_1$, and the common divisor $(g^2bb',\ell)$ is once again absorbed by the dyadic sum over $\ell$. It follows that
\[
\begin{aligned}
    \sup_{\frac{1}{2} \le L \le X^\eps \frac{D_1D_2^2}{gH}} \frac{1}{\sqrt{L}}
    \sum_{\ell \sim L}
    \sum_{|m| \le X^{2\eps}g|\ell|/D_1} 
    \sum_{\substack{u \pmod{|\ell|} \\ g^2bb'u^2 \equiv m \pmod{|\ell|}}} 1 
    &\precprec
    \sup_{\frac{1}{2} \le L \le X^\eps \frac{D_1D_2^2}{gH}} \sqrt{L} \left(1 + \frac{X^{2\eps} gL}{D_1}\right)
    \\
    &\precprec
    X^{4\eps} \left(\frac{D_1^{1/2}D_2}{g^{1/2}H^{1/2}} + \frac{D_1^{1/2} D_2^3}{g^{1/2} H^{3/2}}\right).
\end{aligned}
\]
Summing over $b, b' \sim D_2/g$, the $g$-dependency is $g^{-5/2}$, and then we bound $\sum_{g \le 2D_2} g^{-2} \ll 1$. This gives
\[
    \mA_3 
    \precprec 
    \frac{X^{1+4\eps}H^2}{D_1^{1/2}D_2^3} D_2^2 \left(\frac{D_1^{1/2}D_2}{H^{1/2}} + \frac{D_1^{1/2} D_2^3}{H^{3/2}}\right)
    \precprec 
    X^{1+4\eps}H^{3/2} + X^{1+4\eps}H^{1/2}D_2^2,
\]
which is acceptable for our final bound.

\textbf{Bounding $\mB_3$ and $\mC_3$:} We only show the treatment of $\mB_3$, since the treatment of $\mC_3$ is completely analogous. From the two congruences modulo $|\ell|$, we deduce that $g^2bn^2 \equiv 4b'm \pmod{|\ell|}$. The sum over $\ell, u, m, n$ inside $\mB_3$ is therefore at most
\[
    \sum_{|\ell| \sim L}
    \sum_{\substack{|m| \le X^{2\eps}gL/D_1 \\ \frac{X^{1/2}L}{8D_1D_2} \le |n| \le 8\frac{X^{1/2}L}{D_1D_2} \\ g^2bn^2 \equiv 4b'm \pmod{|\ell|}}}
    \sum_{\substack{u \pmod{|\ell|} \\ 2b'u \equiv n \pmod{|\ell|}}}
    \frac{D_1 D_2}{\sqrt{X}}
    \ll 
    \frac{D_1 D_2}{\sqrt{X}}
    \sum_{\substack{|m| \le X^{2\eps}gL/D_1 \\ \frac{X^{1/2}L}{8D_1D_2} \le |n| \le 8\frac{X^{1/2}L}{D_1D_2}}}
    \sum_{\substack{|\ell| \sim L \\ \ell \mid g^2bn^2 - 4b'm}}
    (\ell, b').
\]
Note that since $D_2 \le X^{0.49}$, we have
\[
    |g^2bn^2| \ge |g^2bn| \gg g D_2 \frac{X^{1/2}L}{D_1D_2} \gg X^{0.001+2\eps} D_2 \frac{gL}{D_1} \gg X^{0.001} |4b'm|,
\]
which forces $g^2bn^2 \neq 4b'm$ if $X$ is large enough (as we may assume).
We then bound $(\ell, b') \le (g^2bn^2, b')$, and absorb the sum over $\ell$ using the divisor bound for $g^2bn^2 - 4b'm$. The sum over $b'$ then absorbs the common divisor $(g^2bn^2, b')$, and we sum over $b, m, n$ trivially to obtain
\begin{equation} \label{eq:B3-bound}
\begin{aligned}
    \mB_3 
    &\precprec 
    \frac{XH^2}{D_1^{1/2}D_2^3}
    \sum_{g \le 2D_2} g^{1/2}
    \sup_{\frac{1}{2} \le L \le X^\eps \frac{D_1D_2^2}{gH}} \frac{1}{\sqrt{L}}
    \frac{D_2^2}{g^2}
    \frac{D_1D_2}{X^{1/2}}
    \left(1 + \frac{X^{2\eps}gL}{D_1}\right) \frac{X^{1/2}L}{D_1D_2}
    \\
    &\precprec
    \frac{XH^2}{D_1^{1/2}D_2^3}
    \sum_{g \le 2D_2} g^{1/2}
    \frac{D_2^2}{g^2}
    \frac{D_1D_2}{X^{1/2}}
    \frac{X^{1/2}}{D_1D_2}
    \cdot 
    X^{4\eps} \left(\frac{D_1^{1/2}D_2}{g^{1/2}H^{1/2}} + \frac{D_1^{1/2} D_2^3}{g^{1/2} H^{3/2}}\right)
    \\
    &\precprec 
    \frac{X^{1+4\eps}H^2}{D_1^{1/2}D_2^3}
    D_2^2
    \left(\frac{D_1^{1/2}D_2}{H^{1/2}} + \frac{D_1^{1/2} D_2^3}{H^{3/2}}\right),
\end{aligned}
\end{equation}
which is again acceptable.

\textbf{Bounding $\mD_3$:}
From the congruences 
\[
    \begin{cases}
        m \equiv g^2bb' u^2 \pmod{|\ell|}, \\ 
        n \equiv 2b'u \pmod{|\ell|}, \\
        n' \equiv 2bu \pmod{|\ell|},
    \end{cases}
\]
we deduce that $4m \equiv g^2nn' \pmod{|\ell|}$, $bn \equiv b'n' \pmod{|\ell|}$, and $2u \equiv c'n + cn' \pmod{|\ell|}$ where $c = c(b,b'), c' = c'(b,b') \in \Z$ are such that $cb + c'b' = 1$ (such $c, c'$ exist since $(b, b') = 1$). The sum over $\ell, b, b', u, m, n, n'$ inside $\mD_3$ is therefore at most
\[
    \sum_{\substack{|m| \le X^{2\eps}gL/D_1 \\ \frac{X^{1/2}L}{8D_1D_2} \le |n| \le 8\frac{X^{1/2}L}{D_1D_2} \\ \frac{X^{1/2}L}{8D_1D_2} \le |n'| \le 8\frac{X^{1/2}L}{D_1D_2}}}
    \sum_{\substack{|\ell| \sim L \\ \ell \mid g^2nn'-4m}}
    \sum_{\substack{b, b' \sim D_2/g \\ (b, b') = 1}} \one_{bn \equiv b'n' \pmod{|\ell|}}
    \sum_{\substack{u \pmod{|\ell|} \\ 2u \equiv c'(b,b')n + c(b,b')n' \pmod{|\ell|}}}
    \frac{D_1^2 D_2^2}{X}.
\]
There are $O(1)$ possibilities for $u$ in the last sum. We then apply Lemma \ref{lem:prim-points-count} to the sum over $b, b'$ to bound the expression above by
\[
    \ll \sum_{\substack{|m| \le X^{2\eps}gL/D_1 \\ \frac{X^{1/2}L}{8D_1D_2} \le |n| \le 8\frac{X^{1/2}L}{D_1D_2} \\ \frac{X^{1/2}L}{8D_1D_2} \le |n'| \le 8\frac{X^{1/2}L}{D_1D_2}}}
    \sum_{\substack{|\ell| \sim L \\ \ell \mid g^2nn'-4m}}
    \left(1 + (n, n')\frac{D_2^2}{g^2L}\right)
    \frac{D_1^2 D_2^2}{X}.
\]
Since $D_2 \le X^{0.49}$, we have
\[
    |g^2nn'| \ge |g^2n| \gg g \frac{X^{1/2}L}{D_1D_2} \gg X^{0.001+2\eps} \frac{gL}{D_1} \gg X^{0.001} |4m|,
\]
which forces $g^2nn' \neq 4m$ when $X$ is large enough (as we may assume). We can then absorb the sum over $\ell$ by the divisor bound for $g^2nn'-4m$, and execute the sum over $m, n, n'$ absorbing the factor $(n, n')$. Plugging this into the expression for $\mD_3$, we obtain
\[
    \mD_3 \precprec 
    \frac{XH^2}{D_1^{1/2}D_2^3} \sum_{g \le 2D_2} g^{1/2} 
    \sup_{\frac{1}{2}\le L \le X^\eps \frac{D_1D_2^2}{gH}} \frac{1}{\sqrt{L}}
    \left(1 + \frac{X^{2\eps}gL}{D_1}\right) \frac{XL^2}{D_1^2D_2^2} \left(1 + \frac{D_2^2}{g^2L}\right) \frac{D_1^2D_2^2}{X}.
\]
We consider separately the contribution of the terms $1$ and $D_2^2/(g^2L)$ from the last parenthesis. The term $D_2^2/(g^2L)$ leads to the exact expression and bound from \eqref{eq:B3-bound}, which is acceptable. On the other hand, the $1$-term contributes
\[
\begin{aligned}
    \mD_{3,1} 
    &\precprec 
    \frac{XH^2}{D_1^{1/2}D_2^3} \sum_{g \le 2D_2} g^{1/2} 
    \sup_{\frac{1}{2}\le L \le X^\eps \frac{D_1D_2^2}{gH}} L^{3/2}
    \left(1 + \frac{X^{2\eps}gL}{D_1}\right)
    \\
    &\precprec 
    \frac{XH^2}{D_1^{1/2}D_2^3} \sum_{g \le 2D_2} g^{1/2} 
    X^{5\eps} \left(\frac{D_1^{3/2}D_2^3}{g^{3/2}H^{3/2}} + \frac{D_1^{3/2}D_2^5}{g^{3/2}H^{5/2}}\right)
    \precprec 
    X^{1+5\eps}H^{1/2}D_1
    +
    \frac{X^{1+5\eps}D_1D_2^2}{H^{1/2}}.
\end{aligned}
\]
These give the last two terms in our final bound.
\end{proof}

We can now complete the proof of Proposition \ref{prop:typeIfact}. By combining Lemmas \ref{lem:typeIfact-poisson}, \ref{lem:typeIfact-CS}, \ref{lem:typeIfact-reciprocity}, and \ref{lem:typeIfact-counting} and taking $\eps \to 0$, we conclude that the desired bound
\[
    \mathop{\sum_{d_1 \sim D_1} \sum_{d_2 \sim D_2}}_{(d_1,d_2)=1}
    \alpha_{d_1} \beta_{d_2}
    \sum_n w_{d_1d_2n}^{(\gamma)}
    \precprec 
    HX^{1/2} R^{-1/2}
\]
holds under the parameter constraints \eqref{eq:typeIfact-diagonals}, \eqref{eq:typeIfact-bezout-terms}, \eqref{eq:typeIfact-final-bound}, and $D_2 \le X^{0.49}$. But these constraints are part of (or easily derived from) the assumptions of Proposition \ref{prop:typeIfact}. In particular, the bounds $D_1 \precprec H^{3/2} R^{-1}$ and $D_1D_2^2 \precprec H^{5/2}R^{-1}$ in \eqref{eq:typeIfact-final-bound} follow from $D_1 \le H X^{-\eta}$ and $D_2 \le H^{3/4} X^{-\eta}$, which are assumed in Proposition \ref{prop:typeIfact}.

\section{Proof of Proposition \ref{prop:typeI2} } \label{sec:typeI2}
We use \cite[Theorem 1.4]{GMquadratic} with $\theta=1/4$, by Selberg's $3/16$ theorem. We need a slight modification. Write $s^2=h_\square$ for the largest square divisor of $h$. Prime-power considerations give $\varrho_h(c)\precprec s$ and, for $ac-b^2=h$, at most $q^{o(1)}s$ projective roots of $au^2+2buv+cv^2\equiv0\pmod q$. Thus the proof of \cite[Proposition 4.1]{GMquadratic} loses at most $s^3$ without its square-free restriction: one factor $s$ in the coset count and two in the quadratic-root counts. The subsequent Cauchy--Schwarz inequality loses at most $s^{3/2}\leq h_\square$, giving the following form.

\begin{theorem}
\label{thm:typeIimport} 
Let $1 \leq  D \leq K \leq Y^{2}$ and $D \leq Y^{1/2}$. Let $1 \leq h \precprec Y^{2}$. Let $\psi_1,\psi_2:\R \to \C$ be smooth functions supported on $[1,2]$ and $[-1,1]$, resp., such that for all $J \geq 0$ we have $\psi_i^{(J)} \precprec_J 1$. Then for $\theta= 1/4$
\begin{align*}
    \sum_{\substack{d \leq D  }}\bigg| \sum_{k \equiv 0 \pmod{d}} \psi_1\Bigl(\frac{k}{K}\Bigr) \Bigl(&\sum_{ \ell^2 +h \equiv 0 \pmod{k}} \psi_2\Bigl(\frac{\ell}{Y}\Bigr)   - \frac{\varrho_{h}(k) }{k} Y \int_\R \psi_2(u) \d u \Bigr) \bigg|\\
   & \precprec \, h_{\square}  D^{1/2}Y^{1/2}(D^{1/2}+h^{1/4}) \Bigl(1+\frac{Y}{D(D+h^{1/2})}\Bigr)^\theta.
\end{align*}
\end{theorem}

\begin{proof}[Proof of Proposition \ref{prop:typeI2}] 
By definition and the triangle inequality
\begin{align*}
    \sum_{d \leq D_2} \mu(d)^2 &\bigg| \sum_{\substack{m\equiv 0 \pmod{d}\\ n}}  \mu(m)^2\psi\bigg(\frac{m}{M}\bigg)   w^{(\gamma)}_{mn} \bigg| \\ =&  \sum_{d \leq D_2} \mu(d)^2 \bigg| \sum_{m\equiv 0 \pmod{d}}  \mu(m)^2\psi\bigg(\frac{m}{M}\bigg)   \sum_{h} \gamma_h \sum_{mn=k+h} (\sigma_1(k)-\sigma_0(k)) \bigg|\\
     \precprec &  \bigg(|\Sigma_1|+|\Sigma_2|\bigg),
\end{align*}
where
\begin{align*}
    \Sigma_1&= \sum_{h\leq H}\sum_{\substack{ d\leq D_2}}  \mu(d)^2 \bigg|   \sum_{\substack{m\equiv 0 \pmod{d}}} \mu(m)^2 \psi\bigg(\frac{m}{M}\bigg)  \bigg(\sum_{\substack{k\geq1\\k^2+h\equiv 0 \pmod{m}}}f\bigg(\frac{k^2}{X}\bigg)- \frac{\varrho_{h}(m) \sqrt{X} I_f}{m}  \bigg)\bigg| \\
    \Sigma_2&=  \sum_{\substack{ d\leq D_2}}  \mu(d)^2 \bigg|   \sum_{\substack{m\equiv 0 \pmod{d}}} \mu(m)^2\psi\bigg(\frac{m}{M}\bigg)\sum_{h}\gamma_h \bigg( \sum_{k+h\equiv 0 \pmod{m}}\sigma_0(k) -\frac{\varrho_{h}(m) \sqrt{X} I_f}{m} \bigg)\bigg| .
\end{align*}

\textbf{Estimation of $\Sigma_1$} Put $D_2' = D_2 X^{\eta/2}= X^{1/4-\eta/2}$. If $M<D_2'$, Poisson summation in $k$ gives an error $O(X^{-100})$, so we may assume $M\geq D_2'$. By expanding $\mu(m)^2 = \sum_{e^2|m} \mu(e)$, the contribution from $e> X^{\eta/4}$ is $\precprec H(\sqrt{X}X^{-\eta/4}+\sqrt{M})$, using $\sum_{h\leq H}\varrho_h(q)\precprec H$ and the divisor bound. It therefore suffices to bound
\begin{align*}
 \sum_{h \leq H}   \sum_{\substack{ d\leq D_2'}}   \bigg|   \sum_{\substack{m\equiv 0 \pmod{d}}} \psi\bigg(\frac{m}{M}\bigg)  \bigg(\sum_{\substack{k\geq1\\k^2+h\equiv 0 \pmod{m}}}f\bigg(\frac{k^2}{X}\bigg)- \frac{\varrho_{h}(m) \sqrt{X} I_f}{m}  \bigg)\bigg|.
\end{align*}
Similarly, there are $\ll HX^{-\eta/32}$ terms where $h_\square > X^{\eta/16}$. For each fixed $h$, the divisor bound and $\sum_{m\sim M}\varrho_h(m)\precprec M$ bound the absolute contribution by $\precprec\sqrt{X}$. Thus these terms contribute $\precprec HX^{1/2-\eta/32}$, and we restrict to $h_\square \leq X^{\eta/16}$. An application of Theorem \ref{thm:typeIimport} with $Y=\sqrt{X}$ and $\psi_2(u)=\mathbf{1}_{u>0}f(u^2)$ now gives us the bound
 \begin{align*}
\sum_{\substack{h \leq H \\ h_\square \leq X^{\eta/16}}} \sum_{\substack{ d\leq D_2' }}   &\bigg|   \sum_{\substack{m\equiv 0 \pmod{d}}} \psi\bigg(\frac{m}{M}\bigg)  \bigg(\sum_{\substack{k\geq1\\k^2+h\equiv 0 \pmod{m}}}f\bigg(\frac{k^2}{X}\bigg)- \frac{\varrho_{h}(m) \sqrt{X} I_f}{m}  \bigg)\bigg| \\
 &\precprec  \max_{h \leq H} H  X^{\eta/16} (D_2')^{1/2} X^{1/4} ((D_2')^{1/2} + h^{1/4})\Bigl(1+\frac{\sqrt{X}}{D_2'(D_2'+h^{1/2})}\Bigr)^\theta \precprec HX^{1/2-\eta/16}
 \end{align*}
for $D_2' \precprec X^{1/4-\eta/2}$ and $H \leq X^{1/2}$.

\textbf{Estimation of  $\Sigma_2$} We bring $m$ outside by triangle inequality and bound the sum over $d|m$ by the divisor bound. We then apply Poisson summation (Lemma \ref{le:poisson}) to evaluate $\sum_k$ in $\sigma_0(k)$ with modulus $[m,r]$. Since $k \sim X$ is  longer than the modulus $[m,r] < X^{1-\eta/2}$, we only get the $t=0$ frequency. The rest of the argument is exactly the same as in Section \ref{sec:typeI} for Type I sums, using \eqref{eq:numberofroots}, Lemma \ref{le:modelforsquarefree}, and  the quadratic large sieve (Lemma \ref{lem:quadlargesieve}). 
\end{proof}

\section{Binary quadratic forms}\label{sec:quadratic}
In this  section we collect basic results that are needed in the proof of Proposition \ref{prop:typeII}.  With $\sigma_0,\sigma_1$ as in definition \ref{def:squaremodel}, we define the following binary counting functions which are central to our proof of the Type II proposition. 
\begin{definition}[Representation functions $\mathcal{Q}_{i,j}$]
Let $n_0,n_1,n_2 \geq 1$ be integers and $\kappa \pmod{n_0}$. Then for $i,j \in \{0,1\}$ and $m \in \Z$ we define
    \begin{align} \label{eq:Rijdef}
 \mathcal{Q}_{i,j}(m ;\mathbf{n},\kappa) := \sum_{\substack{k_1,k_2  \\ n_2 k_1  -n_1 k_2 =  m \\ k_1 \equiv \kappa \pmod{n_0} }} \sigma_i(k_1) \sigma_j(k_2).
\end{align}
\end{definition}
The variable $n_0$ is required to handle $\gcd$ issues and  the reader may want to pretend for simplicity in the first read  that $n_0=1$. For $i=j=1$ the functions $ \mathcal{Q}_{i,j}(m ;\mathbf{n},\kappa)$ then give a smoothly weighted count of representations of $m$ by the indefinite binary quadratic form $n_2 x_1^2 -n_1x_2^2=m$ with variables $x_1,x_2$ in a dyadic window around $\sqrt{X}$.

\subsection{Voronoï for binary quadratic forms}

We require the following Voronoï summation formula for the sum over arithmetic progressions of $\mathcal{Q}_{i,j}(m ;\mathbf{n},\kappa)$ of a shape
 $$ \sum_{u \equiv a \pmod{d}} \psi\left(\frac{u}{U}\right) \mathcal{Q}_{i,j}(u ;\mathbf{n},\kappa).$$
To state it concisely we need to set up some technical notation. We define the completed local main terms $\mathrm{MT}_{i,j}(a)=\mathrm{MT}_{i,j}(a;\mathbf{n},\kappa,d)$  for $i,j \in \{0,1\}$ by
\begin{align*}
    \mathrm{MT}_{1,1}(a) &= \sum_{\substack{\kappa_1, \kappa_2 \pmod{d n_0} \\ n_2 \kappa_1^2 - n_1 \kappa_2^2 \equiv a \pmod{d} \\ \kappa_1^2 \equiv \kappa \pmod{n_0}}} 1, \\
    \mathrm{MT}_{1,0}(a) &= \sum_{r \leq R}  \frac{1}{r^2} \sum_{\alpha \pmod r^\times} G(\alpha, r) \sum_{\substack{\kappa_1 \pmod{d n_0} \\ \kappa_2 \pmod{d n_0 r} \\ n_2 \kappa_1^2 - n_1 \kappa_2 \equiv a \pmod{d} \\ \kappa_1^2 \equiv \kappa \pmod{n_0}}} e_r(-\alpha \kappa_2), \\
    \mathrm{MT}_{0,0}(a) &= \sum_{r_1, r_2 \leq R} \frac{ 1}{r_1^2 r_2^2} \sum_{\substack{\alpha_1 \pmod{r_1^\times} \\ \alpha_2 \pmod{r_2^\times}}} G(\alpha_1, r_1)G(\alpha_2,r_2) \sum_{\substack{\kappa_1 \pmod{d n_0 r_1} \\ \kappa_2 \pmod{d n_0 r_2} \\ n_2 \kappa_1 - n_1 \kappa_2 \equiv a \pmod{d} \\ \kappa_1 \equiv \kappa \pmod{n_0}}} e_{r_1}(-\alpha_1 \kappa_1) e_{r_2}(-\alpha_2 \kappa_2),
\end{align*}
and symmetrically for $\mathrm{MT}_{0,1}(a)$.  For a given $U > 1$ and a smooth compactly supported function $\psi$ we let $\widehat{\Psi}_{i,j}$ denote 
\begin{align*}
    \widehat{\Psi}_{i,j}(\xi_1, \xi_2) = \int_{\R^2} f_i(x_1) f_j(x_2) \psi\left(\frac{n_2 P_i(x_1) - n_1 P_j(x_2)}{U}\right) e(-\xi_1 x_1 - \xi_2 x_2) \rd x_1 \rd x_2,
\end{align*}
where $P_1(x) = x^2, P_0(x) = x$, and $f_1(x) = \mathbf{1}_{x>0}f(x^2/X)$, while $f_0(x) = \frac{1}{2\sqrt{x}} f(x/X)$ for $x>0$ and $f_0(x)=0$ for $x\leq0$. These weights are smooth. Changing variables  shows that the zero-frequency integral and the following majorant do not depend on $i,j$:
\begin{align*}
  | \widehat{\Psi}_{i,j}(\xi_1, \xi_2)| \leq  \mathcal{I}_0 =\int_{\R^2} \bigg|f_i(x_1) f_j(x_2) \psi\left(\frac{n_2 P_i(x_1) - n_1 P_j(x_2)}{U}\right) \bigg| \rd x_1 \rd x_2 \ll \frac{U n_0 }{N}.
\end{align*}

\begin{lemma}[Truncated Voronoï summation formula] \label{le:voronoi}
Let $\eps  >0$, $T\geq 1,$ and let $\psi$ be a smooth function supported in $[-4,4]$ with $\psi^{(J)} \ll_J T^J $ for all $J\geq 0$. Let $N,U,X> 1$ with $X^{\eps}T dN R<U\leq NX/n_0$ and $X^{1/2} > X^{\eps} d n_0$. Let $d\geq 1$ and suppose that $n_0n_1, n_0n_2 \sim N$, with $(n_1,n_2)=1$.

Let $\mathcal{T}^\ast$ denote the set of non-zero dual frequencies $(t_1, t_2) \neq (0,0)$ satisfying 
\begin{align*}
    |t_1 \sqrt{n_1} + t_2 \sqrt{n_2}| \leq X^\eps \frac{Td n_0^{1/2} N^{1/2}}{X^{1/2}}, \quad \text{and} \quad |t_1 \sqrt{n_1} - t_2 \sqrt{n_2}| \leq X^\eps \frac{T d N^{3/2} X^{1/2}}{U n_0^{1/2}}.
\end{align*}
Let $\mathcal{E}(a)$ denote
\begin{align*}
    \mathcal{E}(a) = \frac{1}{d^2 n_0^2} \sum_{(t_1, t_2) \in \mathcal{T}^\ast} \widehat{\Psi}_{1,1}\left(\frac{t_1}{d n_0}, \frac{t_2}{d n_0}\right) \sum_{\substack{\kappa_1, \kappa_2 \pmod{d n_0} \\ n_2 \kappa_1^2 - n_1 \kappa_2^2 \equiv a \pmod{d} \\ \kappa_1^2 \equiv \kappa \pmod{n_0}}} e_{d n_0}(t_1 \kappa_1 + t_2 \kappa_2).
\end{align*}

Then for any $C > 0$ we have 
\begin{align*}
    \sum_{u \equiv a \pmod{d}} \psi\left(\frac{u}{U}\right) \mathcal{Q}_{i,j}(u ;\mathbf{n},\kappa)  = \frac{\widehat{\Psi}_{i,j}(0, 0)}{d^2 n_0^2 }\mathrm{MT}_{i,j}(a) + \1_{(i,j)=(1,1)}\mathcal{E}(a) + O_{\eps,C} (X^{-C}).
\end{align*}

\end{lemma}

\begin{remark} \label{rem:dual}
 We note that the dual range $\mathcal{T}^\ast$ is contained in
\begin{align*}
 |t_1^2 n_1 - t_2^2 n_2| \leq X^{2\eps} \frac{T^2 d^2 N^2}{U} \quad \text{and} \quad  |t_1|, |t_2|  \leq  X^{\eps}\frac{T d N X^{1/2}}{U}.
\end{align*}
In our application $Td$ will be bounded roughly of size $N/n_0$.
\end{remark}

\begin{proof}
We split the variables $k_1,k_2$ to residue classes $\kappa_1,\kappa_2$ to moduli as in $\mathrm{MT}_{i,j}(a)$, that is, for $i=j=1$ the modulus is $dn_0$, while for the model the modulus also depends on the $r$-variables.  Applying the Poisson summation formula to the sums over $k_1,k_2$ generates dual sums over frequencies $t_1, t_2 \in \Z$. The resulting zero-frequency integral $(t_1=t_2=0)$ takes the form $\widehat{\Psi}_{i,j}(0,0)$. We retain the actual zero-frequency integral in the main term; $\mathcal{I}_0$ is only an upper bound for its absolute value. In the phased case this integral is $\widehat{\Psi}^{\,\lambda}_{i,j}(0,0)$, with notation as in Lemma \ref{lem:phasedecayvoronoi}.

To bound the contribution of non-zero frequencies $(t_1, t_2) \neq (0,0)$, we apply repeated integration by parts (non-stationary phase). 

\textbf{1. The purely quadratic case $(i,j)=(1,1)$:}
We examine $\widehat{\Psi}_{1,1}(\xi_1, \xi_2)$ evaluated at $\xi_i = \frac{t_i}{d n_0}$, given by
\begin{align*}
    \widehat{\Psi}_{1,1}(\xi_1, \xi_2) = \int_{\R^2} F(x_1, x_2) e(-\xi_1 x_1 - \xi_2 x_2) \rd x_1 \rd x_2, 
\end{align*}
where $F(x_1, x_2) = f_1(x_1) f_1(x_2) \psi\left(\frac{n_2 x_1^2 - n_1 x_2^2}{U}\right)$. 
We make the linear change of variables $v = \sqrt{n_2}x_1 - \sqrt{n_1}x_2$ and $w = \sqrt{n_2}x_1 + \sqrt{n_1}x_2$, with Jacobian $(2\sqrt{n_1 n_2})^{-1}$. The phase becomes $-v\alpha - w\beta$, where
\begin{align*}
    \alpha = \frac{1}{2}\left(\frac{\xi_1}{\sqrt{n_2}} - \frac{\xi_2}{\sqrt{n_1}}\right) \quad \text{and} \quad \beta = \frac{1}{2}\left(\frac{\xi_1}{\sqrt{n_2}} + \frac{\xi_2}{\sqrt{n_1}}\right).
\end{align*}
Let $$\widetilde{F}(v,w)=F(x_1,x_2) =  f_1\bigg(\frac{v+w}{2\sqrt{n_2}}\bigg) f_1\bigg(\frac{w-v}{2\sqrt{n_1}}\bigg)  \psi\left(\frac{vw}{U}\right)$$ be the smooth weight in the new coordinates. Since $n_1, n_2 \asymp N/n_0$ and the support of $f$ restricts $x_1, x_2 \asymp X^{1/2}$, we have $w \asymp W := \sqrt{N X / n_0}$. The support of $\psi$ enforces $|vw| \ll U$, constraining $v \ll V := U / W$. By $U\leq NX/n_0$ we have $V\leq W$ and it then holds that
\begin{align} \label{eq:derivboundsvw}
    \frac{\partial^k}{\partial w^k} \frac{\partial^\ell}{\partial v^\ell} \widetilde{F}(v,w) \ll_{k,\ell} \delta^{-k-\ell}T^{k+\ell} W^{-k}V^{-\ell}.
\end{align}
Define the differential operators $\mathcal{D}_w = \frac{-1}{2\pi i \beta} \frac{\partial}{\partial w}$ and $\mathcal{D}_v = \frac{-1}{2\pi i \alpha} \frac{\partial}{\partial v}$. Applying $\mathcal{D}_w$ repeatedly $C$ times, integration by parts yields:
\begin{align*}
    \widehat{\Psi}_{1,1} = \frac{1}{2\sqrt{n_1 n_2}} \int \int \widetilde{F}(v,w) e(-v\alpha - w\beta) \rd v \rd w = \frac{1}{2\sqrt{n_1 n_2}} \int \int \left( (-\mathcal{D}_w)^C \widetilde{F} \right) e(-v\alpha - w\beta) \rd v \rd w.
\end{align*}
By the derivative bound \eqref{eq:derivboundsvw}, the integrand is bounded by $\ll_C\delta^{-C}(|\beta| W/T)^{-C}$. Hence, this integral is $O_{\eps,C}(X^{-C})$ unless $|\beta| W \leq  X^\eps T$, which translates to
\begin{align*}
    \left| \frac{\xi_1}{\sqrt{n_2}} + \frac{\xi_2}{\sqrt{n_1}} \right| \leq X^\eps T \sqrt{\frac{n_0}{N X}}.
\end{align*}
For $\xi_i= \frac{t_i}{d n_0}$ this gives $|t_1 \sqrt{n_1} + t_2 \sqrt{n_2}| \leq X^\eps \frac{Td n_0^{1/2} N^{1/2}}{X^{1/2}},$ which is the first inequality in $\mathcal{T}^\ast$.

Applying the same argument with $\mathcal{D}_v$ gives polynomial decay $\ll_C\delta^{-C}(|\alpha| V/T)^{-C}$ which allows us to truncate to $|\alpha| V \leq X^\eps T$. This translates to
\begin{align*}
     \left| \frac{\xi_1}{\sqrt{n_2}} - \frac{\xi_2}{\sqrt{n_1}} \right| \leq X^\eps T \frac{\sqrt{N X / n_0}}{U}.
\end{align*}
For $\xi_i= \frac{t_i}{d n_0}$ this gives $|t_1 \sqrt{n_1} - t_2 \sqrt{n_2}| \leq X^\eps \frac{Td N^{3/2} X^{1/2}}{U n_0^{1/2}},$ which is the second inequality in $\mathcal{T}^\ast$.

\textbf{2. The model cases $(1,0),(0,1),(0,0)$:}
The case $(0,0)$ is similar but easier so let us consider $(1,0)$. Consider the evaluation of $\sigma_0(k_2)$ inside the summation $\mathcal{Q}_{1,0}$. Swapping the arithmetic model sums to the outside, we evaluate the Poisson summation over $k_1 \equiv \kappa_1 \pmod{dn_0}$ and $k_2 \equiv \kappa_2 \pmod{dn_0 r}$. 

Consequently, we examine $\widehat{\Psi}_{1,0}(\xi_1, \xi_2)$ with $\xi_2 = \frac{t_2}{d n_0 r}$ and $\xi_1= \frac{t_1}{dn_0}$. We substitute $y = n_2 x_1^2 - n_1 x_2$. The integral becomes:
\begin{align*}
    \frac{1}{n_1}\int_\R f_1(x_1) e(-\xi_1 x_1 - \xi_2 \frac{n_2}{n_1}x_1^2) \left( \int_\R f_0\left(\frac{n_2 x_1^2 - y}{n_1}\right) \psi\left(\frac{y}{U}\right)  e\left(\xi_2 \frac{y}{n_1}\right) \rd y \right) \rd x_1.
\end{align*}
Let $G(y, x_1) = f_0(\dots) \psi(y/U)$ be the inner amplitude. The support of $\psi$ forces $y \ll U$. Differentiating with respect to $y$, we get $\frac{\partial^k}{\partial y^k} G(y, x_1) \precprec_k X^{-1/2}T^k U^{-k}$. Integrating by parts $C$ times with respect to $y$ using the operator $\frac{n_1}{2\pi i \xi_2} \frac{\partial}{\partial y}$ yields  decay $O_A(X^{-C})$ unless  $\frac{|\xi_2|}{n_1} \leq T\frac{X^\eps}{U}$. Substituting $\xi_2= \frac{t_2}{dn_0r}$, we require
\begin{align*}
    |t_2| \leq X^\eps T\frac{d n_0 r n_1}{U} \asymp X^\eps T\frac{d N r}{U} \leq X^\eps T\frac{d N R}{U}.
\end{align*}
Because we assume that $U > TX^{\eps} dN R$, this bound forces $|t_2| < 1$. Since $t_2 \in \Z$, it must be that $t_2 = 0$. This implies  $\xi_2=0$ and it remains to consider
\begin{align*}
    \frac{1}{n_1}\int_\R f_1(x_1) e(-\xi_1 x_1) \left( \int_\R f_0\left(\frac{n_2 x_1^2 - y}{n_1}\right) \psi\left(\frac{y}{U}\right) \rd y \right) \rd x_1.
\end{align*}
The weight $g(x_1)=\frac{f_1(x_1)}{n_1} \left( \int_\R f_0\left(\frac{n_2 x_1^2 - y}{n_1}\right) \psi\left(\frac{y}{U}\right) \rd y \right) $ satisfies $\frac{\partial^k}{\partial x^k} g(x) \precprec_k \frac{U}{n_1\sqrt X}(\sqrt{X})^{-k} $ and so we get fast decay outside $|\xi_1| \leq X^{-1/2+\eps}$. Substituting $\xi_1 = \frac{t_1}{d n_0} $ we get only $t_1=0$ since $X^{1/2} > X^{\eps}dn_0$.

\end{proof}

For our application we will need the following bound for the integrals when $\psi$ contains a linear phase.
\begin{lemma} \label{lem:phasedecayvoronoi}
   Let $U,N,X>1$ and $n_0n_1,n_0n_2\sim N$. Let $\psi_0$ be a fixed smooth function supported in
$[-4,4]$, with $\psi_0^{(J)}\ll_J 1$ for every $J\geq 0$, and put
\[
    \psi_\lambda(x)=\psi_0(x)e(\lambda x),
    \qquad T=1+|\lambda|,
    \qquad \lambda\in\R.
\]
Denote
\begin{align*}
    \widehat{\Psi}^{\,\lambda}_{i,j}(\xi_1,\xi_2)
    ={}&\int_{\R^2} f_i(x_1)f_j(x_2)
    \psi_0\left(\frac{n_2P_i(x_1)-n_1P_j(x_2)}{U}\right)\\
    &\qquad\times
    e\left(\lambda\frac{n_2P_i(x_1)-n_1P_j(x_2)}{U}
        -\xi_1x_1-\xi_2x_2\right)\rd x_1\rd x_2.
\end{align*}
Then
\begin{equation}\label{eq:voronoi_twisted_decay}
    \left|\widehat{\Psi}^{\,\lambda}_{1,1}(\xi_1,\xi_2)\right|
    \precprec \frac{Un_0}{N(1+|\lambda|)}
    =\frac{Un_0}{NT},
\end{equation}
uniformly for all real $\xi_1,\xi_2$ and $\lambda$.
\end{lemma}

\begin{proof}
   On the support of $f_1(x_1)f_1(x_2)$ we have $x_1,x_2>0$.
Use the coordinates
\[
    v=\sqrt{n_2}x_1-\sqrt{n_1}x_2,
    \qquad w=\sqrt{n_2}x_1+\sqrt{n_1}x_2,
\]
and set
\[
    W=\sqrt{NX/n_0},\qquad
    V=\min\left(W,\frac{U}{W}\right),\qquad
    \vartheta=\frac{VW}{U}\leq 1.
\]
On the support, $w\asymp W$ and $|v|\ll V$. After writing
$v=Vs$ and $w=Wt$, the integral takes the form
\[
    \frac{VW}{2\sqrt{n_1n_2}}
    \int_{\R}\int_{\R}
    A(s,t)e\left(\lambda\vartheta st-\alpha_0s-\beta_0t\right)
    \rd s\rd t,
\]
where $\alpha_0,\beta_0\in\R$ depend on the Fourier variables.
The function $A$ is supported on $|s|\ll1$, $t\asymp1$, and
$\partial_s^jA\precprec_j1$ for each fixed $j$. Indeed,
$V/W\leq1$ and $\vartheta\leq1$, while the derivative losses
from $f$ are $X^{o(1)}$.
Two integrations by parts in $s$, together with the trivial bound
when the coefficient of $s$ is small, give
\[
    \left|\int_{\R}A(s,t)
       e\left((\lambda\vartheta t-\alpha_0)s\right)\rd s\right|
    \precprec (1+|\lambda\vartheta t-\alpha_0|)^{-2}.
\]
Integrating this bound over the fixed compact range of $t$ yields
\[
    \left|\widehat{\Psi}^{\,\lambda}_{1,1}(\xi_1,\xi_2)\right|
    \precprec
    \frac{VW}{\sqrt{n_1n_2}}
    \frac{1}{1+|\lambda|\vartheta}
    \ll
    \frac{Un_0}{N}\frac{\vartheta}{1+|\lambda|\vartheta}
    \leq\frac{Un_0}{N(1+|\lambda|)}.
\]
This proves \eqref{eq:voronoi_twisted_decay}.
\end{proof}

\subsection{Comparing the main terms}
In this section our aim is to compare the main terms in Lemma \ref{le:voronoi}. We first have the following lemma.
\begin{lemma} \label{le:gausssum}
Let $d,n_0',r\geq1$ and $\lambda\in\Z$ be integers, with $r\mid dn_0'$. Denote
\begin{align*}
    S(\lambda, r, n_0') = \sum_{\alpha \pmod r^\times} G(\alpha, r) \sum_{\substack{y \pmod{d n_0'} \\ y \equiv \kappa \pmod{n_0'}}} e_d(\lambda y) e_r(-\alpha y).
\end{align*}
Let $d_1 = d / (d, n_0')$ and $m = (d_1, \lambda)$. Let $r_1 = r / (r, n_0')$. Then $S(\lambda, r, n_0') = 0$ unless $r_1 = d_1 / m$ and we have
\begin{align*}
    \frac{1}{r} |S(\lambda, r, n_0')| \leq  \mathbf{1}_{d_1=r_1m} \sqrt{2d (d,n_0') (r,n_0') m}.
\end{align*}
\end{lemma}
\begin{proof}
Let $y = \kappa + \ell n_0'$ for $0 \leq \ell < d$. This gives
\begin{align*}
    \sum_{\ell=0}^{d-1} e_d(\lambda(\kappa + \ell n_0')) e_r(-\alpha(\kappa + \ell n_0')) = e\left(\kappa \left(\frac{\lambda}{d} - \frac{\alpha}{r}\right)\right) \sum_{\ell=0}^{d-1} e\left( \ell \left(\frac{\lambda n_0'}{d} - \frac{\alpha n_0'}{r}\right) \right).
\end{align*}
By orthogonality since $r \mid d n_0'$, the sum over $\ell$ evaluates to $d$ if $\frac{\lambda n_0'}{d} - \frac{\alpha n_0'}{r} \in \Z$, and $0$ otherwise. Thus, if the sum is non-vanishing, we have
\begin{align*}
    \frac{\alpha n_0'}{r} \equiv \frac{\lambda n_0'}{d} \pmod 1.
\end{align*}
Reducing to lowest terms,  we get
\begin{align*}
        \frac{\alpha n_0'/(n_0',r)}{r_1} \equiv \frac{(\lambda/m) n_0'/(n_0',d)}{d_1/m} \pmod 1.
\end{align*}
For two reduced fractions to be equivalent modulo 1, their reduced denominators must be identical and we get $r_1 = d_1 / m$. 

When this  equality holds, the congruence $\alpha (n_0' / (n_0', r)) \equiv \frac{\lambda}{m} (n_0' / (n_0', d)) \pmod{r_1}$ determines a unique residue class for $\alpha \pmod{r_1}$. Because $\alpha$ is originally summed over the larger modulus $r = r_1 (r, n_0')$, there are at most $(r, n_0')$ distinct lifts of this residue class to modulo $r$. Therefore, there are at most $(r, n_0')$ valid choices for $\alpha \pmod r$. For each such $\alpha$, we apply the standard Gauss sum bound $|G(\alpha, r)| \leq \sqrt{2r}$ (valid also for $r$ having powers of 2). Summing over the at most $(r, n_0')$ valid choices for $\alpha$, and multiplying by the factor $d$ of the $\ell$-sum, we get
\begin{align*}
    |S(\lambda, r, n_0')| \leq d (r, n_0') \sqrt{2r} .
\end{align*}
Dividing by $r$ and substituting $r_1 = d_1 / m$, we obtain the weighted bound:
\begin{align*}
    \frac{1}{r} |S(\lambda, r, n_0')| \leq d \sqrt{2} \frac{(r, n_0')}{\sqrt{r}} = d \sqrt{2} \frac{\sqrt{(r, n_0')}}{\sqrt{r_1}} \leq d \sqrt{2 (r,n_0')} \sqrt{\frac{m}{d_1}} \leq \sqrt{2d (d,n_0') (r,n_0') m}.
\end{align*}
which completes the proof.
\end{proof}

We use this evaluation to capture the cancellation of the main terms of the squares against the model from Lemma \ref{le:voronoi}.
\begin{lemma}[Main Term Cancellation]\label{lem:main_term_cancellation}
Let $\mathrm{MT}_{i,j}(a)$  be as in Lemma \ref{le:voronoi} and define
\begin{align*}
  \mathcal{M}(b,d) =    \mathcal{M}_{n_1,n_2}(b,d) = \sum_{a \pmod{d}} e_d(ab) \sum_{i,j \in \{0,1\}} (-1)^{i+j} \mathrm{MT}_{i,j}(a).
\end{align*}
Then for  $(b,d)=1$  we have
   \begin{align*}
    |\mathcal{M}(b,d)| \precprec
    \mathbf{1}_{d>R}n_0^{3/2}d
    \sqrt{(d,n_0)(d,n_1)(d,n_2)}.
\end{align*}
In particular, if $d \leq R$, we have $\mathcal{M}(b,d)=0$.
\end{lemma}

\begin{proof}
By bringing the sum over $a \pmod d$ inside the definition of $\mathrm{MT}_{i,j,\kappa}(a)$, summing over $a$ translates the condition $n_2 \kappa_1^2 - n_1 \kappa_2^2 \equiv a \pmod d$  into the exponential phase 
\begin{align*}
    e_d(b(n_2 \kappa_1^2 - n_1 \kappa_2^2)) = e_d(b n_2 \kappa_1^2) e_d(-b n_1 \kappa_2^2).
\end{align*}
The alternating sum of the main terms factors exactly into a product analogous to  $ (\sigma_1(k_1) - \sigma_0(k_1))(\sigma_1(k_2) - \sigma_0(k_2))$. The condition $\kappa_1^2 \equiv \kappa \pmod{n_0}$ applies only to the first variable, second variable $\kappa_2$ is unconstrained modulo $n_0$ and gives a factor of $n_0$. Thus, we have 
\begin{align} \label{eq:M=E1E2}
  \mathcal{M}(b,d) = E_1(bn_2, d)  n_0 E_2(-bn_1, d),
\end{align}
where
\begin{align*}
  E_1(\lambda, d) &=   \sum_{\substack{x \pmod{d n_0} \\ x^2 \equiv \kappa \pmod{n_0}}} e_d(\lambda x^2)  -  \sum_{r \leq R}  \frac{1}{r^2} \sum_{\alpha \pmod r^\times} G(\alpha, r) \sum_{\substack{y \pmod{d n_0 r}  \\ y\equiv \kappa \pmod{n_0}}} e_r(-\alpha y)  e_d(\lambda y) \\
  E_2(\lambda,d) &= \sum_{x \pmod{d}} e_d(\lambda x^2)  -  \sum_{r \leq R}  \frac{1}{r^2} \sum_{\alpha \pmod r^\times} G(\alpha, r) \sum_{\substack{y \pmod{d  r}  }} e_r(-\alpha y)  e_d(\lambda y).
\end{align*}

Let us first consider $E_1$ and the second term corresponding to the model in it.  We parameterize $y \pmod{d n_0 r}$ as $y = y_1 + \ell d n_0$, where $y_1 \pmod{d n_0}$ and $\ell \pmod{r}$. The phase becomes
\begin{align*}
    e_d(\lambda y_1) e_r(-\alpha (y_1 + \ell d n_0)).
\end{align*}
Summing over $\ell \pmod{r}$ gives $r$ times the indicator of  $r \mid d n_0$ and we get
\begin{align} \label{eq:modelMT}
     \sum_{\substack{r \leq R \\ r \mid d n_0}} \frac{1}{r} \sum_{\alpha (r)^\times} G(\alpha, r) \sum_{\substack{y_1 \pmod{d n_0} \\ y_1 \equiv \kappa \pmod{n_0}}} e_d(\lambda y_1) e_r(-\alpha y_1). 
\end{align}
To compare this with the first term in $E_1$, which involves the  squares, we group the residues $x$ according to their square $y \equiv x^2 \pmod{dn_0}$. Let $\rho_{dn_0}(y) = |\{ x \pmod{dn_0} : x^2 \equiv y \pmod{dn_0} \}|$. By orthogonality
\begin{align*}
    \rho_{dn_0}(y) = \frac{1}{dn_0} \sum_{\xi \pmod{dn_0}} \sum_{x \pmod{dn_0}} e_{dn_0}(\xi(x^2 - y)).
\end{align*}
Writing the frequency  in lowest terms as $\frac{\xi}{dn_0} = \frac{\alpha}{r}$ with $r \mid dn_0$ and $(\alpha, r) = 1$ we get 
\begin{align*}
    \rho_{dn_0}(y) = \sum_{r \mid dn_0} \frac{1}{r} \sum_{\alpha (r)^\times} G(\alpha, r) e_r(-\alpha y).
\end{align*}
and we get
\begin{align*}
    \sum_{\substack{x \pmod{d n_0} \\ x^2 \equiv \kappa \pmod{n_0}}} e_d(\lambda x^2) =   \sum_{\substack{r \mid d n_0}} \frac{1}{r} \sum_{\alpha (r)^\times} G(\alpha, r) \sum_{\substack{y \pmod{d n_0} \\ y \equiv \kappa \pmod{n_0}}} e_d(\lambda y) e_r(-\alpha y). 
\end{align*}
Thus, subtracting the model \eqref{eq:modelMT} only the tail $r>R$ survives 
\begin{align*}
    E_1(\lambda, d) = \sum_{\substack{r > R \\ r \mid d n_0}} \frac{1}{r} \sum_{\alpha (r)^\times} G(\alpha, r) \sum_{\substack{y \pmod{d n_0} \\ y \equiv \kappa \pmod{n_0}}} e_d(\lambda y) e_r(-\alpha y) = \sum_{\substack{r > R \\ r \mid d n_0}} \frac{1}{r} S(\lambda, r, n_0),
\end{align*}
where $S(\lambda, r, n_0)$ is as in Lemma \ref{le:gausssum}. The second factor $E_2(\lambda, d)$ is similar but easier and we have
\begin{align*}
    E_2(\lambda, d) = \sum_{\substack{r > R \\ r \mid d}} \frac{1}{r} S(\lambda, r, 1).
\end{align*}

We now bound these tails using Lemma \ref{le:gausssum}.
For $E_1(bn_2,d)$, put $d_1=d/(d,n_0)$ and $m=(d_1,n_2)$.
A non-zero term satisfies
\[
    r=c\,d_1/m,\qquad c=(r,n_0)\mid n_0.
\]
Thus each $c$ determines at most one $r$; the condition $r>R$
is $d_1/m>R/c$. Dropping this condition gives
\begin{align*}
    |E_1(bn_2,d)|
    &\ll \sqrt{d(d,n_0)(d,n_2)}
        \sum_{c\mid n_0}\sqrt c\\
    &\precprec \sqrt{dn_0(d,n_0)(d,n_2)}.
\end{align*}
For $E_2(-bn_1,d)$ the only possible modulus is
$r=d/(d,n_1)$, so
\[
    |E_2(-bn_1,d)|
    \ll \mathbf{1}_{d>R}\sqrt{d(d,n_1)}.
\]
The asserted bound follows from \eqref{eq:M=E1E2}.
\end{proof}

\subsection{Exponential sum bound for the error term}
For medium sized conductor $d$ we need the following exponential sum bound to handle non-zero frequencies $t_1,t_2$.
\begin{lemma} \label{lem:exp_sum_bound}
Let $n_0, n_1, n_2, d$ be positive integers with $n_0$ square-free and $(n_1, n_2) = 1$, and let $b$ be an integer coprime to $d$. For any integers $t_1, t_2$, the exponential sum
\begin{align*}
    \mathcal{E}(b, d, t_1, t_2) = \sum_{a \pmod{d}}  e_d(ab) \sum_{\substack{\kappa_1, \kappa_2 \pmod{d n_0} \\ n_2 \kappa_1^2 - n_1 \kappa_2^2 \equiv a \pmod{d} \\ \kappa_1^2 \equiv \kappa \pmod{n_0}}} e_{d n_0}(t_1 \kappa_1 + t_2 \kappa_2)
\end{align*}
satisfies 
\begin{align*}
    \mathcal{E}(b, d, t_1, t_2) \ll \tau(n_0) n_0 (d,n_0) d \sqrt{(d,n_1)(d,n_2)}.
\end{align*}
\end{lemma}
\begin{proof}
Bringing the sum over $a \pmod{d}$ inside, the exponential sum factors into two 1-dimensional sums over $\kappa_1$ and $\kappa_2$
\begin{align*}
    \mathcal{E}(b, d, t_1, t_2) &= \sum_{\substack{\kappa_1, \kappa_2 \pmod{d n_0} \\ \kappa_1^2 \equiv \kappa \pmod{n_0}}} e_{dn_0} \Big( b n_0 (n_2 \kappa_1^2 - n_1 \kappa_2^2) + t_1 \kappa_1 + t_2 \kappa_2 \Big) = S_1 \times S_2,
\end{align*}
where
\begin{align*}
    S_1 &= \sum_{\substack{\kappa_1 \pmod{d n_0} \\ \kappa_1^2 \equiv \kappa \pmod{n_0}}} e_{d n_0}\left(b n_0 n_2 \kappa_1^2 + t_1 \kappa_1\right), \\
    S_2 &= \sum_{\kappa_2 \pmod{d n_0}} e_{d n_0}\left(-b n_0 n_1 \kappa_2^2 + t_2 \kappa_2\right).
\end{align*}

The sum $S_2$ is a quadratic Gauss sum which satisfies
\begin{align} \label{eq:S2_bound}
    |S_2| \le  n_0 \sqrt{2 d (d, n_1)}.
\end{align}

The sum $S_1$ is restricted by the congruence $\kappa_1^2 \equiv \kappa \pmod{n_0}$. The number of roots $v^2 \equiv \kappa \pmod{n_0}$  is bounded by $\tau(n_0)$. For each root $v $, we can parameterize the variable $\kappa_1 \pmod{dn_0}$ uniquely as $\kappa_1 = v + y n_0$, where $y$ ranges completely over residue classes modulo $d$ and
\begin{align*}
    e_{dn_0}\left(b n_0 n_2 \kappa_1^2 + t_1 \kappa_1\right) = e_{dn_0}(b n_0 n_2 v^2 + t_1 v) \cdot e_d\left(b n_0^2 n_2 y^2 + (2 b n_0 n_2 v + t_1) y\right).
\end{align*}
The sum over $y \pmod d$  is precisely a Gauss sum modulo $d$
\begin{align*}
    \sum_{y \pmod d} e_d(b n_0^2 n_2 y^2 + (2 b n_0 n_2 v + t_1) y),
\end{align*}
where the quadratic coefficient is $ b n_0^2 n_2$. 
Applying the Gauss sum bound this is bounded by
\begin{align*}
      \bigg|\sum_{y \pmod d} e_d(b n_0^2 n_2 y^2 + (2 b n_0 n_2 v + t_1) y) \bigg| \leq (d,n_0) \sqrt{2 d (d,n_2)}.
\end{align*}
Thus, we conclude
\begin{align} \label{eq:S1_bound}
    |S_1| \le \tau(n_0) (d,n_0) \sqrt{2 d (d,n_2)}.
\end{align}

Multiplying the bounds \eqref{eq:S2_bound} and \eqref{eq:S1_bound} together yields the combined bound for the full exponential sum
\begin{align*}
|\mathcal{E}(b, d, t_1, t_2)| = |S_1| |S_2| \le 2 \tau(n_0) n_0 ( d,n_0) d \sqrt{( d,n_1)(d,n_2)},
\end{align*}
which completes the proof.
\end{proof}

\subsection{Number of dual frequencies}
We need the following elementary bound for the number of dual frequencies in Lemma \ref{le:voronoi}.
\begin{lemma} \label{lem:number_of_dual}
Let $M,N',T \ge 10$ with $M \leq N'/10$. Let $n_1, n_2 \sim N'$ be square-free with $(n_1, n_2) = 1$. The number of  $|t_1|, |t_2| \leq T, (t_1,t_2) \neq (0,0) $ satisfying 
\[
     |t_1^2 n_2 - t_2^2 n_1| \leq M
\]
is bounded by
\[
    \ll M^{1/2} \log T.
\]
\end{lemma}

\begin{proof}
By $(t_1,t_2) \neq (0,0) $  and $M \leq N'/10$ we have $t_1,t_2 \neq 0$. By symmetry, we can assume $t_1, t_2 > 0$. We divide the inequality by $n_1 t_1^2$ to obtain
\begin{align*}
    \left| \frac{n_2}{n_1} - \frac{t_2^2}{t_1^2} \right| \leq \frac{M}{n_1 t_1^2}.
\end{align*}
We factor the difference of squares on the left-hand side:
\begin{align*}
    \left| \sqrt{\frac{n_2}{n_1}} - \frac{t_2}{t_1} \right| \left( \sqrt{\frac{n_2}{n_1}} + \frac{t_2}{t_1} \right) \leq \frac{M}{n_1 t_1^2}.
\end{align*}
Because $n_1, n_2 \in [N', 2N']$ and $t_1,t_2 > 0$, we have
\begin{align*}
    \left| \sqrt{\frac{n_2}{n_1}} - \frac{t_2}{t_1} \right| \leq \frac{ \sqrt{2} M}{n_1 t_1^2}.
\end{align*}

Let $y/x$ be the fraction $t_2/t_1$ reduced to lowest terms, such that $t_1 = g x$ and $t_2 = g y$ where $g = (t_1, t_2) \ge 1$. Substituting this into our error bound yields:
\begin{align*}
    \left| \sqrt{\frac{n_2}{n_1}} - \frac{y}{x} \right| \leq \frac{\sqrt{2} M}{N' g^2 x^2} \le \frac{\sqrt{2} M}{N' x^2}.
\end{align*}
By hypothesis, $M \leq N'/10$. Therefore, the coefficient $\sqrt{2} M / N' \le \sqrt{2}/10 < 1/2$ and yields
\begin{align*}
    \left| \sqrt{\frac{n_2}{n_1}} - \frac{y}{x} \right| < \frac{1}{2 x^2}.
\end{align*}
Since $n_1,n_2$ are square-free and coprime, $\sqrt{n_2/n_1}$ is irrational. By Legendre's theorem on Diophantine approximation, any reduced fraction $y/x$ that approximates within $1/(2x^2)$ must be a principal convergent of the continued fraction expansion of that number. Because the denominators of principal convergents grow exponentially, the number of such convergents with denominator $x \le T$ is at most $O(\log T)$ uniformly in $n_1,n_2$. Thus, for every $g$ there are at most $O(\log T)$ valid coprime $(x, y)$.

We now bound the number of 
$g$. Substituting $t_1 = gx$ and $t_2 = gy$ back into the original absolute difference yields:
\begin{align*}
    g^2 |x^2 n_2 - y^2 n_1| \leq M.
\end{align*}
Because $n_1$ and $n_2$ are square-free and coprime, $|x^2 n_2 - y^2 n_1|$ is a non-zero integer and thus it is $\geq 1$ and we get
\begin{align*}
    g \le M^{1/2}.
\end{align*}
Hence, the number of $g$ is $\leq M^{1/2}$ and for each $g$ there are $O(\log T)$ many solutions.
\end{proof}

\subsection{An $L^2$ bound}
To use the large sieve inequality we need an $L^2$-bound for the representation functions. First, we prove an $\ell^\infty$ bound for $\mathcal{Q}_{1,1}$.
\begin{lemma} Let $n_0n_1,n_0n_2\sim N\leq X$, with $n_1,n_2$ square-free and coprime and $n_1n_2>1$. We have \label{le:divisorbound_realquadratic}
    \begin{align*}
        \mathcal{Q}_{1,1}(u,n_0,n_1,n_2,\kappa) \precprec 1.
    \end{align*}
\end{lemma}
\begin{proof}
    If $u=0$, the sum is zero since $n_1n_2$ is not a square. Assume $u\neq0$. By definition,
    \begin{align*}
        \mathcal{Q}_{1,1}(u,n_0,n_1,n_2,\kappa)  = \sum_{\substack{y_1,y_2\geq1 \\ n_2 y_1^2-n_1 y_2^2=u \\ y_1^2\equiv \kappa \pmod{n_0}}} f(y_1^2 /X)f(y_2^2 /X).
    \end{align*}
    Let $D = n_1 n_2$. Multiplying the principal equation by $n_2$, we obtain the Pell-type equation $(n_2 y_1)^2 - D y_2^2 = n_2 u$. 
   Since $D$ is not a perfect square, we work in the ring of integers $\mathcal{O}_K$ of the real quadratic field $K = \mathbb{Q}(\sqrt{D})$. The equation implies that the element $\alpha = n_2 y_1 + y_2 \sqrt{D} \in \mathcal{O}_K$ has algebraic norm $N_{K/\mathbb{Q}}(\alpha) = n_2 u$. 
    The number of principal ideals $(\alpha)$ of norm $|n_2u|$ is at most $\tau(|n_2u|)^2\precprec1$. 
    
    For a fixed ideal $(\alpha)$, the generator is unique up to multiplication by a unit $\epsilon \in \mathcal{O}_K^\times$. By Dirichlet's Unit Theorem, the unit group is generated by $\pm 1$ and a fundamental unit $\epsilon_0 > 1$. The smooth weights $f$ enforce the size restrictions $y_1, y_2 \ll X^{1/2}$, which means the  size of the generator is strictly bounded by
    \begin{align*}
        |\alpha| \leq n_2 y_1 + \sqrt{D} y_2 \ll n_2 X^{1/2} + \sqrt{n_1 n_2} X^{1/2} \ll N X^{1/2}.
    \end{align*}
    Similarly, we have $ |\overline{\alpha}| \ll N X^{1/2}$ which by $|\alpha\overline{\alpha}| \geq 1$ implies that
    \begin{align} \label{eq:alphaconst}
        1/(NX^{1/2}) \ll |\alpha| \ll N X^{1/2}.
    \end{align}
    Under the Minkowski embedding $\alpha \mapsto \log|\alpha|$, the units map to a lattice in $\mathbb{R}$. Since $\log\epsilon_0\gg1$ uniformly in real quadratic fields, \eqref{eq:alphaconst} confines the exponent $k$ of $\epsilon_0^k$ to an interval of length $O(\log(NX^{1/2}))\ll\log X$. Thus, there are at most $O(\log X)$ valid units for each ideal.
\end{proof}

\begin{lemma} \label{le:L2bound}
Let $n_0n_1,n_0n_2\sim N$, where $n_0,n_1,n_2$ are square-free
and pairwise coprime and $n_1n_2>1$. Suppose that
$R^4\leq N<X^{1/2}$ and $U>X^{2\eps}N^2$. Then
\[
    \sum_{|u|\leq U}
    |\mathcal Q_{i,j}(u,n_0,n_1,n_2,\kappa)|^2
    \precprec \frac UN.
\]
\end{lemma}
\begin{proof}
For the purely quadratic term, Lemma \ref{le:divisorbound_realquadratic}
reduces the claim to an $L^1$ bound. If $U\geq NX/n_0$,
counting the two positive square roots gives
$\sum_u\mathcal Q_{1,1}(u)\precprec X/n_0\leq U/N$.
Otherwise $U>X^{2\eps}N^2$ implies $n_0\ll X^{1/2-\eps}$.
Apply Lemma \ref{le:voronoi} with $d=T=1$, with $\eps/4$
in place of $\eps$, and with a non-negative smooth majorant
of $\mathbf1_{[-1,1]}$. Its size hypotheses hold since $N\geq R^4$.
Every dual frequency satisfies
$|n_1t_1^2-n_2t_2^2|<1$, so there are no non-zero dual frequencies.
Writing $\rho_{n_0}(\kappa)=\#\{x\pmod{n_0}:x^2\equiv\kappa\pmod{n_0}\}$,
the main term is bounded by
\[
    \frac{\mathcal I_0}{n_0^2}\,
        n_0\rho_{n_0}(\kappa)
    \precprec \frac UN,
\]
since $\rho_{n_0}(\kappa)\leq\tau(n_0)$.

For the model terms, expanding the Ramanujan sums gives
\[
    \frac1r\sum_{\alpha\pmod r^\times}G(\alpha,r)e_r(-\alpha k)
    =\sum_{a\mid r}\mu(r/a)\rho_a(k).
\]
The bound $\rho_a(k)\precprec\sqrt{(a,k)}$ therefore implies,
on the support of $f(k/X)$,
\[
    |\sigma_0(k)|
    \precprec \frac{R}{\sqrt X}
       \sum_{a\leq R}\frac{\sqrt{(a,k)}}a
    \precprec \frac{R}{\sqrt X}.
\]
For a mixed term, the square root lies in $\precprec1$ residue
classes modulo $n_0n_1$ or $n_0n_2$; for the $(0,0)$ term,
the first variable lies in one residue class modulo $n_0n_1$.
Thus
\[
    \|\mathcal Q_{1,0}\|_\infty+\|\mathcal Q_{0,1}\|_\infty
       \precprec R/N,\qquad
    \|\mathcal Q_{0,0}\|_\infty\precprec R^2/N.
\]
Counting the other variable in an interval of length
$O(U/n_1)$ or $O(U/n_2)$ gives
\[
    \sum_{|u|\leq U}
       (|\mathcal Q_{1,0}(u)|+|\mathcal Q_{0,1}(u)|)
       \precprec RU/N,\qquad
    \sum_{|u|\leq U}|\mathcal Q_{0,0}(u)|
       \precprec R^2U/N.
\]
Multiplying the $L^\infty$ and $L^1$ bounds gives at most
$R^4U/N^2\leq U/N$, as required.
\end{proof}

\section{Proof of Proposition \ref{prop:typeII}} \label{sec:typeII}
\subsection{First Cauchy-Schwarz}
We may assume that the coefficients $\beta_n$ are supported in the interval $[N/4,2N]$ with $MN=X$, by the support of $w^{\gamma}_{mn}$. We first split this interval into $O(1)$ dyadic intervals; the estimates below are unchanged by fixed factors in the dyadic scales. By applying Cauchy-Schwarz in $m$, extending the resulting non-negative sum to all positive $m$, and expanding the square, we have
\begin{align*}
    \sum_{m \sim M} \bigg| \sum_{n} \alpha_m \beta_n w^{(\gamma)}_{mn} \bigg|  \precprec M^{1/2} \mathcal{U}(N)^{1/2},
\end{align*}
where $\mathcal{U}(N)$ is given by
\begin{align*}
     \mathcal{U}(N) &= \sum_{n_1, n_2} \beta_{n_1} \overline{ \beta_{n_2} } \sum_{m}w^{(\gamma)}_{mn_1} \overline{w^{(\gamma)}_{mn_2}} \\
     &= \sum_{n_1, n_2} \beta_{n_1} \overline{ \beta_{n_2} } \sum_{h_1,h_2 \leq H}  \gamma_{h_1} \overline{\gamma_{h_2}} \sum_{\substack{m,k_1,k_2 \\ m n_1 = k_1+h_1 \\ m n_2 = k_2+h_2 }} \sigma^{\flat}(k_1)\sigma^{\flat}(k_2).
\end{align*} 
To establish the Proposition \ref{prop:typeII}, it suffices to show that
\begin{align} \label{eq:bound1}
    \mathcal{U}(N) \precprec N H^2 R^{-1/2}.
\end{align}
We partition the sum into two pieces. $\mathcal{U}_0(N)$ denotes the diagonal contribution where $n_1=n_2$, and $\mathcal{U}_1(N)$ denotes the off-diagonal contribution where $n_1 \neq n_2$. 

\textbf{The diagonal term $\mathcal{U}_0(N)$:}
For the diagonal $n_1=n_2=n$, we bound the number of factorizations $mn$ by a divisor bound and thus get
\begin{align*}
     \mathcal{U}_0(N) \precprec \sum_{n} |w_n^{(\gamma)}|^2 \precprec \sum_{n} (|a_n^{(\gamma)}|^2 + |b_n^{(\gamma)}|^2).
\end{align*}
The linear equations require $n = k_1 + h_1 = k_2 + h_2$, which implies  $k_1 - k_2 = h_2 - h_1 = u$, with $|u| \leq H$.  

For  the $|a_n^{(\gamma)}|^2 $, we consider $\sigma_1(k_1)\sigma_1(k_2)$. Here $k_1 = x_1^2$ and $k_2 = x_2^2$ are both squares. If $u \neq 0$, the equation $x_1^2 - x_2^2 = u$ has at most $\tau(|u|) \precprec 1$ solutions.  Summing over the $H^2$ choices for $(h_1, h_2)$, the total contribution is $\precprec H^2$.
    If $u = 0$, then $h_1 = h_2$ and $x_1 = x_2 = x$ and this contribution is $\precprec H X^{1/2}$. Thus, we have
   $\sum_{n} |a_n^{(\gamma)}|^2 \precprec H^2 + H X^{1/2}$.

For the $|b_n^{(\gamma)}|^2 $ term, we have $|\sigma_0(k_1)\sigma_0(k_2)| \precprec R^2 X^{-1}$.  Summing over $k_1-k_2=h_2-h_1 \ll H$ with $k_i \ll X, h_i \ll H$, we get $\sum_{n} |b_n^{(\gamma)}|^2 \precprec H^2 R^2$.

Combining these, the diagonal contribution is bounded by
\begin{align*}
    \mathcal{U}_0(N) \precprec H X^{1/2} + H^2 R^2  \precprec NH^2 R^{-1}.
\end{align*}
by using the assumptions $N \geq X^{\nu+\eta}$, $H=X^{1/2-\nu}$, and $R = X^{\eta_1}$ with $\eta_1=\eta/100$.

\textbf{The off-diagonal term $\mathcal{U}_1(N)$:}
For $n_1 \neq n_2$, we extract the greatest common divisor $n_0=(n_1,n_2)$. We make the change of variables $n_j \mapsto n_0 n_j$, so that now $(n_1,n_2)=1$ and $n_j\sim N/n_0$. The square-free support of $\beta$ makes $n_0,n_1,n_2$ square-free and pairwise coprime throughout. The simultaneous equations $m n_0 n_1 = k_1+h_1$ and $m n_0 n_2 = k_2+h_2$ imply that $n_0$ divides $k_1 + h_1$. Furthermore, eliminating $m$ we get $  n_1 (k_2+h_2) = n_2(k_1+h_1) $, which after rearrangement gives
\begin{align*}
 n_2 k_1 - n_1 k_2 = n_1 h_2 - n_2 h_1. 
\end{align*}
Furthermore, the existence of $m$ solving the simultaneous equations is equivalent to this equation and the congruence condition $k_1 \equiv -h_1 \pmod{n_0} $  by $(n_1,n_2)=(n_0,n_1)=1$. Thus, we have
\begin{align*}
 \mathcal{U}_1(N) 
&=\sum_{n_0 \ll N} \sum_{\substack{(n_1,n_2)=1 \\ n_1n_2 \neq 1}} \beta_{n_0 n_1} \overline{ \beta_{n_0 n_2} }  \sum_{h_1,h_2} \gamma_{h_1} \overline{\gamma_{h_2}} \sum_{\substack{k_1,k_2  \\ n_2 k_1  -n_1 k_2 = n_1 h_2 -  n_2 h_1 \\ k_1 \equiv -h_1 \pmod{n_0} }} \sigma^{\flat}(k_1)\sigma^{\flat}(k_2).
\end{align*}
Recall the definition of $\mathcal{Q}_{i,j}$ from \eqref{eq:Rijdef}. We define 
\begin{align*}
    \mathcal{Q}(u ;\mathbf{n},\kappa) := \sum_{\substack{k_1,k_2  \\ n_2 k_1  -n_1 k_2 =  u \\ k_1 \equiv \kappa \pmod{n_0} }} \sigma^{\flat}(k_1)\sigma^{\flat}(k_2) = \sum_{i,j\in\{0,1\}}(-1)^{i+j}  \mathcal{Q}_{i,j}(u;\mathbf{n},\kappa),
\end{align*}
where we expanded $\sigma^\flat(k) = \sigma_1(k) - \sigma_0(k)$.
By grouping the sum by the residue class $\kappa \equiv -h_1 \pmod{n_0}$, we obtain 
\begin{align*}
    \mathcal{U}_1(N) &= \sum_{n_0 \ll N}  \sum_{\kappa \pmod{n_0}}  \sum_{\substack{h_1,h_2\\ h_1\equiv -\kappa \pmod{n_0}}}  \gamma_{h_1} \overline{\gamma_{h_2}} \sum_{\substack{(n_1,n_2)=1 \\ n_1n_2 \neq 1}} \beta_{n_0 n_1} \overline{ \beta_{n_0 n_2} }  \mathcal{Q}(n_1 h_2 -  n_2 h_1 ;\mathbf{n},\kappa) \\
    &=: \sum_{n_0 \ll N} \sum_{\kappa \pmod{n_0}}  \mathcal{U}_1(N,n_0,\kappa).
\end{align*}
Our next move is to apply Cauchy-Schwarz with $n_1,n_2$ inside to remove the $\gamma_{h}$ weights. However, if $n_0$ is almost of maximal size, then there is not enough mass to do this.  We therefore need to handle separately pseudo-diagonals and we handle trivially $n_0 > R^3$. 

\subsection{Bounding pseudo-diagonals $n_0 > R^3$}
For the off-diagonal terms where the greatest common divisor satisfies $n_0 > R^3$ we use the following bound.
\begin{lemma}[Density bound] \label{lem:large_n0_bound}
    For fixed $n_0$ and $\kappa \pmod{n_0}$, we have 
    \begin{align*}
        \mathcal{U}_1(N,n_0,\kappa) \precprec \frac{R^2 N H^2}{n_0^3}. 
    \end{align*}
\end{lemma}
\begin{proof}
    We recall the definition of the off-diagonal sum, where $n_1, n_2 \sim N/n_0$ and $(n_1, n_2) = 1$:
    \begin{align*}
        \mathcal{U}_1(N,n_0,\kappa) = \sum_{\substack{n_1,n_2 \\ n_1 n_2 \neq 1}} \beta_{n_0 n_1} \overline{\beta_{n_0 n_2}} \sum_{\substack{h_1,h_2 \leq H \\ h_1 \equiv -\kappa (n_0)}} \gamma_{h_1} \overline{\gamma_{h_2}} \mathcal{Q}(n_1 h_2 - n_2 h_1 ; \mathbf{n}, \kappa).
    \end{align*}
  Since $\beta$ is supported on square-free numbers, we have $(n_0,n_1n_2)=1$.  Taking absolute values and using the divisor bound, we get
    \begin{align*}
        \mathcal{U}_1(N,n_0,\kappa) \precprec \sum_{\substack{n_1, n_2 \sim N/n_0 \\ (n_1,n_2)=(n_0,n_1n_2)=1 \\ n_1n_2 \neq 1    }} \sum_{|u| \ll H N/n_0} S(u, \mathbf{n},\kappa) |\mathcal{Q}(u; \mathbf{n}, \kappa)|,
    \end{align*}
    where $S(u, \mathbf{n},\kappa) $ is the number of pairs $(h_1, h_2) \in [1, H]^2$ satisfying $n_1 h_2 - n_2 h_1 = u$ and $h_1 \equiv -\kappa \pmod{n_0}$.  We have since $H>N$ and $(n_0,n_1)=1$
    \begin{align*}
  S(u, \mathbf{n},\kappa)  \ll \sum_{h \leq H} \mathbf{1}_{ \substack{h \equiv -u \overline{n_2} \pmod{n_1} \\ h \equiv -\kappa \pmod{n_0}}}  \ll \frac{H}{N}.
  \end{align*}
Therefore,
\begin{align*}
          \mathcal{U}_1(N,n_0,\kappa)    \precprec \frac{H}{N} \sum_{n_1,n_2\sim N/n_0} \sum_{|u|\ll HN/n_0}|\mathcal Q(u;\mathbf n,\kappa)|.
    \end{align*}

    The $L^1$ estimates proved in Lemma \ref{le:L2bound}, with $U\asymp HN/n_0\gg H\gg X^{2\eps}N^2$, give
    \begin{align*}
        \sum_{|u| \ll HN/n_0} |\mathcal{Q}(u; \mathbf{n}, \kappa)| \precprec \frac{R^2H}{n_0}.
    \end{align*}
 This gives
    \begin{align*}
        |\mathcal{U}_1(N,n_0,\kappa)| &\precprec \frac{R^2 N H^2}{n_0^3},
    \end{align*}
   completing the proof of the lemma.
\end{proof}

We now apply this lemma to  the off-diagonal contribution for $n_0 > R^3$.  Summing the bound over the $n_0$ and the residue classes $\kappa \pmod{n_0}$ to get
\begin{align*}
    \sum_{n_0 > R^3} \sum_{\kappa \pmod{n_0}} \mathcal{U}_1(N, n_0, \kappa) &\precprec N H^2  R^2\sum_{n_0 > R^3} \frac{1}{n_0^2} \precprec N H^2 R^{-1},
\end{align*}
which suffices for the goal \eqref{eq:bound1}.

\subsection{Second Cauchy-Schwarz for $n_0 \leq R^3$}
We anticipate the second Cauchy-Schwarz inequality by inserting a  smooth weight. Since $h_j \leq H$ and $n_j \sim N/n_0$, we have $|n_1 h_2 - n_2 h_1| \leq 2H N / n_0$. Setting $U = H N / n_0$, we can then introduce a smooth weight $\psi\left(\frac{n_1 h_2 - n_2 h_1}{U}\right)$, where $0\leq\psi\leq1$, $\psi=1$ on $[-2,2]$, and $\psi$ is supported on $[-4,4]$. Then for $n_0\leq R^3$
\begin{align*}
   & \mathcal{U}_1(N, n_0, \kappa) = \sum_{\substack{h_1, h_2 \leq H \\ h_1 \equiv -\kappa \pmod{n_0}}}  \gamma_{h_1} \overline{\gamma_{h_2}} \sum_{\substack{(n_1,n_2)=1 }} \beta_{n_0 n_1} \overline{ \beta_{n_0 n_2} }  \mathcal{Q}(n_1 h_2 - n_2 h_1; \mathbf{n}, \kappa) \\ &= \sum_{\substack{h_1, h_2 \leq H \\ h_1 \equiv -\kappa \pmod{n_0}}}  \gamma_{h_1} \overline{\gamma_{h_2}}\sum_{\substack{(n_1,n_2)=1 }} \beta_{n_0 n_1} \overline{ \beta_{n_0 n_2} }   \psi\left(\frac{n_1 h_2 - n_2 h_1}{U}\right) \mathcal{Q}(n_1 h_2 - n_2 h_1; \mathbf{n}, \kappa).
\end{align*}
By applying the Cauchy-Schwarz inequality to the variables $h_1, h_2$ we get
\begin{align*}
     \mathcal{U}_1(N,n_0,\kappa) \precprec& \frac{H}{n_0^{1/2}}  \mathcal{V}(H,N,n_0,\kappa)^{1/2}, \\
     \mathcal{V}(H,N,n_0,\kappa) =& \sum_{\substack{n_1,n_2,n_3,n_4 \\ (n_1,n_2)=(n_3,n_4)=1 }}\beta_{n_0 n_1} \overline{ \beta_{n_0 n_2} } \overline{ \beta_{n_0 n_3} } \beta_{n_0 n_4}  \sum_{h_1,h_2\in\Z} \psi\bigg(\frac{h_1 }{UR^{-4}}\bigg) \psi\bigg(\frac{n_1h_2-n_2h_1}{U}\bigg)\psi\bigg(\frac{n_3h_2-n_4h_1}{U}\bigg)  \\
    & \times \mathcal{Q}(n_1 h_2 -  n_2 h_1 ;n_0,n_1,n_2,\kappa)\overline{\mathcal{Q}(n_3 h_2 -  n_4 h_1 ;n_0,n_3,n_4,\kappa)}.
\end{align*}
Here $UR^{-4}/H=N/(n_0R^4)\gg1$. Before expanding the non-negative square, we extend $(h_1,h_2)$ to $\Z^2$, remove the congruence, and insert the crude majorant $\psi(h_1/(UR^{-4}))$. To obtain the bound \eqref{eq:bound1} after summing over $n_0$ and $\kappa$, it then suffices to show that
\begin{align} \label{eq:bound2}
    \mathcal{V}(H,N,n_0,\kappa)  \precprec \frac{N^2 H^2}{n_0^3 R}.
\end{align}
We partition $\mathcal{V} = \mathcal{V}_0 + \mathcal{V}_1$, where $\mathcal{V}_0(H,N,n_0,\kappa)$ isolates the diagonal terms satisfying $n_1 n_4 = n_2 n_3$, and $\mathcal{V}_1$ handles the off-diagonal terms.

\textbf{The diagonal term $\mathcal{V}_0$:}
The condition $n_1 n_4 = n_2 n_3$, paired with $(n_1, n_2) = (n_3, n_4) = 1$, forces $n_1 = n_3$ and $n_2 = n_4$.  Thus, the diagonal collapses to a double sum over $n_1, n_2 \sim N/n_0$:
\begin{align*}
    \mathcal{V}_0(H,N,n_0,\kappa) =&\sum_{\substack{n_1, n_2 \\ (n_1,n_2)=1  }} |\beta_{n_0 n_1}|^2 |\beta_{n_0 n_2}|^2 \sum_{h_1,h_2\in\Z}\psi\bigg(\frac{h_1}{UR^{-4}}\bigg) \psi\bigg(\frac{n_1h_2-n_2h_1}{U}\bigg)^2 |\mathcal{Q}(n_1h_2-n_2h_1;n_0,n_1,n_2,\kappa)|^2.
\end{align*}

For a fixed pair $(n_1, n_2)$, we substitute $u = n_1 h_2 - n_2 h_1$. Since $H > N$, the number of integer pairs with $|h_1|\leq4UR^{-4}$ mapping to a fixed $u$ is $\ll1+Un_0/(R^4N)\ll H/R^4$.  By our $L^2$ bounds (Lemma \ref{le:L2bound}) we have $\sum_{|u|\leq4U}|\mathcal Q(u)|^2\precprec R^2U/N$. Substituting $U = H N / n_0$ and summing over $n_1,n_2$ we get
\begin{align*}
    \mathcal{V}_0(H,N,n_0,\kappa) \precprec \sum_{\substack{n_1, n_2 \ll N/n_0 \\ (n_1,n_2)=1  }} \frac{R^2 H}{n_0} \frac{H}{R^4} \precprec \frac{ N^2 H^2}{n_0^3 R^2}.
\end{align*}
which is sufficient for our target \eqref{eq:bound2}. Note that without any restriction for the size of $h_1$ we would sum over infinitely many pairs $h_1,h_2$.
\subsection{Off-diagonal terms $n_1 n_4 \neq n_2 n_3$}
We now evaluate the off-diagonal terms where the determinant $\Delta_{\mathbf{n}} = n_1 n_4 - n_2 n_3 \neq 0$. We make a linear change of variables $(u_1,u_2)=(n_1 h_2 -  n_2 h_1, n_3 h_2 -  n_4 h_1)$. Defining the matrix
\begin{align*}
    A=\begin{pmatrix} -n_{2} & n_{1} \\ -n_4 & n_{3} \end{pmatrix},
\end{align*}
we have $A \binom{h_1}{h_2}= \binom{u_1}{u_2}$ and $\Delta_{\mathbf{n}} = \det A$. Inverting this yields
\begin{align*}
      \begin{pmatrix} h_1 \\ h_2 \end{pmatrix} &= A^{-1}\begin{pmatrix} u_1 \\ u_2 \end{pmatrix}  =  \frac{1}{\Delta_{\mathbf{n}} } \begin{pmatrix} n_{3} u_1 - n_{1} u_2 \\ n_{4} u_1 - n_{2} u_2 \end{pmatrix}.   
\end{align*}
Because $h_1, h_2$ must be integers, this imposes the congruence conditions 
\begin{align}
    n_{3}u_1  &\equiv n_{1}u_2  \pmod{|\Delta_{\mathbf{n}}|}, \label{eq:cong1} \\
    n_{4}u_1 &\equiv n_2 u_2 \pmod{|\Delta_{\mathbf{n}}|}. \label{eq:cong2}
 \end{align}
We claim that, for $(n_1,n_2)=(n_3,n_4)=1$, this system collapses into a single congruence
\begin{align*}
    u_1 \equiv \theta_\mathbf{n} u_2 \pmod{|\Delta_{\mathbf{n}}|}
\end{align*}
for some integer $\theta_\mathbf{n}$ satisfying $(\theta_\mathbf{n},|\Delta_{\mathbf{n}}|)=1$. 

To see this, observe that by B\'ezout's identity there exist integers $x, y$ such that $x n_3 + y n_4 = 1$. Multiplying \eqref{eq:cong1} by $x$ and \eqref{eq:cong2} by $y$, and summing the results yields:
\begin{align*}
    u_1 &\equiv (x n_1 + y n_2) u_2 \pmod{|\Delta_{\mathbf{n}}|}.
\end{align*}
Thus, we may  set $\theta_\mathbf{n} = x n_1 + y n_2$. To verify that $(\theta_\mathbf{n}, |\Delta_{\mathbf{n}}|) = 1$, we observe that
\begin{align*}
    n_3 \theta_\mathbf{n} + y \Delta_{\mathbf{n}} &= n_3(x n_1 + y n_2) + y(n_1 n_4 - n_2 n_3) = n_1(x n_3 + y n_4) = n_1, \\
    n_4 \theta_\mathbf{n} - x \Delta_{\mathbf{n}} &= n_4(x n_1 + y n_2) - x(n_1 n_4 - n_2 n_3) = n_2(x n_3 + y n_4) = n_2.
\end{align*}
Thus, if some prime $p$ divides both $\theta_\mathbf{n}$ and $|\Delta_{\mathbf{n}}|$, then $p \mid n_1 $ and $p \mid n_2$, which contradicts our assumption that $(n_1, n_2) = 1$. Thus, we have $(\theta_\mathbf{n}, |\Delta_{\mathbf{n}}|) = 1$.

Making the above substitution, we obtain
\begin{align*}
     \mathcal{V}_1(H,N,n_0,\kappa) =& \sum_{\substack{n_1,n_2,n_3,n_4 \\ (n_1,n_2)=(n_3,n_4)=1 \\ \Delta_{\mathbf{n}} \neq 0}}\beta_{n_0 n_1} \overline{ \beta_{n_0 n_2} } \overline{ \beta_{n_0 n_3} } \beta_{n_0 n_4}   \sum_{\substack{u_1,u_2 \\  u_1 \equiv \theta_\mathbf{n} u_2 (|\Delta_{\mathbf{n}}|)}} \psi \bigg(\frac{u_1}{U} \bigg)  \psi \bigg(\frac{u_2}{U} \bigg)\\
    & \times \psi\bigg(\frac{n_3u_1-n_1u_2 }{\Delta_{\mathbf{n}}UR^{-4}}\bigg) \mathcal{Q}(u_1; n_0,n_1,n_2,\kappa)\overline{\mathcal{Q}(u_2 ;n_0,n_3,n_4,\kappa)}.
\end{align*}
We now  remove the weight $\psi\bigg(\frac{n_3u_1-n_1u_2 }{\Delta_{\mathbf{n}}UR^{-4}}\bigg)$ by Fourier inversion
\begin{align*}
    \psi\bigg(\frac{n_3u_1-n_1u_2 }{\Delta_{\mathbf{n}}UR^{-4}}\bigg) = \int \widehat{\psi_\mathbf{n}}(\xi) e(\xi n_3 u_1)  e(-\xi n_1 u_2)  \d \xi, \quad  \psi_\mathbf{n}(x) =      \psi\bigg(\frac{x}{\Delta_{\mathbf{n}}UR^{-4}}\bigg),
\end{align*}
We use dyadic intervals $E/2<|\Delta_{\mathbf n}|\leq E$, abbreviated $|\Delta_{\mathbf n}|\sim E$. Integration by parts gives
\begin{align} \label{eq:psiYdecay}
    \widehat{\psi_\mathbf{n}}(\xi) \ll_C Y_E (1+   |\xi| Y_E )^{-C}, \quad Y_E= E UR^{-4} = E H (N/n_0) R^{-4}.
\end{align}
Therefore,  it suffices to bound
\begin{align*}
   & \sup_{E \ll (N/n_0)^2}  \sum_{\substack{n_1,n_2,n_3,n_4 \\ (n_1,n_2)=(n_3,n_4)=1 \\ |\Delta_{\mathbf{n}}| \sim E}}|\beta_{n_0 n_1}  \beta_{n_0 n_2}   \beta_{n_0 n_3}  \beta_{n_0 n_4} |  \int  |\widehat{\psi_\mathbf{n}}(\xi)| \\&\bigg|\sum_{\substack{u_1,u_2 \\  u_1 \equiv \theta_\mathbf{n} u_2 (|\Delta_{\mathbf{n}}|)}} \psi \bigg(\frac{u_1}{U} \bigg)  \psi \bigg(\frac{u_2}{U} \bigg)e(\xi n_3 u_1)  e(-\xi n_1  u_2) \mathcal{Q}(u_1; n_0,n_1,n_2,\kappa)\overline{\mathcal{Q}(u_2 ;n_0,n_3,n_4,\kappa)} \bigg| \d \xi.
\end{align*}
We expand the congruence $u_1 \equiv \theta_\mathbf{n} u_2 \pmod{|\Delta_{\mathbf{n}}|}$ using orthogonality as $$\frac{1}{|\Delta_{\mathbf{n}}|} \sum_{a \pmod{|\Delta|}} e_{|\Delta_{\mathbf{n}}|}(a u_1)e_{|\Delta_{\mathbf{n}}|}(-a \theta_\mathbf{n} u_2).$$ We group the frequencies by their conductor $d = |\Delta_{\mathbf{n}}| / (a, |\Delta_{\mathbf{n}}|)$ and set $b = a / (a, |\Delta_{\mathbf{n}}|)$. By the triangle inequality, noting that $|\Delta_{\mathbf{n}}| \ll (N/n_0)^2$, we obtain 
\begin{align*}
 \sup_{\substack{E \ll (N/n_0)^2 }}  &\sum_{d\leq E} \sum_{\substack{n_1,n_2,n_3,n_4\sim N/n_0 \\ (n_1,n_2)=(n_3,n_4)=1 \\ d \mid \Delta_{\mathbf{n}},\ |\Delta_{\mathbf{n}}| \sim E  }} \frac{1}{|\Delta_{\mathbf{n}}|} \sum_{\substack{b\pmod{d}^\times}}  \int |\widehat{\psi_\mathbf{n}}(\xi)|  \\
& \times \bigg| \sum_{u_1} \psi \bigg(\frac{u_1}{U} \bigg)e(\xi n_3 u_1) \mathcal{Q}(u_1; n_0,n_1,n_2,\kappa) e_d(bu_1) \bigg| \\
& \times\bigg| \sum_{u_2}\psi \bigg(\frac{u_2}{U} \bigg)  e(\xi n_1 u_2)\mathcal{Q}(u_2; n_0,n_3,n_4,\kappa) e_d(b \theta_{\mathbf{n}} u_2) \bigg| \d \xi
 \end{align*}
 Applying Cauchy-Schwarz we get
 \begin{align*}
   |\mathcal{V}_1(H,N,n_0,\kappa)|\precprec\sqrt{\Sigma_{1,2}\Sigma_{3,4}}
\end{align*}
where
\begin{align*}
\Sigma_{1,2}= & \sup_{\substack{E \ll (N/n_0)^2}}\sum_{d\leq E} \sum_{\substack{n_1,n_2,n_3,n_4 \sim N/n_0\\ (n_1,n_2)=(n_3,n_4)=1 \\ d \mid \Delta_{\mathbf{n}},\ |\Delta_{\mathbf{n}}| \sim E  }}  \frac{1}{|\Delta_{\mathbf{n}}|}\sum_{\substack{b\pmod{d}^\times}} \int  |\widehat{\psi_\mathbf{n}}(\xi)|  \\
&\times \bigg| \sum_{u}\psi \bigg(\frac{u}{U} \bigg) e(\xi n_3 u)  \mathcal{Q}(u;n_0,n_1,n_2,\kappa) e_d(b  u) \bigg|^2 \d \xi
\end{align*}
and
\begin{align*}
      \Sigma_{3,4}  =& \sup_{\substack{E \ll (N/n_0)^2}} \sum_{ d\leq E} \sum_{\substack{n_1,n_2,n_3,n_4\sim N/n_0 \\ (n_1,n_2)=(n_3,n_4)=1 \\ d \mid \Delta_{\mathbf{n}},\ |\Delta_{\mathbf{n}}| \sim E }}  \frac{1}{|\Delta_{\mathbf{n}}|}\sum_{\substack{b\pmod{d}^\times}} \int |\widehat{\psi_\mathbf n}(\xi)| \\
      & \times\bigg| \sum_{u}\psi \bigg(\frac{u}{U} \bigg) e(\xi n_1 u)\mathcal{Q}(u;n_0,n_3,n_4,\kappa) e_d(b \theta_{\mathbf{n}} u) \bigg|^2 \d \xi \\
        =& \sup_{\substack{E \ll (N/n_0)^2}} \sum_{d\leq E} \sum_{\substack{n_1,n_2,n_3,n_4 \sim N/n_0 \\ (n_1,n_2)=(n_3,n_4)=1 \\ d \mid \Delta_{\mathbf{n}},\ |\Delta_{\mathbf{n}}|\sim E  }}  \frac{1}{|\Delta_{\mathbf{n}}|}\sum_{\substack{b\pmod{d}^\times}} \int  |\widehat{\psi_\mathbf{n}}(\xi)| \\
        & \times \bigg| \sum_{u}\psi \bigg(\frac{u}{U} \bigg) e(\xi n_1 u)  \mathcal{Q}(u;n_0,n_3,n_4,\kappa) e_d(b  u) \bigg|^2 \d \xi.
\end{align*}
The change of variables $b \theta_{\mathbf{n}} \mapsto b$ in the additive character for $\Sigma_{3,4}$ is  justified because $(\theta_{\mathbf{n}}, |\Delta_{\mathbf{n}}|) = 1$. This shows that the two factors $\Sigma_{1,2},\Sigma_{3,4}$ are symmetric and it suffices to prove that
\begin{align} \label{eq:bound3}
    \Sigma_{1,2}\precprec\frac{N^2H^2}{n_0^3R}.
\end{align}

First, by making the change of variables $\xi \mapsto \xi / n_3$, using \eqref{eq:psiYdecay}, and taking the integral outside, we see that
\begin{align*}
\Sigma_{1,2} \ll  & \sup_{\substack{E \ll (N/n_0)^2 }}\sum_{d\leq E} \sum_{\substack{n_1,n_2,n_3,n_4 \sim N/n_0\\ (n_1,n_2)=(n_3,n_4)=1 \\ d \mid \Delta_{\mathbf{n}},\ |\Delta_{\mathbf{n}}| \sim E  }}  \frac{1}{|\Delta_{\mathbf{n}}|}\sum_{\substack{b\pmod{d}^\times}} \int \frac{Y_E n_0}{N} (1+   |\xi| Y_E n_0/N )^{-C}   \\
&\times \bigg| \sum_{u}\psi \bigg(\frac{u}{U} \bigg)e(\xi u)  \mathcal{Q}(u;n_0,n_1,n_2,\kappa) e_d(b  u) \bigg|^2 \d \xi \\
\ll &  \sup_{\substack{E \ll (N/n_0)^2 \\ |\xi| \leq X^\eps \frac{ R^4 N }{ n_0 E U }}}\sum_{d\leq E} \sum_{\substack{n_1,n_2,n_3,n_4 \sim N/n_0\\ (n_1,n_2)=(n_3,n_4)=1 \\ d \mid \Delta_{\mathbf{n}},\ |\Delta_{\mathbf{n}}| \sim E  }}  \frac{1}{|\Delta_{\mathbf{n}}|}\sum_{\substack{b\pmod{d}^\times}}   \\
&\times \bigg| \sum_{u}\psi \bigg(\frac{u}{U} \bigg)e\bigg(\xi u \bigg)  \mathcal{Q}(u;n_0,n_1,n_2,\kappa) e_d(b  u) \bigg|^2 + O_{\eps,C}(X^{-C}).
\end{align*}

For a fixed $d$ and fixed $n_1,n_2$ with $(n_1,n_2)=1$, we count the possible values of $\Delta$ and the valid pairs $(n_3, n_4)$. The equation $n_1 n_4 - n_2 n_3 = \Delta$ uniquely determines $n_3$ for a given $n_4$ and $\Delta$ and we have
\begin{align*}
\sum_{\substack{0<|\Delta| \ll (N/n_0)^2 \\d \mid \Delta }} \frac{1}{|\Delta|}\sum_{\substack{n_3,n_4 \sim N/n_0 \\ (n_3,n_4)=1 \\ n_1 n_4 - n_2 n_3 = \Delta  }} 1 &\leq \sum_{\substack{0<|\Delta| \ll (N/n_0)^2 \\d \mid \Delta }}\frac{1}{|\Delta|} \sum_{\substack{n_4 \sim N/n_0  \\ n_1 n_4  \equiv \Delta \pmod{n_2} }} 1  \\
&\ll \sum_{\substack{0<|\Delta| \ll (N/n_0)^2 \\d \mid \Delta }} \frac{1}{|\Delta|} \ll \frac{\log X}{d}.
\end{align*}
Therefore,
\begin{align*}
    \Sigma_{1,2} \precprec \sup_{\substack{E \ll (N/n_0)^2 \\ |\xi| \leq X^\eps \frac{ R^4 N }{ n_0 E U }}} \sum_{\substack{n_1,n_2\sim N/n_0 \\ (n_1,n_2)=1  }}   \sum_{d\leq E }  \frac{1}{d} \sum_{\substack{b\pmod{d}^\times}} \bigg| \sum_{u}\psi \bigg(\frac{u}{U} \bigg) e(\xi u) \mathcal{Q}(u;n_0,n_1,n_2,\kappa) e_d(b  u) \bigg|^2.
\end{align*}
We split the global sum into two parts according to the size of the conductor $d$. We define the conductor threshold $D=X^{\eta/16} N/n_0$, which is chosen so that $d > D$ can be handled via the large sieve. We write
\begin{align*}
       \Sigma_{1,2}\precprec\Sigma_{\leq D}+\Sigma_{>D}.
\end{align*}

  \subsection{Large conductors $d > D$}
We split the sum over $d$ dyadically into ranges $d \sim Q \in [D,(4N/n_0)^2]$. For fixed $n_1,n_2$, we employ the additive Large sieve inequality (Lemma \ref{lem:LaSi}) with the $L^2$-bound from Lemma \ref{le:L2bound}, which gives $\sum_{|u|\leq4U}|\mathcal Q(u)|^2\precprec R^2U/N$ and so
\begin{align*}
 \frac{1}{Q} \sum_{d\sim Q}  \sum_{\substack{b\pmod{d}^\times}} \bigg| \sum_{u}\psi \bigg(\frac{u}{U} \bigg) e(\xi u)  \mathcal{Q}(u;n_0,n_1,n_2,\kappa) e_d(b  u) \bigg|^2 &\precprec    \frac{1}{Q} (Q^2+U) \frac{ R^2 U}{N}  \\
   &\precprec \left(  \frac{N^2}{n_0^2} + \frac{H N }{n_0 D}  \right) \frac{R^2H}{n_0}  
\end{align*}
since $ Q \in [D,(4N/n_0)^2]$ and $U= H N/n_0$. 
Summing over the $(N/n_0)^2$ pairs of $n_1, n_2$, we obtain
\begin{align*}
   \Sigma_{> D} &\precprec   \frac{R^2 N^4 H}{n_0^5} + \frac{R^2 N^3 H^2}{n_0^4 D}.
\end{align*}
Substituting our choice for the conductor threshold $D = X^{\eta/16} N/n_0$ we get
\begin{align*}
    \Sigma_{> D} \precprec \frac{R^2 N^4 H}{n_0^5} + X^{-\eta/32} \frac{R^2 N^2 H^2}{n_0^3}.
\end{align*}
The second term is sufficient for the target \eqref{eq:bound3} since $R=X^{\eta_1}$ with $\eta_1=\eta/100$. The first term is sufficient  for the target \eqref{eq:bound3}  by the Type II range assumption $N \leq X^{1/4-\nu/2-\eta}$ in Proposition \ref{prop:typeII}.

\subsection{Small conductors $d \leq D$}
Fix $0<\eps<\eta/1000$, and write $N'=N/n_0$.
In the sums below we retain the square-free support inherited from
$\beta_n$. For brevity, $\sum^\ast_{n_1,n_2\sim N'}$ denotes the
sum over square-free $n_1,n_2\sim N'$ satisfying
$(n_1,n_2)=(n_0,n_1n_2)=1$.
After rescaling the Fourier variable, we have
\begin{align*}
    \Sigma_{\leq D}
    \precprec{}&
    \sup_{\substack{1\leq E\ll D\\
        |\xi|\leq X^\eps R^4N/(n_0E)}}
    \sum_{d\leq\min(E,D)}\frac{1}{d}
    \sum_{n_1,n_2\sim N'}^\ast
    \sum_{b\pmod{d}^\times}\\
    &\qquad\times
    \left|\sum_u\psi\left(\frac{u}{U}\right)
        e\left(\xi\frac{u}{U}\right)
        \mathcal{Q}(u;n_0,n_1,n_2,\kappa)e_d(bu)\right|^2.
\end{align*}
Here, if the original determinant parameter exceeds $D$, replacing
it by $D$ only enlarges the permitted range of $\xi$.
For each admissible pair $(E,\xi)$, put
\[
    \psi_\xi(x)=\psi(x)e(\xi x),\qquad T=1+|\xi|.
\]
For $d\leq\min(E,D)$ we have
\[
    Td\leq d+X^\eps\frac{R^4N}{n_0}\frac{d}{E}
    \ll D+X^\eps R^4N'\ll D,
\]
since $D=X^{\eta/16}N'$ and $R=X^{\eta/100}$.
The upper bound $U\leq NX/n_0$ holds since $H<X$. The other hypotheses of Lemma \ref{le:voronoi} follow from the Type II
range:
\[
    \frac{X^\eps TdNR}{U}
    \ll X^{\eps+\eta/16}\frac{NR}{H}=o(1),
    \qquad
    \frac{X^\eps dn_0}{X^{1/2}}
    \ll X^{\eps+\eta/16}\frac{N}{X^{1/2}}=o(1).
\]

We apply Lemmas \ref{le:voronoi} and \ref{lem:phasedecayvoronoi} with $\psi_\xi$ and use the
notations of Lemmas \ref{lem:main_term_cancellation} and
\ref{lem:exp_sum_bound}. We bound the zero-frequency integrals
by $\mathcal{I}_0\ll Un_0/N=H$, and use
\eqref{eq:voronoi_twisted_decay} for the non-zero frequencies:
\[
    \left|\widehat{\Psi}^{\,\xi}_{1,1}
       \left(\frac{t_1}{dn_0},\frac{t_2}{dn_0}\right)\right|
    \precprec\frac{H}{T}.
\]
Consequently, suppressing negligible errors, we obtain
\begin{align*}
    &\left|\sum_u\psi\left(\frac{u}{U}\right)
       e\left(\xi\frac{u}{U}\right)
       \mathcal{Q}(u;n_0,n_1,n_2,\kappa)e_d(bu)\right|\\
    &\qquad\precprec
    \frac{H}{d^2n_0^2}
    \left(|\mathcal{M}(b,d)|+
       \frac{1}{T}\sum_{(t_1,t_2)\in\mathcal{T}^\ast}
          |\mathcal{E}(b,d,t_1,t_2)|\right).
\end{align*}
Here $\mathcal{T}^\ast$ depends on $\xi,d,n_1,n_2$ through $T$.
Using $|A+B|^2\leq2(|A|^2+|B|^2)$ and $\varphi(d)/d\leq1$, we get
\begin{align}\label{eq:voronoi_squared}
    \Sigma_{\leq D}
    \precprec{}&
    \sup_{\substack{1\leq E\ll D\\
        |\xi|\leq X^\eps R^4N/(n_0E)}}
    \frac{H^2}{n_0^4}
    \sum_{d\leq\min(E,D)}\frac{1}{d^4}
    \sum_{n_1,n_2\sim N'}^\ast\notag\\
    &\quad\times\left(
       \sup_{b\pmod{d}^\times}|\mathcal{M}(b,d)|^2+
       \frac{1}{T^2}\sup_{b\pmod{d}^\times}
       \left(\sum_{(t_1,t_2)\in\mathcal{T}^\ast}
          |\mathcal{E}(b,d,t_1,t_2)|\right)^2\right).
\end{align}

\textbf{Main Term Contribution:}
By Lemma \ref{lem:main_term_cancellation} and the pairwise
coprimality of $n_0,n_1,n_2$,
\[
    |\mathcal M(b,d)|^2
    \precprec \mathbf{1}_{d>R}n_0^3d^2(d,n_0n_1n_2).
\]
For any positive integer $A$,
\[
    \sum_{d>R}\frac{(d,A)}{d^2}
    \leq \sum_{c\mid A}c\sum_{\substack{d>R\\c\mid d}}\frac1{d^2}
    \ll \frac{\tau(A)}R.
\]
Consequently, the main terms in \eqref{eq:voronoi_squared}
contribute
\[
    \precprec \frac{H^2}{n_0}
    \sum_{n_1,n_2\sim N'}^\ast
    \sum_{R<d\leq D}\frac{(d,n_0n_1n_2)}{d^2}
    \precprec \frac{N^2H^2}{n_0^3R},
\]
which suffices for \eqref{eq:bound3}.

\textbf{Error Term Contribution:}
By Remark \ref{rem:dual}, every dual frequency satisfies
\[
    |n_1t_1^2-n_2t_2^2|\leq M,
    \qquad
    M:=X^{2\eps}\frac{T^2d^2N^2}{U}
      =X^{2\eps}\frac{T^2d^2n_0N}{H}.
\]
The bound $Td\ll D$ gives
\[
    \frac{M}{N'}
    \ll X^{2\eps}\frac{D^2n_0^2}{H}
    =X^{2\eps+\eta/8}\frac{N^2}{H}
    <X^{2\eps-15\eta/8}.
\]
Thus $M\leq N'/10$ for sufficiently large $X$.
If $M<1$, then $\mathcal{T}^\ast$ is empty, since for square-free
coprime $n_1,n_2>1$ the integer $n_1t_1^2-n_2t_2^2$ is non-zero
whenever $(t_1,t_2)\neq(0,0)$.
For $M\geq1$, apply Lemma \ref{lem:number_of_dual}, interchanging
$n_1,n_2$ if necessary and replacing $M$ by $\max(M,10)$ when needed.
The bound on $|t_1|,|t_2|$ in Remark \ref{rem:dual} is polynomial
in $X$, so the logarithmic loss is $O(\log X)$. Therefore, in all cases,
\begin{equation}\label{eq:small_conductor_dual_count}
   |\mathcal{T}^\ast|^2\precprec M
    =X^{2\eps}\frac{T^2d^2n_0N}{H}.
\end{equation}

By Lemma \ref{lem:exp_sum_bound},
\[
    \sup_{b\pmod{d}^\times}|\mathcal{E}(b,d,t_1,t_2)|
    \precprec n_0d(d,n_0)\sqrt{(d,n_1)(d,n_2)}.
\]
Hence the dual-frequency contribution to \eqref{eq:voronoi_squared}
is bounded by
\begin{align*}
    \Sigma_{\mathrm{dual}}
    &\precprec
    \sup_{E,\xi}\frac{H^2}{T^2n_0^2}
    \sum_{d\leq\min(E,D)}\frac{(d,n_0)^2}{d^2}
    \sum_{n_1,n_2\sim N'}^\ast |\mathcal{T}^\ast|^2(d,n_1)(d,n_2)\\
    &\precprec
    X^{2\eps}\frac{HN}{n_0}
    \sum_{d\leq D}(d,n_0)^2
    \left(\sum_{n\sim N'}(d,n)\right)^2.
\end{align*}
The supremum is over the admissible $(E,\xi)$ above. In the second
line we first used \eqref{eq:small_conductor_dual_count}, canceling
the factors $T^2$ and $d^2$, and only then enlarged the range to
$d\leq D$ and discarded the restrictions on $n_1,n_2$.
The elementary estimates
\[
    \sum_{n\sim L}(d,n)\ll L\tau(d),
    \qquad
    \sum_{d\leq D}(d,n_0)^2
    \leq D\sum_{c\mid n_0}c
    \leq Dn_0\tau(n_0)
\]
now yield
\begin{align*}
    \Sigma_{\mathrm{dual}}
    &\precprec X^{2\eps}\frac{HN}{n_0}
       \left(\frac{N}{n_0}\right)^2Dn_0=X^{2\eps}\frac{HDN^3}{n_0^2}
     =X^{2\eps+\eta/16}\frac{HN^4}{n_0^3}.
\end{align*}
Since $N^2/H<X^{-2\eta}$, this is sufficient for the target bound \eqref{eq:bound3}. \qed

\section{Harman's sieve preliminaries}
Throughout this section we let $\A^{(h)}$ and $\B^{(h)}$ denote the sequences $a^{(h)}_n$ and $b^{(h)}_n$ for a fixed $h \leq H =X^{1/2-\nu}$. As usual, we denote $P(z) := \prod_{p < z} p$ and for a sequence $\mathcal{C}=(c_n)$ and integer $d \geq 1$, we denote 
\begin{align*}
    S(\mathcal{C}_d,z) := \sum_{(n,P(z))=1} c_{dn}.
 \end{align*}

Let $\eta \in (0,\nu)$ be small. We define the width of the Type II range
\begin{align} \label{eq:gammadef}
    \gamma:= 1/4-3\nu/2 - 2\eta.
\end{align}
We set parameters
\begin{align}  \label{eq:sieveparams}
 V=X^{\nu+\eta},   \quad Z := X^{\gamma}, \quad W=VZ, \quad Y=X/Z^2, \quad U =X^{1/2-\nu-2\eta}, \quad T = X^{1/2-4\nu}
\end{align}
Here $Z=W/V$ matches the width of the Type II range $[V,W]$ and $U = X^{1/2-\eta} V^{-1}$ is chosen so that $UV$ matches the Type I range $X^{1/2-\eta}$.

\subsection{Evaluating sequence $b_n$ over almost primes}
We first require the following lemma to count the model sequence $b_n$ over products of primes.
Let $\omega(u)$ denote the Buchstab function (eg.  \cite[Chapter 1]{harman}), so that by the Prime number theorem for $Y^{\eps} < Z < Y$ 
\[
\sum_{Y< n \leq 2Y} \bm{1}_{(n,P(Z))=1} = (1+o(1)) \omega \left(\frac{\log Y}{\log Z} \right) \frac{Y}{\log Z}.
\]
We have
\begin{align*}
    \omega(u) \,=\begin{cases} 0, &u < 1 \\
1/u, & 1 \leq u < 2 \\
(1+\log(u-1))/u, &2 \leq u \leq 3.
\end{cases}
\end{align*}
and for all $u$ we have $\omega(u) \leq 1$.

To count almost primes over $b_n$, we first need the following "Prime number theorem" for $b_n$
\begin{lemma} \label{lem:pnt_bn}
Let $X^\eta<H\leq X^{1/2+o(1)}$ and let $(\gamma_h)_{h \leq H}$ be complex coefficients. For any piecewise smooth function $G(\alpha_1,\dots,\alpha_k)$ supported on $\alpha_i > \eta$ and $\alpha_1+\cdots+\alpha_k < 1-\eta$,  there exist main terms $M^{(h)}_G$
such that
\begin{align*}
    M^{(h)}_G = \frac{\sqrt{X} I_f }{\log X}  \mathfrak{S}_h \left(\int G(\boldsymbol{\alpha}) \frac{\d \boldsymbol{\alpha}}{\alpha_1 \cdots \alpha_k (1-\sum \alpha_i)}+o(1)\right).
\end{align*}
and we have
\begin{align*}
    \sum_{p_1,\dots,p_k} G\bigg(\frac{\log p_i}{\log X} \bigg)\sum_p b^{(\gamma)}_{pp_1\cdots p_k} &=   \sum_{h \leq H} \gamma_h  M^{(h)}_G + O_C\bigg( \|\gamma\|_\infty H \frac{\sqrt{X}}{(\log X)^C}\bigg).
\end{align*}  
Similarly we  get $k=0$ the expected  main term for $\sum_p b_{p}$.
\end{lemma}

\begin{proof}
By the triangle inequality and the prime number theorem, it suffices to show that for any  $d \leq X^{1-\eta}$ with $(d,P(X^\eta))=1$  and $N = X/d$, we have
\begin{align*}
    \sum_p b_{p d} - 
\frac{\sqrt X}{d}
\sum_{h\le H}\gamma_h\mathfrak S_h
\int_0^\infty
\frac{f(u^2)}
{\log((Xu^2+h)/d)}\d u \ll_C \|\gamma\|_\infty H \frac{N}{\sqrt{X}} \frac{1}{(\log X)^C}.
\end{align*}
Put $w(p)=f(pd/X)/(2\sqrt{pd})$. Replacing the actual weight $f((pd-h)/X)/(2\sqrt{pd-h})$ by $w(p)$ gives an error
\[
 \ll \|\gamma\|_\infty\frac{H^2NR^{3/2}}{\delta X^{3/2}},
\]
which is negligible. Expanding the additive model, we obtain
\begin{align*}
    \sum_p b_{p d} ={}& \sum_{r \leq R} \frac{1}{r} \sum_{\alpha \pmod r^\times} G(\alpha, r) \bigg( \sum_{h \leq H} \gamma_h e_r(\alpha h) \bigg) \bigg( \sum_{p \asymp N} w(p) e_r(-\alpha p d) \bigg)\\
    &+O\left(\|\gamma\|_\infty\frac{H^2NR^{3/2}}{\delta X^{3/2}}\right).
\end{align*}
We let $R_1 = (\log X)^B$ (for a sufficiently large constant $B \geq 2C$). We partition into two regimes, small moduli $r \leq R_1$ and medium moduli $R_1 < r \leq R$. Since $(d,P(X^\eta))=1$, we have $(d,r)=1$.

\textbf{1. Small Moduli ($r \leq R_1$):}
For $r\leq R_1$, Siegel--Walfisz evaluates the prime sum, retaining its logarithmic integral. Lemma \ref{le:modelforsquarefree} gives the local factor $\mathbf{1}_{r\,\mathrm{odd}}(\frac{-h}{r})$ for square-free $r$. Restoring $h$ in the smooth integral at negligible cost, the main term is
\begin{align*}
    \text{MT}_{\leq R_1} = \frac{\sqrt X}{d}
\sum_{h\leq H}\gamma_h
\left(\int_0^\infty
\frac{f(u^2)}
{\log((Xu^2+h)/d)}\d u\right)
\sum_{\substack{r\leq R_1\\r\,\mathrm{odd}}}
\frac{\mu(r)}{\varphi(r)}\left(\frac{-h}{r}\right).
\end{align*}

To complete the series to $\mathfrak{S}_h$, we must bound the contribution from the tail 
$$ \Delta = \sum_{h \leq H} \gamma_h\sum_{\substack{r > R_1\\r\,\mathrm{odd}}} \frac{\mu(r)}{\varphi(r)} \left(\frac{-h}{r}\right).$$
The bounded normalized logarithmic integral may be absorbed into $\gamma_h$. We divide the tail into two regions $\Delta= \Delta_1+\Delta_2$ depending on $R_1 <r \leq \exp(H^{\eta_1})$ and $r > \exp(H^{\eta_1})$. 

For $r > \exp(H^{\eta_1})$ the sum over $r$ exhibits cancellation by the Siegel-Walfisz theorem  and we get 
\begin{align*}
    |\Delta_2| \ll_C \|\gamma\|_\infty\frac{H}{(\log X)^C}.
\end{align*}

For  $R_1 <r \leq \exp(H^{\eta_1})$ we apply the quadratic large sieve (Lemma \ref{lem:quadlargesieve}), noting that the $\log(MN)$ factor is bounded by $H^{\eta_1}$ for large $r$, to get
\begin{align*}
   |\Delta_1|\ll \sum_{R_0=2^j \in [R_1,\exp(H^{\eta_1})]}
   \left|\sum_{h\leq H}\gamma_h
   \sum_{\substack{r\sim R_0\\r\,\mathrm{odd}}}
   \frac{\mu(r)}{\varphi(r)}\left(\frac{-h}{r}\right)\right|
   \ll_C\|\gamma\|_\infty\frac{H}{(\log X)^C}.
    \end{align*}

\textbf{2. Medium Moduli ($R_1 < r \leq R$):}
For the medium moduli, we apply the Cauchy-Schwarz on $h$. Partitioning into dyadic blocks $r \sim R_0 \in [R_1, R]$, and applying Cauchy-Schwarz strictly over $h \leq H$:
\begin{align*}
    |\mathcal{E}_{R_0}| &\leq \bigg( \sum_{h \leq H} |\gamma_h|^2 \bigg)^{1/2} V^{1/2}, \\
    V&:=\sum_{h \leq H} \bigg| \sum_{r \sim R_0} \frac{1}{r} \sum_{\alpha (r)^\times} G(\alpha, r) \sum_{p \asymp N} w(p) e_r(-\alpha p d) e_r(\alpha h) \bigg|^2.
\end{align*}
\textit{Off-Diagonal Frequencies:} If $\frac{\alpha_1}{r_1} \neq \frac{\alpha_2}{r_2} \pmod 1$, the sum over $h$ is bounded by $\ll R_0^2$. Bounding the Gauss sums and primes trivially, the off-diagonal variance satisfies $V_{\text{off}} \ll R_0^6 \frac{N^2}{X}$.

\textit{Diagonal Frequencies:} Here $r_1 = r_2 = r$ and $\alpha_1 = \alpha_2 = \alpha$ and the sum over $h$ gives a factor of $H$. Using $|G(\alpha, r)|^2 \leq 2r$ for $(\alpha,r)=1$, the diagonal is bounded by
\begin{align*}
    V_{\text{diag}} \ll H \sum_{r \sim R_0} \frac{1}{r} \sum_{\alpha (r)^\times} \bigg| \sum_{p \asymp N} w(p) e_r(-\alpha p) \bigg|^2.
\end{align*}
Applying the Large sieve (Lemma \ref{lem:LaSi})  we obtain $V_{\text{diag}} \ll \frac{H}{R_0} (N + R_0^2) \frac{N}{X} \ll \frac{H}{R_0} \frac{N^2}{X }$.

Recombining these components into the Cauchy-Schwarz inequality gives:
\begin{align*}
    |\mathcal{E}_{R_0}| \ll \|\gamma\|_\infty H \frac{N}{\sqrt{X}} \bigg( R_0^{-1/2} + \frac{R_0^3}{\sqrt{H}} \bigg).
\end{align*}
which is sufficient since $(\log X)^B<R_0 \leq X^{\eta_1}$ and $H \geq X^\eta$, by taking $B$ large in terms of $C$.
The argument is uniform with an additional interval restriction on the final prime, with the same restriction in the main-term integral.
\end{proof}

Using this lemma and the usual inductive argument that produces the Buchstab function  (eg.  \cite[Chapter 1]{harman}), we deduce the following.
\begin{lemma}  \label{lem:buchstabintegral}Let $G(\alpha_1,\dots,\alpha_k)$ be a piecewise smooth function supported on $\alpha_i > \eta$ and $\alpha_1+\cdots+\alpha_k < 1-\eta$ and $X^\eta <H \leq X^{1/2+o(1)}$. There exists $   \mathfrak{M}^{(h)}_G$ such that 
\begin{align*}
    \mathfrak{M}^{(h)}_G=\mathfrak{S}_h \frac{\sqrt{X} I_f }{\log X} \left(\int G(\boldsymbol{\alpha}) \omega (\boldsymbol{\alpha }) \frac{\d \boldsymbol{\alpha}}{\alpha_1\cdots\alpha_{k-1}\alpha_k^2}+o(1)\right),\\
    \omega(\boldsymbol{\alpha})=\omega\left(\frac{1-\sum \alpha_i}{\alpha_k}\right).
\end{align*}
and 
\[
\sum_{h \leq H} \bigg|\sum_{p_1, \dots , p_k} G\bigg(\frac{\log p_i}{\log X} \bigg) S(\B^{(h)}_{p_1\cdots p_k},p_k) -   \mathfrak{M}^{(h)}_G \bigg| \ll \frac{H \sqrt{X}}{(\log X)^C}.
\]

\end{lemma}

We also retain the boundary terms where a rough remainder equals $1$.
Lemma \ref{lem:pnt_bn} with $k=1$ and the final prime restricted to
$p_2<p_1<\sqrt X$, followed by the prime number theorem in $p_1$, gives
\begin{equation}\label{eq:semiprime_boundary}
    \sum_{h\leq H}\left|
    \sum_{p_2<p_1<\sqrt X}b^{(h)}_{p_1p_2}-\mathfrak B_h\right|
    \ll_C\frac{H\sqrt X}{(\log X)^C},
\end{equation}
where
\[
    \mathfrak B_h:=\mathfrak S_h
    \iint_{2\leq y<x<\sqrt X}
    \frac{f((xy-h)/X)}{2\sqrt{xy-h}\log x\log y}\,\d x\d y
    \ll\mathfrak S_h\frac{\sqrt X I_f}{(\log X)^2}.
\]
The integrand is understood to vanish for $xy\leq h$.
The bound follows from $x,y\asymp\sqrt X$ on its support.

\subsection{Fundamental lemmas of Harman's sieve}
We have the following three fundamental lemmas of Harman's sieve, corresponding to each of the Propositions \ref{prop:typeI}, \ref{prop:typeIfact}, and \ref{prop:typeI2}.
Throughout these lemmas assume $0<\nu<1/6-2\eta$. We will use the bound
\begin{equation}\label{eq:rough_square_error}
 \sum_{h\leq H}\sum_n \bigl(a_n^{(h)}+|b_n^{(h)}|\bigr)
 \mathbf{1}_{\exists p\geq z_0:\ p^2\mid n}
 \precprec \frac{H\sqrt X R^{3/2}}{z_0}+\frac{X}{\sqrt H}
 \qquad (2\leq z_0\leq\sqrt X).
\end{equation}
For $a_n^{(h)}$, the primes $p\leq\sqrt H$ contribute
$\ll H\sqrt X/z_0+\sqrt{XH}$. For $p>\sqrt H$, write
$r^2+h=p^2m$; then $m\ll X/H$, and each pair $(p,m)$ permits
$O(1)$ values of $r$, giving $\ll X/\sqrt H$. For the model, use
$|b_n^{(h)}|\ll R^{3/2}X^{-1/2}$ and
$\sum_{p\geq z_0}\lfloor 2X/p^2\rfloor\ll X/z_0$.
In particular, \eqref{eq:rough_square_error} with $z_0=Z$ is
$\ll_C H\sqrt X/(\log X)^C$, also with divisor-bounded weights.

\begin{lemma} \label{le:harman1}
 Let $U=X^{1/2-\nu-2\eta}$ and $ Z=X^\gamma$ as in \eqref{eq:gammadef}.   For any divisor bounded $\alpha_u$ supported on square-free numbers with $(u,P(Z))=1$ we have
 \begin{align*}
    \sum_{h \leq H} \bigg|\sum_{u \leq U} \alpha_u (S(\A^{(h)}_u,Z)-S(\B^{(h)}_u,Z))\bigg| \ll \frac{H\sqrt{X}}{(\log X)^C}.
 \end{align*}
\end{lemma}
\begin{proof}
 The sum inside is
    \begin{align*}
     \sum_{u \leq U} \alpha_u \sum_{(n,P(Z))=1} w^{(h)}_{un}  =   \sum_{u \leq U} \alpha_u \sum_{d|P(Z)} \mu(d) \sum_{n} w^{(h)}_{udn}.
    \end{align*}
We split the sum into two parts, $d \leq V$ and $d > V$.    For the part where $d \leq V$ we have $du \leq UV = X^{1/2-\eta}$ and we can use Type I information Proposition \ref{prop:typeI}. For the part where $d > V$, since $d=p_1\cdots p_k$ with $p_i < Z$, by the greedy algorithm there is some $d' = p_1 \cdots p_{k'} | d$ such that $d' \in [V,VZ]= [V,W]$. Thus, the parts where $d > V$ can be bounded as a Type II sum by Proposition \ref{prop:typeII}.   Here details such as removal of cross-conditions are standard (eg. by Perron's formula as in \cite[Chapter 3.2]{harman}).
\end{proof}

\begin{lemma} \label{le:harman1_factorable}
Let $X^\eta < H = X^{1/2-\nu} \le X^{1/2+o(1)}$,  $ V=X^{\nu+\eta}$, and $ Z=X^\gamma$ as in \eqref{eq:gammadef}. Let $U_1,U_2 > 1$. Suppose that there exists
\begin{align*}
    (D_1, D_2) \in \{ (U_1 V,U_2), \, (U_1,U_2 V) \} 
\end{align*}
such that
\begin{align*}
      D_1 \le X^{1/2-\nu-\eta},
    \qquad 
    D_2 \le X^{3/8 - 3\nu/4 - \eta},
    \qquad
    D_1^{5/2} D_2^3 \le X^{2-\eta}.  
\end{align*}
Let $\alpha_{u_1}$, $\beta_{u_2}$ be divisor-bounded coefficients supported on square-free integers with $(u_1u_2,P(Z))=1$. Then
 \begin{align*}
    \sum_{h \leq H} \bigg|\sum_{u_1 \leq U_1} \sum_{u_2 \leq U_2} \alpha_{u_1} \beta_{u_2} (S(\A^{(h)}_{u_1u_2},Z)-S(\B^{(h)}_{u_1u_2},Z))\bigg| \ll \frac{H\sqrt{X}}{(\log X)^C}.
 \end{align*}
\end{lemma}
\begin{proof}
 The terms with $(u_1,u_2)>1$ are negligible by
\eqref{eq:rough_square_error}, since a common prime is at least $Z$.
For the remaining terms the sum inside is
    \begin{align*}
     &\sum_{\substack{u_1\leq U_1,\ u_2\leq U_2\\(u_1,u_2)=1}}
       \alpha_{u_1}\beta_{u_2}\sum_{(n,P(Z))=1}w^{(h)}_{u_1u_2n}\\
     &\qquad=
       \sum_{\substack{u_1\leq U_1,\ u_2\leq U_2\\(u_1,u_2)=1}}
       \alpha_{u_1}\beta_{u_2}\sum_{d\mid P(Z)}\mu(d)
       \sum_n w^{(h)}_{u_1u_2dn}.
    \end{align*}
We split the sum into two parts, $d \leq V$ and $d > V$.    For the part where $d \leq V$, we can use the factorable Type I information Proposition \ref{prop:typeIfact} with either $(d_1,d_2) =(u_1d,u_2)$ or $(d_1,d_2)=(u_1,du_2)$; these variables are square-free and coprime.

For the part where $d > V$, since $d=p_1\cdots p_k$ with $p_i < Z$, by the greedy algorithm there is some $d' = p_1 \cdots p_{k'} | d$ such that $d' \in [V,VZ]= [V,W]$. Thus, the parts where $d > V$ can be bounded as a Type II sum by Proposition \ref{prop:typeII}.   Here details such as removal of cross-conditions are standard (eg. by Perron's formula as in \cite[Chapter 3.2]{harman}).
\end{proof}

\begin{lemma} \label{le:harman2}
 Let $Z=X^\gamma$, $K \in [X^{1/4+\eta},X^{3/4-\eta}]$, and $H=X^{1/2-\nu}$ with $0<\nu <1/6-2\eta$.  Then we have
 \begin{align*}
  \sum_{h \leq H}  \bigg|  \sum_{k \sim K}  \mathbf{1}_{(k,P(Z))=1} (S(\A^{(h)}_k,Z)-S(\B^{(h)}_k,Z)) \bigg| \ll \frac{H\sqrt{X}}{(\log X)^C}.
 \end{align*}
\end{lemma}
\begin{proof}
By \eqref{eq:rough_square_error}, we may insert $\mu(k)^2$.
Expanding the sieve on $n$ first gives
\[
 \sum_{d_2\mid P(Z)}\mu(d_2)
 \sum_{\substack{k\sim K\\(k,P(Z))=1}}\mu(k)^2
 \sum_n w^{(h)}_{kd_2n}.
\]
The part $d_2>V$ is a Type II sum by the greedy argument.
For $d_2\leq V$ we have $(k,d_2)=1$. Expanding the sieve on $k$
and writing $k=d_1\ell$ therefore gives
\[
 \sum_{\substack{d_2\mid P(Z)\\d_2\leq V}}\mu(d_2)
 \sum_{d_1\mid P(Z)}\mu(d_1)
 \sum_{\ell\sim K/d_1}\sum_n
 \mu(d_1d_2\ell)^2 w^{(h)}_{d_1d_2\ell n}.
\]
For $d=d_1d_2\leq V$, this has exactly the square-free weight
$\mu(d\ell)^2$ in Proposition \ref{prop:typeI2}.
After a smooth dyadic partition and removal of the size cross-conditions
by Perron's formula, that proposition applies since
$V<X^{1/4-\eta}$ and $d_2K\leq VK<X^{1-\eta}$.
For $d_1d_2>V$, the non-zero terms have $d_1d_2$ square-free.
By the greedy algorithm we find a divisor $q\mid d_1d_2$
in $[V,VZ]=[V,W]$ and we can apply Proposition \ref{prop:typeII}.
\end{proof}

\section{Proof of Theorem \ref{thm:primesasymp}} \label{sec:primesasymp}
Throughout this section we let $\A^{(h)}$ and $\B^{(h)}$ denote the sequences $a^{(h)}_n$ and $b^{(h)}_n$ for a fixed $h \leq H =X^{1/2-\nu}$.
    
Theorem \ref{thm:primesasymp} is an immediate corollary of the following proposition and Lemma \ref{lem:pnt_bn}.
\begin{proposition} \label{prop:asympsieve}
   Let $\nu >0$ be sufficiently small and $H =X^{1/2-\nu}$. There exists a decomposition for $\sqrt{X} < n < X$
    \begin{align*}
        \mathbf{1}_\PP(n) = \mathfrak{s}(n) + \mathfrak{e}(n)
    \end{align*}
    such that the following two hold.
 \begin{enumerate}
        \item We have
        \begin{align} \label{eq:Sasymp}
           \sum_{h \leq H}\bigg| \sum_{n} \mathfrak{s}(n) (a_n^{(h)}- b_n^{(h)}) \bigg| \ll_C \frac{H\sqrt{X}}{(\log X)^{C}}.
        \end{align}
        \item We have
        \begin{align} \label{eq:deltasieve}
 \bigg| \sum_{n} \mathfrak{e}(n) a_n^{(h)}  \bigg|  + \bigg| \sum_{n} \mathfrak{e}(n) b_n^{(h)}  \bigg|  \ll \nu \mathfrak{S}_h \frac{\sqrt{X}}{\log X}
        \end{align}
           for all but $\ll_C H/(\log X)^C$ values of $h \leq H$. 
    \end{enumerate}
\end{proposition}
To deduce Theorem \ref{thm:primesasymp}, we can write
\begin{align*}
    \sum_p  a_p^{(h)} &=   \sum_n  \mathfrak{s}(n) a_n^{(h)} + \sum_n  \mathfrak{e}(n) a_n^{(h)} \\
    &=  \sum_n  \mathfrak{s}(n)  b_n^{(h)} +  \sum_n  \mathfrak{s}(n) (a_n^{(h)} - b_n^{(h)}) + \sum_n  \mathfrak{e}(n) a_n^{(h)} \\
    & =  \sum_p    b_p^{(h)} +  \sum_n  \mathfrak{s}(n) (a_n^{(h)} - b_n^{(h)}) + \sum_n  \mathfrak{e}(n) a_n^{(h)}  -  \sum_n  \mathfrak{e}(n) b_n^{(h)}.
\end{align*}
The first term $\sum_p    b_p^{(h)} $ is then evaluated by Lemma \ref{lem:pnt_bn} and gives the expected main term. The contribution from the other three terms are bounded by Proposition \ref{prop:asympsieve} for all but  $\ll_C H/(\log X)^C$ values of $h \leq H$.
For $H=X^{1/2-o(1)}$, fix $\nu_0>0$ and choose the sieve parameters in terms of $\nu_0$. The same proof is uniform for $X^{1/2-\nu_0}\leq H\leq X^{1/2}$; then let $\nu_0\to0$.
\subsection{Construction of the sieve decomposition}
Let $\nu \in (0,1/18)$ and $\eta > 0$ be small and recall the definitions \eqref{eq:sieveparams} of the sieve parameters $U,V,W,Z$. Let us denote the indicator of $z$-rough numbers by
\begin{align*}
    \rho(n,z) = \mathbf{1}_{(n,P(z))=1}.
\end{align*}
Then we have the point-wise Buchstab's identity (inclusion-exclusion): for any $y < z$
\begin{align} \label{eq:buchstabid}
    \rho(n,z) = \rho(n,y) - \sum_{\substack{n=pm \\ y \leq p < z}}  \rho(m,p).
\end{align}
We first apply Buchstab's identity twice \eqref{eq:buchstabid} to get for any $\sqrt{X} <n  < X $
\begin{align*}
    \mathbf{1}_\P (n) = \rho(n,\sqrt{X}) = \rho_0(n)- \rho_1(n) + \rho_2(n)
\end{align*}
where
\begin{align*}
    \rho_0(n) &:= \rho(n,Z) \\
    \rho_1(n) &:=  \sum_{\substack{n=pm \\ Z \leq p < \sqrt{X}}}  \rho(m,Z) \\
    \rho_2(n) &:= \sum_{\substack{n=p_1p_2m \\ Z \leq p_2 < p_1 < \sqrt{X}}}  \rho(m,p_2).
\end{align*}

For $\rho_1$ we separate the part where $p < U$ to get
\begin{align*}
     \rho_1 =  \rho_1^{< U} +  \rho^{\geq U}_1.
\end{align*}

For $\rho_2$ recall that $Y=X/Z^2$ and partition according to the size of $p_1p_2$ into two parts
\begin{align*}
    \rho_2 = \rho_2^{\leq Y}  + \rho_2^{>Y}.
\end{align*}
We observe that
\begin{align*}
    \rho_2^{>Y}(n)=
    \sum_{\substack{n=p_1p_2m \\ Z\leq p_2<p_1<\sqrt X \\
        p_1p_2>Y}}\rho(m,p_2).
\end{align*}
Here $m<X/(p_1p_2)<Z^2$, so $m=1$ or $m$ is prime.
On the support $n\geq X/2$, the additional condition $p_1<T$
gives $p_1p_2<T^2=X^{1-8\nu}<X/2$, so $m>1$ and
$\rho(m,p_2)=\mathbf1_{m\geq p_2}\rho(m,Z)$ in this part.
We split this further according to the size of $p_1$ to obtain
\begin{align*}
    \rho_2^{>Y} = \rho_2^{> Y, < T}   +  \rho_2^{> Y, \geq T}.
\end{align*}
We can thus write
\begin{align*}
    \mathbf{1}_\PP = \mathfrak{s} + \mathfrak{e}
\end{align*}
where
\begin{align} \label{eq:mainsieve}
    \mathfrak{s} = \rho_0 - \rho_1^{< U} + \rho_2^{>Y,<T}
\end{align}
and
\begin{align} \label{eq:errorsieve}
    \mathfrak{e} =  -\rho^{\geq U}_1 + \rho_2^{\leq Y} + \rho_2^{> Y, \geq T}.
\end{align}

\subsection{Proof of the first part in Proposition \ref{prop:asympsieve}}
By triangle inequality and \eqref{eq:mainsieve}, it suffices to show that for $\rho' \in \{\rho_0,\rho_1^{< U} ,\rho_2^{>Y,<T}\}$
\begin{align*}
           \sum_{h \leq H}\bigg| \sum_{n} \rho'(n) (a_n^{(h)}- b_n^{(h)}) \bigg| \ll_C \frac{H\sqrt{X}}{(\log X)^{C}}.
        \end{align*}
For  $\rho' \in \{\rho_0,\rho_1^{< U}\}$ this follows by Lemma \ref{le:harman1}. For $\rho'=\rho_2^{>Y,<T}$ this follows by Lemma \ref{le:harman1_factorable} since we have $V p_1 < VT=X^{1/2-3\nu+\eta}\leq X^{1/2-\nu-2\eta}$ and for $\nu <1/18$ we have for some $\eta > 0$
\begin{align*}
    p_2 \leq (p_1p_2 m)^{1/3} < X^{1/3} < X^{3/8-3\nu/4-\eta}.
\end{align*}
The mixed restriction follows from $p_2\leq m$:
\[
 (Vp_1)^{5/2}p_2^3
 =V^{5/2}p_1(p_1p_2^2)^{3/2}
 \leq V^{5/2}TX^{3/2}
 =X^{2-3\nu/2+5\eta/2}\leq X^{2-\eta}.
\]

\subsection{Proof of the second part in Proposition \ref{prop:asympsieve}}
By \eqref{eq:errorsieve} and triangle inequality, it suffices to check the claim for $\rho^{\geq U}_1$, $ \rho_2^{\leq Y},$ and $\rho_2^{> Y, \geq T}$ 

We have
\begin{align*}
    \rho^{\geq U}_1(n) = \sum_{\substack{n=pm \\ U \leq p < \sqrt{X}}}  \rho(m,Z) 
\end{align*}
where recall that $U= X^{1/2-\nu-2\eta}$.
For the sequence $b_n^{(h)}$ the claim follows  by Lemma \ref{lem:buchstabintegral} since $\log p / \log X$ is restricted to an interval of length $\ll \nu$. For $a_n^{(h)}$ we use the crude upper bound
\begin{align*}
      \rho^{\geq U}_1(n)   \leq  \sum_{\substack{n=km \\ U \leq k < \sqrt{X}}} \rho(k,Z) \rho(m,Z)=: \rho_3(n).
\end{align*}
Thus, we have
\begin{align*}
0 \leq \sum_{n}  \rho^{\geq U}_1(n) a_n^{(h)}  \leq     \sum_{n} \rho_3(n) a_n^{(h)}  
=   \sum_{n} \rho_3(n) (a_n^{(h)} -b_{n}^{(h)}) +  \sum_{n} \rho_3(n) b_{n}^{(h)}.
\end{align*}
 The contribution from first sum is bounded by Lemma \ref{le:harman2}. The contribution from second sum is bounded by Lemma \ref{lem:buchstabintegral}.

 We then consider
 \begin{align*}
     \rho_2^{\leq Y} =  \sum_{\substack{n=p_1p_2m \\ Z \leq p_2 < p_1 < \sqrt{X} \\ p_1p_2 \leq Y }}  \rho(m,p_2),
 \end{align*}
 where $Z=X^{1/4-3\nu/2-2\eta}$ and $Y=X/Z^2 =X^{1/2+3\nu+4\eta} $. Then
 \begin{align*}
     p_1 \leq Y/p_2 \leq Y/Z \leq  X^{1/4+9\nu/2+6\eta}
 \end{align*}
 and thus both of $\log p_i /\log X$ are constrained to lie in $1/4+O(\nu)$.  For the sequence $b_n^{(h)}$ the claim  then follows  by Lemma \ref{lem:buchstabintegral}. For the sequence $a_n^{(h)}$ we write
 \begin{align*}
     \rho_2^{\leq Y}(n) \leq   \sum_{\substack{n=p_1p_2m \\ Z \leq p_2 < p_1 < \sqrt{X} \\ p_1p_2 \leq Y }}  \rho(m,Z) := \rho_4(n)
 \end{align*}
Then similar to the above
\begin{align*}
0 \leq \sum_{n}      \rho_2^{\leq Y} (n) a_n^{(h)}  \leq     \sum_{n} \rho_4(n) a_n^{(h)}  
=   \sum_{n} \rho_4(n) (a_n^{(h)} -b_{n}^{(h)}) +  \sum_{n} \rho_4(n) b_{n}^{(h)}.
\end{align*}
 The contribution from the  first sum is bounded by Lemma \ref{le:harman1_factorable} for sufficiently small $\nu>0$ since $p_2 \leq p_1 \leq X^{1/4+9\nu/2+6\eta}$. The contribution from second sum is bounded by Lemma \ref{lem:buchstabintegral}. 

Finally, consider 
\begin{align*}
    \rho_2^{> Y, \geq T}(n) = \sum_{\substack{n=p_1p_2m \\ Z \leq p_2 < p_1 < \sqrt{X} \\ p_1p_2 > Y \\ p_1 \geq T }}  \rho(m,p_2) 
\end{align*}
Then $X^{1/2-4\nu}=T\leq p_1<\sqrt X$ restricts $\log p_1/\log X$ to an interval of length $\ll\nu$.
For $b_n^{(h)}$, use Lemma \ref{lem:buchstabintegral} when $m>1$
and \eqref{eq:semiprime_boundary} when $m=1$.
The latter main term is $O(\mathfrak S_h\sqrt X I_f/(\log X)^2)$.
For $a_n^{(h)}$ we use a crude bound
\begin{align*}
    \rho_2^{> Y, \geq T}(n) = \sum_{\substack{n=p_1p_2m \\ Z \leq p_2 < p_1 < \sqrt{X} \\ p_1p_2 > Y \\ p_1 \geq T }}  \rho(m,p_2)  \leq  \sum_{\substack{n=km \\T \leq  k \leq \sqrt{X} }}  \rho(k,Z) \rho(m,Z) =: \rho_5(n)
\end{align*}
Then
\begin{align*}
0 \leq \sum_{n}         \rho_2^{> Y, \geq T} (n) a_n^{(h)}  \leq     \sum_{n} \rho_5(n) a_n^{(h)}  
=   \sum_{n} \rho_5(n) (a_n^{(h)} -b_{n}^{(h)}) +  \sum_{n} \rho_5(n) b_{n}^{(h)}.
\end{align*}
 The contribution from the  first sum is bounded by Lemma \ref{le:harman2}. The contribution from second sum is bounded by Lemma \ref{lem:buchstabintegral}. 

\qed
\section{Proof of Theorem \ref{thm:primes} }
\label{sec:primesthmproof}
Throughout this section we let $\A^{(h)}$ and $\B^{(h)}$ denote the sequences $a^{(h)}_n$ and $b^{(h)}_n$ for a fixed $h \leq H =X^{1/2-\nu}$. 

Theorem \ref{thm:primes} follows immediately from the following sieve minorant proposition (cf. \eqref{eq:minorantdef} for the definition of the minorant $\m$).
\begin{proposition} \label{prop:sieve}
  Let $H=X^{1/2-\nu} >X^{2/5-o(1)}$ and let $\eta>0$ be sufficiently small.   There exists a sieve minorant $\mathfrak{m} \leq \mathbf{1}_{\P}$ such that the following two hold.
    \begin{enumerate}
        \item We have
        \begin{align} \label{eq:sieveasymp}
           \sum_{h \leq H}\bigg| \sum_{n} \mathfrak{m}(n) (a_n^{(h)}- b_n^{(h)}) \bigg| \ll_C \frac{H\sqrt{X}}{(\log X)^{C}}.
        \end{align}
        \item There exist a constant $\mathfrak C_\nu>0.098$ and main terms
        \[
            \mathfrak M^{(h)}_{\mathfrak m}
            =(\mathfrak C_\nu+o(1))\mathfrak S_h
            \frac{\sqrt X I_f}{\log X},
        \]
        uniformly for $h\leq H$, such that
        \begin{align}\label{eq:deficiencyclaim}
            \sum_{h\leq H}\left|
            \sum_n\mathfrak m(n)b_n^{(h)}-
            \mathfrak M^{(h)}_{\mathfrak m}\right|
            \ll_C\frac{H\sqrt X}{(\log X)^C}.
        \end{align}
    \end{enumerate}
\end{proposition}
We will construct the sieve minorant $\mathfrak{m}$ through the usual Buchstab iterations in Harman's sieve. The minorant is constructed so that the first part of the proposition then follows by Propositions \ref{prop:typeI}, \ref{prop:typeI2}, and \ref{prop:typeII}. The second part is proved using Lemma \ref{lem:buchstabintegral} below.

\subsection{Construction of the sieve minorant $\mathfrak{m}$}

To construct the minorant $\mathfrak{m}$, we first apply Buchstab's identity twice \eqref{eq:buchstabid} to get for any $\sqrt{X} <n  < X $
\begin{align*}
    \mathbf{1}_\P (n) = \rho(n,\sqrt{X}) = \rho_0(n)- \rho_1(n) + \rho_2(n)
\end{align*}
where
\begin{align*}
    \rho_0(n) &:= \rho(n,Z) \\
    \rho_1(n) &:=  \sum_{\substack{n=pm \\ Z \leq p < \sqrt{X}}}  \rho(m,Z) \\
    \rho_2(n) &:= \sum_{\substack{n=p_1p_2m \\ Z \leq p_2 < p_1 < \sqrt{X}}}  \rho(m,p_2).
\end{align*}

For $\rho_{1}$ we first split according to whether $p < U$ to get
\begin{align*}
   \rho_1= \rho_{1}^{< U} + \rho_{1}^{\geq U}.
\end{align*}
For the second term we apply Buchstab's identity \eqref{eq:buchstabid}  to the variable $p=k$ with $\mathbf{1}_\P(k) =\rho(k,X^{1/4})$, noting that $U >X^{1/4}$ (so-called role-reversal), to obtain
\begin{align*}
-\rho_{1}^{\geq U}(n) =  -\sum_{\substack{n=pm \\ U \leq p < \sqrt{X}}}  \rho(m,Z)  =   -\rho_{1,0}^{\geq U}(n) +   \rho_{1,1}^{\geq U}(n)
\end{align*}
where
\begin{align*}
\rho_{1,0}^{\geq U}(n)  := &\sum_{\substack{n=km \\ U \leq k < \sqrt{X}}}  \rho(m,Z)  \rho(k,Z) \\
\rho_{1,1}^{\geq U}(n) := & \sum_{\substack{n= p km \\ U \leq pk < \sqrt{X} \\ Z \leq p < X^{1/4} }}  \rho(m,Z)  \rho(k,p).
\end{align*}
We split the second sum according to  whether $p$ lies in the Type II range $[V,W]$ to get
\begin{align*}
    \rho_{1,1}^{\geq U}(n) = \underbrace{\rho_{1,1}^{\geq U,<V}(n)}_{p < V} + \underbrace{\rho_{1,1}^{\geq U,[V,W]}(n)}_{p \in [V,W]} + \underbrace{\rho_{1,1}^{\geq U,> W}(n)}_{p > W}.
\end{align*}
In 
\begin{align*}
\rho_{1,1}^{\geq U,> W}(n) = \sum_{\substack{n= p km \\ U \leq pk < \sqrt{X} \\ W < p < X^{1/4} }}  \rho(m,Z)  \rho(k,p).
\end{align*}
we may apply Buchstab's identity \eqref{eq:buchstabid}  to $\rho(m,Z)$ (in the reverse direction) to extract a prime factor $Z \leq q< W$, to get
\begin{align*}
\rho_{1,1}^{\geq U,> W}(n) = \rho_{1,1}^{\geq U,> W,[Z,V)}(n) + \rho_{1,1}^{\geq U,> W,[V,W]}(n)  + \rho_{1,1,1}^{\geq U,> W}(n),
\end{align*}
where
\begin{align*} 
\rho_{1,1}^{\geq U,> W,[Z,V)}(n) :=    \sum_{\substack{n= p kqm \\ U \leq pk < \sqrt{X} \\ W <p < X^{1/4} \\ Z\leq q < V }}  \rho(m,q)  \rho(k,p),  \\
 \rho_{1,1}^{\geq U,> W,[V,W]}(n) :=    \sum_{\substack{n= p kqm \\ U \leq pk < \sqrt{X} \\ W <p < X^{1/4} \\ \max\{V,Z\}\leq q < W }}  \rho(m,q)  \rho(k,p), \\
 \rho_{1,1,1}^{\geq U,> W}(n) := \sum_{\substack{n= p km \\ U \leq pk < \sqrt{X} \\ W < p < X^{1/4} }}  \rho(m,W)  \rho(k,p).
\end{align*}

Similarly, for the term $\rho_2(n)$ we split the sum according to whether $p_2$ lies in the type II range $[V,W]$ to get
\begin{align*}
        \rho_2 = \rho_{2}^{<V} + \rho_{2}^{[V,W]} +  \rho_{2}^{> W}.
\end{align*}

Combining the above, we have for $\sqrt{X} < n < X$
\begin{align} \label{eq:sievedecomp}
    \mathbf{1}_{\P}(n) = \mathfrak{m}(n) + \mathfrak{d}(n),
\end{align}
where the minorant is defined by
\begin{align} \label{eq:minorantdef}
    \m := \rho_0 - \rho_{1}^{< U} - \rho_{1,0}^{\geq U}  +  \rho_{1,1}^{\geq U,[V,W]}+\rho_{1,1}^{\geq U,> W,[V,W]} + \rho_{2}^{[V,W]}
\end{align}
and the terms dropped by positivity are defined by
\begin{align*}
    \mathfrak{d} :=  \rho_{1,1}^{\geq U,<V} +\rho_{1,1}^{\geq U,> W,[Z,V)} + \rho_{1,1,1}^{\geq U,> W} +  \rho_{2}^{<V} + \rho_{2}^{> W}.
\end{align*}
Since $\mathfrak{d} \geq 0$, we have the minorant property $    \mathfrak{m}  \leq \mathbf{1}_{\P} .$ We also already observe that for $\nu=1/10$ the ratio $V/Z = X^{O(\eta)}$, so that the parts $\rho_{1,1}^{\geq U,<V} ,\rho_{2}^{<V},\rho_{1,1}^{\geq U,> W,[Z,V)}$ of $\mathfrak{d}$ with a variable $<V$  will contribute $O(\eta)$ to the overall deficiency.

\subsection{Proof of the first part in Proposition \ref{prop:sieve}}

\begin{proof}[Proof of \eqref{eq:sieveasymp}]
    Recall that
    \begin{align*}
     \mathfrak{m} =   \rho_0 - \rho_{1}^{< U} - \rho_{1,0}^{\geq U} +  \rho_{1,1}^{\geq U,[V,W]} + \rho_{1,1}^{\geq U,> W,[V,W]} + \rho_{2}^{[V,W]}.
    \end{align*}
    Thus, by the triangle inequality, it suffices to show that for each $\rho'$ in the above sum
    \begin{align*}
         \sum_{h \leq H}\bigg| \sum_{n} \rho'(n) (a_n^{(h)}- b_n^{(h)}) \bigg| \ll \frac{H\sqrt{X}}{(\log X)^{C}}.
    \end{align*}
    For $\rho' = \rho_0 $ and $\rho'=\rho_{1}^{< U}$ this follows by Lemma \ref{le:harman1}. For $\rho'=\rho_{1,0}^{\geq U} $ this follows by Lemma \ref{le:harman2}. Finally, for $\rho'=\rho_{1,1}^{\geq U,[V,W]}(n),$ $\rho'=\rho_{1,1}^{\geq U,> W,[V,W]},$ and $\rho'=\rho_{2}^{[V,W]}$, we have a variable in the Type II range $[V,W]$ and we can use Proposition \ref{prop:typeII}.
\end{proof}

\subsection{Proof of the second part in Proposition \ref{prop:sieve}}
\begin{proof}[Proof of \eqref{eq:deficiencyclaim}]
By the sieve decomposition \eqref{eq:sievedecomp} \begin{align*}
      \sum_{n} \mathfrak{m}(n) b_n^{(h)}  =   \sum_{n} (\mathbf{1}_{\P}(n) - \mathfrak{d}(n)) b_n^{(h)},
\end{align*}
where $$ \mathfrak{d}= \rho_{1,1}^{\geq U,<V} +\rho_{1,1}^{\geq U,> W,[Z,V)} + \rho_{1,1,1}^{\geq U,> W} +  \rho_{2}^{<V} + \rho_{2}^{> W}.$$
Let $\mathfrak M^{(h)}_{\mathfrak d}$ be the sum of the exact main terms
for these five pieces from Lemma \ref{lem:buchstabintegral}, including
$\mathfrak B_h$ from \eqref{eq:semiprime_boundary} for the $m=1$ part
of $\rho_2^{>W}$, and put
\[
    \mathfrak M^{(h)}_{\mathfrak m}
    :=M^{(h)}_0-\mathfrak M^{(h)}_{\mathfrak d},
\]
where $M^{(h)}_0$ is the $k=0$ main term in Lemma \ref{lem:pnt_bn}.
These lemmas and the triangle inequality give \eqref{eq:deficiencyclaim}.
Their leading asymptotics give, uniformly for $h\leq H$,
\[
    \mathfrak M^{(h)}_{\mathfrak m}
    =(\mathfrak C_\nu+o(1))\mathfrak S_h\frac{\sqrt X I_f}{\log X},
\]
where
\[
    \mathfrak C_\nu
    =1-(\mathfrak D_1^{<V}+\mathfrak D_1^{[Z,V)}+
    \mathfrak D_1+\mathfrak D_2^{<V}+\mathfrak D_2).
\]
The term $\mathfrak B_h$ does not affect this leading constant.
For the numerical bound we now take $\nu=1/10$; for larger $H$, the same sieve parameters remain admissible. Since $V/Z = X^{O(\eta)}$, we quickly see that $\mathfrak{D}_1^{< V}+\mathfrak{D}_1^{ [Z,V)}  + \mathfrak{D}_2^{< V} \ll \eta$ since in all three cases a prime variable is constrained in $[Z,V)$. Corresponding to $\rho_{1,1,1}^{\geq U,> W}$, since $\rho(k,p) $ for $p > W, pk < \sqrt{X}$ restricts to $k$ being a prime, we get the deficiency
\begin{align*}
    \mathfrak{D}_{1} :=  \int_{{\substack{1/4-\nu/2 \leq \alpha_1 \leq 1/2 \\ 1/2-\nu \leq \alpha_1 + \alpha_2 \leq 1/2 \\ \alpha_1 \leq \alpha_2}} } \omega\bigg( \frac{1- \alpha_1-\alpha_2}{1/4-\nu/2}\bigg) \frac{\d \alpha_1 \d \alpha_2}{\alpha_1 \alpha_2 (1/4-\nu/2)} + O(\eta).
\end{align*}
The argument in the Buchstab function is in $[2,3]$ and the integral can be written explicitly as
$$ \int_{1/5}^{1/4} \int_{\alpha_1}^{1/2-\alpha_1} \frac{1 + \log(4 - 5\alpha_1 - 5\alpha_2)}{\alpha_1 \alpha_2 (1-\alpha_1-\alpha_2)} \, \d\alpha_2 \, \d\alpha_1 + O(\eta).$$

Corresponding to $\rho_{2}^{> W}$ we get a deficiency 
\begin{align*}
 \mathfrak{D}_{2} :=   \int_{{\substack{1/4-\nu/2 \leq \alpha_2 < \alpha_1 < 1/2 \\   \alpha_1+ 2 \alpha_2 < 1 }} }  \omega\bigg( \frac{1- \alpha_1-\alpha_2}{\alpha_2}\bigg)  \frac{\d \alpha_1 \d \alpha_2}{\alpha_1 \alpha_2^2} + O(\eta).
\end{align*}
The argument in the Buchstab function is $ \leq 3$ and the integral can be written  explicitly as
\begin{align*}
    \left( \int_{1/5}^{1/4} \int_{\alpha_2}^{1/2} + \int_{1/4}^{1/3} \int_{\alpha_2}^{1-2\alpha_2} \right) \frac{1}{\alpha_1 \alpha_2 (1-\alpha_1-\alpha_2)} \d \alpha_1 \d \alpha_2
 + 
\int_{1/5}^{1/4} \int_{\alpha_2}^{1-3\alpha_2} \frac{\log\left(\frac{1-\alpha_1-2\alpha_2}{\alpha_2}\right)}{\alpha_1 \alpha_2 (1-\alpha_1-\alpha_2)} \d \alpha_1 \d \alpha_2.
\end{align*}
For $\nu=1/10$  we have
\begin{align*}
  \mathfrak{D}_{1} < 0.1319 \quad \text{and} \quad   \mathfrak{D}_{2} < 0.7697.
\end{align*}
Hence, $  \mathfrak{C}_\nu  = 1- (\mathfrak{D}_1^{< V} +\mathfrak{D}_1^{[Z,V)} +    \mathfrak{D}_1 + \mathfrak{D}_2^{< V} +\mathfrak{D}_2) > 0.0984-O(\eta) \geq 0.098$ once $\eta>0$ is chosen sufficiently small.
\end{proof}

\section{Proof of Theorem \ref{thm:mobius}}
\label{sec:möbius}
Let $|\gamma_h|\leq1$ and choose $0<\eta<\nu/100$.
In this section, we show Theorem \ref{thm:mobius}, the statement that 
\begin{equation*}
  \sum_{h < H  } \gamma_h \sum_{n^2 + h\le X }\mu(n^2 + h) 
  \ll \nu H X^{1/2}. 
\end{equation*}
This will follow, after a smooth dyadic partition of unity, from the estimate 
\begin{equation*}
    \sum_{n } \mu(n) a_n^{(\gamma)}\ll \nu H X^{1/2}.
\end{equation*}
We may suppose that $\nu < 1/100$ in what follows.

The result will follow from our Type II estimates and Ramar\'e's identity, the observation that for $y < z $, we have
\begin{equation*}
  \mathbf{1}_{\exists p\mid n:\ y\le p<z} = \sum_{\substack{p\mid n\\y\le p<z}} \frac{1}{1 + \omega_{y, z}(n/p)} + O\bigg(\sum_{\substack{p\mid n\\y\le p<z}}\mathbf{1}_{p^2\mid n}\bigg),
\end{equation*}
where we write
\begin{equation*}
  \omega_{y, z}(n) := \sum_{\substack{p\mid n\\y\le p<z}} 1.
\end{equation*}
The error term from $\mathbf{1}_{p^2\mid n}$ is negligible by
\eqref{eq:rough_square_error} with $z_0=y=X^{\nu+\eta}$.
We begin by removing the ``main term'' contribution. By
Siegel Walfisz and the large sieve (to bound the $r\le (\log X)^{10A}$ contribution and the 
$(\log X)^{10A} < r\le R$ contribution respectively, where $r$ is as in the definition of $b_n^{(\gamma)}$ 
from \S\ref{sec:setup}), we have that
\begin{equation}\label{eq:cuqcn8l8j3}
  \sum_{n }  \mu(n)b_n^{(\gamma)} \ll_A \frac{H X^{1/2}}{\log^AX}.
\end{equation}
Applying Hal\'asz to deal with the contribution of small moduli $r\ll (\log X)^{1/100}$ and 
the large sieve for larger moduli yields that 
\begin{equation*}
    \sum_{\substack{n \\ p | n\implies p\not\in [y, z)}} \mu(n) b_n^{(\gamma)} \ll \nu H X^{1/2} (\log X)^{-1/1000}.
\end{equation*}

It follows that
\begin{equation}\label{eq:cuqcn8mtgz}
  \sum_{n} \mu(n) a_n^{(\gamma)} = -\sum_{\substack{m,p\\y\le p<z\\(m,p)=1}}  \beta_m w_{mp}^{(\gamma)}
  + O_A \biggl( \nu H X^{1/2}
  + \sum_{h < H } \sum_{n^2 + h\ll X } \mathbf{1}_{ p | n^2 + h\implies p\not\in [y, z)}
  \biggr),
\end{equation}
where we take
\begin{equation*}
  \beta_m := \frac{\mu(m)}{1 + \omega_{y, z}(m)}.
\end{equation*}
With $y = X^{\nu + \eta} $ and $z = X^{1/20} $, we remove the terms
$p\mid m$ by \eqref{eq:rough_square_error} and then apply
Proposition \ref{prop:typeII} to obtain
\begin{equation}\label{eq:cuqcn8k1aj}
  \sum_{\substack{m,p\\y\le p<z\\(m,p)=1}} \beta_m w_{mp}^{(\gamma)} \precprec \frac{H X^{1/2}}{R^{1/4}}.
\end{equation}

An upper bound sieve applied to the sum over $h $ yields that
\begin{equation*} 
    \sum_{h < H} \mathbf{1}_{ p | n^2 + h\implies p\not\in [y, z) }
    \ll H\prod_{y\le p < z} \biggl(1 - \frac{1}{p}\biggr)\asymp \nu H.
\end{equation*}
Thus, 
\begin{equation}\label{eq:cuqcn8ktb9}
  \sum_{h < H } \sum_{n^2 + h\le X } \mathbf{1}_{ p | n^2 + h\implies p\not\in [y, z) }
  \ll \nu H X^{1/2}.
\end{equation}
The desired result follows upon collecting \eqref{eq:cuqcn8l8j3}, \eqref{eq:cuqcn8mtgz}, \eqref{eq:cuqcn8k1aj}, and \eqref{eq:cuqcn8ktb9}.
For $H=X^{1/2-o(1)}$, fix $\nu_0>0$, take $\eta<\nu_0/100$ and $y=X^{\nu_0+\eta}$, and run the proof uniformly for $X^{1/2-\nu_0}\leq H\leq X^{1/2}$. This gives $O(\nu_0HX^{1/2})$; now let $\nu_0\to0$.

\bibliography{quadratic}
\bibliographystyle{abbrv}

\end{document}